\documentclass[12pt,a4paper]{amsart}
\usepackage{amssymb,mathtools}
\usepackage[margin=1in]{geometry}
\usepackage{extarrows}
\usepackage{stmaryrd}
\usepackage[all]{xy}
\newdir{ >}{{}*!/-0.7em/@{>}}
\usepackage{booktabs,tabularx,array,enumitem,microtype}
\usepackage[colorlinks=true,linkcolor=blue,citecolor=blue,urlcolor=blue]{hyperref}
\hypersetup{pdftitle={When Gorenstein injective objects form a covering class},
  pdfauthor={Yuhui Hua, Shunjie Li and Chencheng Zhang},
  pdfkeywords={Gorenstein injective object, covering class, coproduct, filtered colimit, Noetherian scheme, Zariski descent, totally acyclic complex}}
\setlist{nosep,leftmargin=2.2em}
\setlist[enumerate]{label=\textup{(\arabic*)},ref=\arabic*}
\theoremstyle{plain}
\newtheorem{lemma}{Lemma}[section]
\newtheorem{definition}[lemma]{Definition}
\newtheorem{proposition}[lemma]{Proposition}
\newtheorem{example}[lemma]{Example}

\newtheorem{corollary}[lemma]{Corollary}
\theoremstyle{definition}

\newtheorem{question}[lemma]{Question}
\theoremstyle{remark}
\newtheorem{remark}[lemma]{Remark}
\theoremstyle{plain}
\newtheorem{maintheorem}{Theorem}

\newtheorem{mainexample}[maintheorem]{Example}

\newcommand{\Ext}{\operatorname{Ext}}
\newcommand{\Spec}{\operatorname{Spec}}

\title[On the existence of Gorenstein injective covers]{On the existence of Gorenstein injective covers}
\author[Y. Hua]{Yuhui Hua\textsuperscript{*}}
\thanks{\textsuperscript{*}Corresponding author.}
\address{School of Mathematical Sciences, Shanghai Jiao Tong University,
Shanghai 200240, P.~R.~China}
\email{huayuhui@sjtu.edu.cn}
\author[S.-J. Li]{Shun-Jie Li}
\address{School of Mathematical Sciences, Shanghai Jiao Tong University,
Shanghai 200240, P.~R.~China}
\email{lishunjie@sjtu.edu.cn}
\author[C. Zhang]{Chencheng Zhang}
\address{School of Mathematical Sciences, Shanghai Jiao Tong University,
Shanghai 200240, P.~R.~China}
\email{zhangchencheng@sjtu.edu.cn}
\thanks{The authors were supported by the National Natural Science
Foundation of China (No.~12131015).}
\date{}
\begin{document}
\raggedbottom
\begin{abstract}

For a locally finitely generated Grothendieck category, we characterise the Gorenstein injective covering property by several equivalent conditions, including closure under coproducts. To complete this characterisation, we construct a countable commutative Noetherian Cohen--Macaulay local ring of dimension one for which this property fails, answering negatively Iacob's question \cite[Question~1]{Iacob} whether character modules of Gorenstein injective modules over left Noetherian rings are Gorenstein flat. Using this characterisation, we prove that for Noetherian schemes the covering property is Zariski-local. Moreover, assuming that $\mathcal{GI}(\mathrm{QCoh}(X))$ is a covering class, we establish compatible adjunctions and recollements for homotopy categories of injectives and totally acyclic injective complexes under open restriction. However, this property is not stalk-local: we construct a commutative Noetherian ring $S$ such that $\mathcal{GI}(S_{\mathfrak{p}})$ forms a covering class for all prime ideals $\mathfrak{p}$, but $\mathcal{GI}(S)$ do not form a covering class.

\end{abstract}
\maketitle

\section{Introduction}

Covers and envelopes (also called minimal right and left approximations), originally studied by E. E. Enochs \cite{Enochs81} and Auslander--Smal\o~\cite{AS80} in the early 1980s, are fundamental and important concepts in relative homological algebra, representation theory and so on. Gorenstein homological algebra, introduced by Enochs--Jenda \cite{EnochsJenda1995, EnochsJenda2000}, has been developed to an advanced level. We mainly focus on Gorenstein injective objects. To obtain Gorenstein injective resolutions, it is natural to ask whether the class of Gorenstein objects forms a (pre)covering class or a (pre)enveloping class.

Let $R$ be a ring. Denote by $\mathcal{GI}(R)$ the class of Gorenstein injective modules. \v Saroch-\v S\v tov\'i\v cek \cite{SS2018} have proved that for any ring $R$, $(^\perp\mathcal{GI}(R),\mathcal{GI}(R))$ is a hereditary perfect cotorsion pair, and hence $\mathcal{GI}(R)$ is an enveloping class. For a general Grothendieck category $\mathcal{G}$, Estrada-Gillespie \cite{EG26} have proved that $(^\perp\mathcal{GI},\mathcal{GI})$ is a complete cotorsion pair (which makes $\mathcal{GI}$ an enveloping class) if and only if $\mathcal{G}$ admits a set of Tate trivial generators.

A.~Iacob \cite{Iacob, Iacob24} proved that $\mathcal{GI}(R)$ is a covering class if and only if it is special precovering if and only if it is closed under direct limits if and only if $R$ is left Noetherian and the character modules of Gorenstein injective left $R$-modules are Gorenstein flat right $R$-modules. For results related to the covering property of Gorenstein injective complexes and DG Gorenstein injective complexes, see \cite{EEI, Iacob22, Iacob24}.

In this paper, we study when the class of Gorenstein injective objects $\mathcal{GI}(\mathcal{A})$ in a locally finitely generated Grothendieck category $\mathcal{A}$ becomes a covering class. Our main theorem gives a list of equivalent characterizations for $\mathcal{GI}(\mathcal{A})$ to be a covering class. 
\begin{maintheorem}\label{thm:equivalence}
Let $\mathcal A$ be a locally finitely generated Grothendieck category with a generator $U$.
The following conditions are equivalent.
\begin{enumerate}
\item $\mathcal{GI}(\mathcal A)$ forms a covering class.
\item $\mathcal{GI}(\mathcal A)$ forms a precovering class.
\item $\mathcal{GI}(\mathcal A)$ forms a special covering class.
\item $\mathcal{GI}(\mathcal A)$ forms a special precovering class.
\item $\mathcal{GI}(\mathcal A)$ is closed under all small filtered colimits.
\item $\mathcal{GI}(\mathcal A)$ is closed under arbitrary coproducts.
\item $\mathcal{GI}(\mathcal A)$ is closed under $|\mathrm{End}_{\mathcal A}(U)|$-indexed coproducts.
\item $\mathcal A$ is locally Noetherian and the following canonical inclusion has a right adjoint
\[
 K_{\mathrm{tac}}(\mathrm{Inj}\,\mathcal A)\longrightarrow K(\mathrm{Inj}\,\mathcal A).
\]
\end{enumerate}
\end{maintheorem}

Theorem~\ref{thm:equivalence} applies, for example, to modules over a Noetherian ring, complexes in a locally Noetherian Grothendieck category, quasi-coherent sheaves on a Noetherian scheme, comodules over a coalgebra over a field.

For any ring $R$, it is known that if $\mathcal{GI}(R)$ is a covering class, then $R$ is left Noetherian \cite{CFH}. 
A. Iacob \cite[Question~1]{Iacob} asks whether, over every left Noetherian ring, the character modules of Gorenstein injective left modules are Gorenstein flat right modules. Our next result answers this question negatively, even for a countable $1$-dimensional Cohen--Macaulay local ring.

\begin{mainexample}\label{thm:main}
There exists a countable one-dimensional Cohen--Macaulay local ring $R$ such that $\mathcal{GI}(R)$ is not precovering.
More precisely, there is a sequence $(G_n)_{n\in\mathbb N_+}$ of Gorenstein injective $R$-modules such that $M=\coprod_{n\in\mathbb N_+}G_n$ has infinite Gorenstein injective dimension and admits no Gorenstein injective precover.
\end{mainexample}

For a Noetherian scheme $X$, write $P(X)$ for the property that $\mathcal{GI}(\mathrm{QCoh}(X))$ is a covering class.
For a commutative ring $R$, write $P(R)$ for the property that $\mathcal{GI}(R)$ is covering. The equivalence of abelian categories $\mathrm{QCoh}(\mathrm{Spec}\,R)\simeq \mathrm{Mod}_R$ identifies $P(\mathrm{Spec}\,R)$ with $P(R)$ when $R$ is Noetherian. Our next theorem shows that the property $P(X)$ is Zariski-local for any Noetherian scheme $X$.

\begin{maintheorem}\label{thm:locality}
Let $X$ be a Noetherian scheme. The following conditions are equivalent.
\begin{enumerate}
\item For every open subscheme $U\subseteq X$, $\mathcal{GI}(\mathcal A_U)$ forms a covering class.
\item $\mathcal{GI}(\mathcal A_X)$ forms a covering class.
\item There is an open cover $X=\bigcup_{i\in I}U_i$ such that $\mathcal{GI}(\mathcal A_{U_i})$ forms a covering class for every $i\in I$.
\item There is an affine open cover $X=\bigcup_{i\in I}\mathrm{Spec}\,R_i$ such that $\mathcal{GI}(R_i)$ forms a covering class for every $i\in I$.
\item For every open cover $X=\bigcup_{i\in I}U_i$, $\mathcal{GI}(\mathcal A_{U_i})$ forms a covering class for every $i\in I$.
\item For every affine open cover $X=\bigcup_{i\in I}\mathrm{Spec}\,R_i$, $\mathcal{GI}(R_i)$ forms a covering class for every $i\in I$.
\end{enumerate}
\end{maintheorem}

For any commutative Noetherian ring $R$ with a dualising complex, $\mathcal{GI}(R)$ is closed under filtered colimits by \cite[Theorem~6.9]{CFH}, and hence $\mathcal{GI}(R)$ is a covering class.
Our next example is a commutative Noetherian ring $S$ whose Gorenstein injective modules do not form a covering class, although its localisation $S_{\mathfrak{p}}$ at every prime ideal $\mathfrak{p}$ admits a dualising complex.

\begin{mainexample}\label{ex:stalkwise}
There exists a countable one-dimensional Cohen--Macaulay ring $S$ such that, for every $\mathfrak p\in\mathrm{Spec}\,S$, the local ring $S_{\mathfrak p}$ admits a dualising complex and hence $\mathcal{GI}(S_{\mathfrak p})$ is covering, whereas $\mathcal{GI}(S)$ is not precovering.
\end{mainexample}

For $X=\mathrm{Spec}\,S$, the identifications $\mathcal O_{X,\mathfrak p}\cong S_{\mathfrak p}$ and $\mathrm{QCoh}(X)\simeq \mathrm{Mod}_S$ give the equivalent geometric statement: every stalk ring admits a dualising complex and has a covering class of Gorenstein injective modules, while $\mathcal{GI}(\mathrm{QCoh}(X))$ does not form a precovering class. This means that, for a Noetherian scheme $X$, the property $P(X)$ is not necessarily stalk-local. 

After the preliminaries in Section~\ref{sec:preliminaries}, Section~\ref{sec:equivalence} proves Theorem~\ref{thm:equivalence}, and Section~\ref{sec:nonclosure} constructs Example~\ref{thm:main}.
Section~\ref{sec:locality} proves Theorem~\ref{thm:locality}, and Section~\ref{sec:stalkwise-descent} constructs Example~\ref{ex:stalkwise}.

\section{Preliminaries}\label{sec:preliminaries}

Throughout, all rings are associative and unital, and modules are right modules unless otherwise stated.
We write $\mathbb N=\{0,1,2,\ldots\}$ and $\mathbb N_+=\{1,2,\ldots\}$.

\subsection{Locally generated and locally Noetherian Grothendieck categories}
Let $\mathcal A$ be a \emph{Grothendieck category}, that is, an abelian category with arbitrary coproducts, exact filtered colimits and a generator.
An object $U\in\mathcal A$ is a \emph{generator} if every object of $\mathcal A$ is a quotient of a coproduct of copies of $U$.

An object $X\in\mathcal A$ is \emph{finitely generated} if $X=\bigcup_{i\in I}X_i$ for a directed family of subobjects $X_i\subseteq X$ implies $X=X_{i_o}$ for some $i_0\in I$.
Here \emph{directed} means that, for any $i,j\in I$, there is some $k\in I$ such that $X_i\subseteq X_k$ and $X_j\subseteq X_k$.
Moreover, $X$ is \emph{Noetherian} if every ascending chain of its subobjects stabilizes. An object $X \in \mathcal A$ is \emph{finitely presented} if $\mathrm{Hom}_{\mathcal A}(X,-)$ commutes with all filtered colimits.

A Grothendieck category $\mathcal A$ is \emph{locally finitely generated} if it has a generating set of finitely generated objects.
It is \emph{locally Noetherian} if it has a generating set of Noetherian objects; equivalently, every object is the directed union of its Noetherian subobjects \cite[II, \S4, p.~356]{Gabriel1962}.
In a locally Noetherian Grothendieck category, the Noetherian objects are precisely the finitely presented objects \cite[II, \S4, Corollary~1]{Gabriel1962}.

Every Grothendieck category has injective envelopes \cite[II, \S6, Theorem~2]{Gabriel1962} and an injective cogenerator \cite[\S1.9 and Theorem~1.10.1]{Grothendieck1957}.
If $\mathcal A$ is locally Noetherian, injective objects are closed under coproducts and filtered colimits \cite[II, \S4, Corollary~1]{Gabriel1962}, and both the Noetherian objects and the indecomposable injective objects have a set of representatives \cite[II, \S4, p.~356, and \S5, Proposition~11]{Gabriel1962}.

For any ring $R$, the category $\mathrm{Mod}_R$ is locally finitely generated, and is locally Noetherian if and only if $R$ is right Noetherian.
If $\mathcal A$ is a locally finitely generated Grothendieck category, then so is $\mathrm{Ch}(\mathcal A)$; moreover, $\mathrm{Ch}(\mathcal A)$ is locally Noetherian if and only if $\mathcal A$ is locally Noetherian.
The category $\mathrm{QCoh}(X)$ is locally finitely generated when, for example, $X$ is quasi-compact and quasi-separated \cite[Corollaire~(6.9.12)]{EGAI}; under these hypotheses, it is locally Noetherian if and only if $X$ is Noetherian.
Further locally Noetherian examples include torsion abelian groups and comodules over a coalgebra over a field \cite[Theorem~2.1.3 and Corollary~2.1.4]{Sweedler1969}.

\subsection{Chain Complexes}

Let $\mathcal A$ be an abelian category, and denote $\mathrm{Ch}(\mathcal A)$ by the category of complexes in $\mathcal A$.
For a complex $X$, put
\[
 Z^n(X)=\ker d_X^n,\qquad
 B^n(X)=\mathrm{im}\,d_X^{n-1},\qquad
H^n(X)=Z^n(X)/B^n(X).
\]
The \emph{shift functor} $[1]$ is given by $(X[1])^n=X^{n+1}$ and $d_{X[1]}^n=-d_X^{n+1}$.
A complex is \emph{acyclic} if $H^n(X)=0$, for all $n\in \mathbb Z$.
A chain map $f:X\longrightarrow Y$ is a \emph{quasi-isomorphism} if $H^n(f):H^n(X)\longrightarrow H^n(Y)$ are isomorphisms for all $n\in \mathbb Z$.
The \emph{derived category} $D(\mathcal A)$ is the localization of $K(\mathcal A)$ with respect to quasi-isomorphisms; $D^b(\mathcal A)$ is its full subcategory of objects with bounded cohomology.
For a chain map $f:X\to Y$, its \emph{mapping cone} is given by
\[
 \mathrm{Cone}(f)^n=X^{n+1}\oplus Y^n,\qquad
 d_{\mathrm{Cone}(f)}^n=\left(\begin{smallmatrix}
     -d_X^{n+1} & 0 \\
     f^{n+1} & d_Y^n
 \end{smallmatrix}\right).
\]
A complex $X$ is \emph{contractible} if $\mathrm{Id}_X$ is homotopic to zero.
For complexes $E$ and $X$, the \emph{Hom complex} is given by
\[
 \mathrm{Hom}_{\mathcal A}^\bullet(E,X)^n
 =\prod_{p\in\mathbb Z}\mathrm{Hom}_{\mathcal A}(E^p,X^{p+n}),\qquad
 (df)^p=d_X^{p+n}f^p-(-1)^n f^{p+1}d_E^p.
\]

A complex $I$ of injective objects in $\mathcal A$ is \emph{totally acyclic} if it is acyclic and $\mathrm{Hom}_{\mathcal A}(E,I)$ is acyclic for every injective object $E$. Write $K(\mathrm{Inj}\,\mathcal A)$ for the full subcategory of $K(\mathcal A)$ formed by complexes of injective objects, and $K_{\mathrm{ac}}(\mathrm{Inj}\,\mathcal A)$ and $K_{\mathrm{tac}}(\mathrm{Inj}\,\mathcal A)$ for its acyclic and totally acyclic subcategories.
For $\mathcal A=\mathrm{Mod}_R$, we write $\mathrm{Ch}(R)$, $K(R)$ and $D(R)$.
For locally Noetherian $\mathcal A$, coproducts in $K(\mathrm{Inj}\,\mathcal A)$ are computed degreewise.

\subsection{Gorenstein injective objects}
The notion of Gorenstein injective modules is due to Enochs and Jenda \cite{EnochsJenda1995}; see also \cite[Definition~10.1.1]{EnochsJenda2000}.
Let $\mathcal{A}$ be an abelian category with enough injective objects, e.g., a Grothendieck category.
An object $G$ is \emph{Gorenstein injective} if $G\cong Z^0(I)$ for some totally acyclic injective complex $I$, which we call a \emph{complete injective resolution} of $G$.
We denote this class by $\mathcal{GI}(\mathcal A)$, or by $\mathcal{GI}(R)$ in $\mathrm{Mod}_R$.
A complete injective resolution gives
\begin{equation}\label{eq:gi-ext-vanishing}
 \mathrm{Ext}_{\mathcal A}^j(E,G)=0
 \qquad \text{for\ any\ } j\in\mathbb N_+,\ E\text{ injective},\ G\in\mathcal{GI}(\mathcal A).
\end{equation}
Consequently, an acyclic complex of injective objects whose cycles are Gorenstein injective is totally acyclic.
The class of Gorenstein injective objects is closed under extensions and direct summands \cite[Corollary~4.5 and Proposition~4.11]{SSW2008}.
The \emph{Gorenstein injective dimension} of $M$ is the least $d\in\mathbb N$ such that there is an exact sequence $0\to M\to G^0\to\cdots\to G^d\to0$ with all $G^i$ Gorenstein injective, and is $+\infty$ if no such $d$ exists.

\subsection{(Pre)covering classes}
Let $\mathcal F$ be a class of objects in an abelian category $\mathcal A$.
An $\mathcal F$-\emph{precover} of $M$ is a morphism $p:F\to M$, with $F\in\mathcal F$ such that every morphism $F'\to M$ with $F'\in\mathcal F$ factors through $p$.
It is an $\mathcal F$-\emph{cover} if every endomorphism $a:F\to F$ satisfying $pa=p$ is an isomorphism.
We call $\mathcal F$ a \emph{precovering class} or a \emph{covering class} if every object has the corresponding approximation.
Note that a \emph{right approximation} means a precover in the same sense.

An $\mathcal F$-precover or $\mathcal F$-cover $p:F\to M$ is called \emph{special} if $\mathrm{Ext}_{\mathcal A}^1(F',\ker p)=0$ for every $F'\in\mathcal F$.
We call $\mathcal F$ a \emph{special precovering class} or a \emph{special covering class} if every object admits a special $\mathcal F$-precover or a special $\mathcal F$-cover, respectively.
Our convention does not require precovers or covers to be epimorphisms.

\subsection{Compactly generated triangulated categories}

Let $\mathcal T$ be a triangulated category with arbitrary coproducts.
A triangulated subcategory is \emph{localising} if it is closed under arbitrary coproducts.
An object $C\in\mathcal T$ is \emph{compact} if $\mathrm{Hom}_{\mathcal T}(C,-)$ preserves arbitrary coproducts. Denote the subcategory of compact objects by $\mathcal T^c$ .
We say a set $\mathcal G$ of compact objects \emph{generates} $\mathcal T$ if
\[
 \mathrm{Hom}_{\mathcal T}(G,X[n])=0,
 \quad\forall \ G\in\mathcal G,\ n\in\mathbb Z
 \quad\Longrightarrow\quad X\simeq0.
\]
We call $\mathcal T$ \emph{compactly generated} if such a set $\mathcal{G}$ exists.
The \emph{thick closure} of a set $S$ of objects is the smallest triangulated subcategory containing $S$ that is closed under direct summands, denoted by $\mathrm{thick}(S)$.

\subsection{Model category}

A \emph{model structure} on a category is a triple $(\mathrm{CoFib},\mathrm{Fib},\mathrm{Weq})$ of classes of morphisms,
in which the morphisms are called \emph{cofibrations}, \emph{fibrations}, \emph{weak equivalences}, respectively, satisfying $(\rm CM1)$-$(\rm CM4)$(\cite{Q1}).

A \emph{model category} is a category $\mathcal M$ together with a model structure $(\mathrm{CoFib},\mathrm{Fib},\mathrm{Weq})$ on $\mathcal M$ such that $\mathcal M$ is complete and cocomplete.

Let $\mathcal M$ be a model category and $X\in \mathcal M$. $X$ is called \emph{cofibrant} (resp. \emph{fibrant}) if the unique morphism $\varnothing \to X$ (resp. $X\to *$) from the initial (terminal) object is a cofibration (resp. fibration). A \emph{cofibrant replacement} of $X$ is a cofibrant object $QX$ together with a trivial fibration $q_X:QX\longrightarrow X$. Equivalently, it is
obtained by factoring the morphism $\varnothing \to X$ as
\[
\varnothing \xrightarrow{\ i\ } QX \xrightarrow{\ q_X\ } X,
\]
where $i$ is a cofibration and $q_X$ is a trivial fibration. Dually, a \emph{fibrant replacement} of $X$ is a fibrant object $RX$ together with a trivial cofibration $r_X:X\longrightarrow RX$. Equivalently,
it is obtained by factoring the morphism $X\to *$ as
\[
X \xrightarrow{\ r_X\ } RX \xrightarrow{\ p\ } *,
\]
where $r_X$ is a trivial cofibration and $p$ is a fibration.

Now, let $\mathcal{N}$ be another model category. A \emph{Quillen adjunction}
(or \emph{Quillen pair}) between $\mathcal{M}$ and $\mathcal{N}$ is an adjunction
\[
F:\mathcal M \longrightarrow \mathcal N,
\qquad
G:\mathcal N \longrightarrow \mathcal M,
\]
with $F\dashv G$, such that the following equivalent conditions hold:
\begin{enumerate}
\item $F$ preserves cofibrations and trivial cofibrations;
\item $G$ preserves fibrations and trivial fibrations.
\end{enumerate}
When this holds, $F$ is called a \emph{left Quillen functor} and $G$ is
called a \emph{right Quillen functor}.

For a Quillen adjunction $F\dashv G$, the left and right derived functors are defined on objects by
\[
\mathbb L F(X)=F(QX),
\qquad
\mathbb R G(Y)=G(RY),
\]
where $Q$ and $R$ denote cofibrant and fibrant replacements, respectively.

\section{Characterisations of the covering property of \texorpdfstring{$\mathcal{GI}$}{GI}}\label{sec:equivalence}

Let $\mathcal A$ be a locally finitely generated Grothendieck category with a generator $U$.
The following conditions are equivalent.
\begin{enumerate}
\item $\mathcal{GI}(\mathcal A)$ forms a covering class.
\item $\mathcal{GI}(\mathcal A)$ forms a precovering class.
\item $\mathcal{GI}(\mathcal A)$ forms a special covering class.
\item $\mathcal{GI}(\mathcal A)$ forms a special precovering class.
\item $\mathcal{GI}(\mathcal A)$ is closed under all small filtered colimits.
\item $\mathcal{GI}(\mathcal A)$ is closed under arbitrary coproducts.
\item $\mathcal{GI}(\mathcal A)$ is closed under $|\mathrm{End}_{\mathcal A}(U)|$-indexed coproducts.
\item $\mathcal A$ is locally Noetherian and the following canonical inclusion has a right adjoint
\[
 K_{\mathrm{tac}}(\mathrm{Inj}\,\mathcal A)\longrightarrow K(\mathrm{Inj}\,\mathcal A).
\]
\end{enumerate}

Subsection~\ref{subsec:finite-endomorphisms} treats the case in which $\mathrm{End}_{\mathcal A}(U)$ is finite, when $\mathcal{GI}(\mathcal A)$ forms a covering class.
Subsection~\ref{subsec:noetherian-reduction} shows that each of conditions~\textup{(1)--(7)} forces $\mathcal A$ to be locally Noetherian.
We then work in the locally Noetherian case and prove the implications below, with the three main steps marked by $\square$, $\star$ and $\triangle$.
\[
\xymatrix@C=2cm@R=.5cm{
(4) \ar@{->}[r] & (2) \ar@{->}[r] & (6) \ar@{->}[d]_{\square} \ar@<-0.6ex>@{->}[r] & (7) \ar@<-0.6ex>@{->}[l]_{\star} \\
(3) \ar@{->}[u] & (1) \ar@{->}[l] & (5) \ar@{->}[l]_{\triangle} \ar@{->}[r] & (8) \ar@{->}[lu]
}
\]
The implication $\square$ uses the functorial complete injective resolutions constructed in Subsection~\ref{subsec:functorial-complete-resolutions} and the closure of localising subcategories of $K(\mathrm{Inj}\,\mathcal A)$ under degreewise filtered colimits established in Subsection~\ref{subsec:colimit-comparison}.
A cardinality argument in Subsection~\ref{subsec:cardinal-criterion} gives the implication $\star$.
For $\triangle$, we first show that the class $\mathcal C$ of totally acyclic complexes of injective objects is Kaplansky.
Under condition~\textup{(5)}, we then show that $\mathcal C$ is covering in $\mathrm{Ch}(\mathcal A)$ and deduce that $\mathcal{GI}(\mathcal A)$ is covering in $\mathcal A$; see Subsection~\ref{subsec:small-tac}.
The proof of Theorem~A is completed in Subsection~\ref{sec:ProofA}.
Finally, Subsection~\ref{subsec:tac-recollement} gives criteria for the inclusion $K_{\mathrm{tac}}(\mathrm{Inj}\,\mathcal A)\to K(\mathrm{Inj}\,\mathcal A)$ to admit a left or a right adjoint.

\subsection{Generators with finite endomorphism rings}\label{subsec:finite-endomorphisms}

\begin{lemma}\label{lem:finite-endomorphism-generator}
Let $\mathcal A$ be a Grothendieck category with a generator $U$ whose endomorphism ring is finite.
Then $\mathcal A$ is equivalent to the category of modules over a finite ring.
In particular, $\mathcal A$ is locally Noetherian and $\mathcal{GI}(\mathcal A)$ forms a covering class.
\end{lemma}

\begin{proof}
Put $B=\mathrm{End}_{\mathcal A}(U)$.
Since $\mathcal{A}$ is a Grothendieck category with a generator $U$, one has a fully faithful functor $H=\mathrm{Hom}_{\mathcal A}(U,-):\mathcal A\to\mathrm{Mod}_B$ with an exact left adjoint $Q$ \cite[Theorem~2.5.2]{Krause2022}.
Moreover, By \cite[Proposition~2.2.11 and Lemma~2.2.10]{Krause2022},
there is a localisation sequence of abelian categories
\[
\xymatrix@C=2cm{
\ker Q
  \ar@<0.7ex>@{..>>}[r]+L;[]+R
& \mathrm{Mod}_B
  \ar@<0.7ex>@{>->}[l]+R+<0.5em,0pt>;[]+L
  \ar@<0.7ex>@{>->}[r]+L+<-0.5em,0pt>;[]+R^{H}
& \mathcal A
  \ar@<0.7ex>@{->>}[l]+R;[]+L^{Q}
}
\]
Since $B$ is artinian, the Serre subcategory $\ker Q$ corresponds to an idempotent $e\in B$ such that $\ker Q$ is equivalent to $\mathrm{Mod}_{B/BeB}$ \cite[Example~2.2.19(2)]{Krause2022}.
In this case $\mathcal{A}$ is equivalent to $\mathrm{Mod}_{eBe}$, which is a module category over a finite ring.
By \cite[Corollary~4.13]{Yu2016}, Gorenstein injective modules form a covering class over every Artin algebra, so $\mathcal{GI}(\mathcal A)$ forms a covering class.
This completes the proof.
\end{proof}

\subsection{Reduction to the locally Noetherian case}\label{subsec:noetherian-reduction}
Each of conditions~\textup{(1)--(7)} forces $\mathcal A$ to be locally Noetherian.

\begin{lemma}\label{lem:cardinal-coproducts}
Let $\mathcal A$ be a locally finitely generated Grothendieck category. If every countable coproduct $E=\coprod_{n\in\mathbb N}E_n$ of injective objects $(E_n)_{n\in\mathbb N}$ is Gorenstein injective, then $\mathcal A$ is locally Noetherian.
\end{lemma}

\begin{proof}
  Take $E_n$'s and $E$ as in the statement.
By assumption, $E$ is Gorenstein injective.
Hence a complete injective resolution yields an injective precover $p : I\to E$.
Thus for each structure morphism $\iota_n : E_n \hookrightarrow E$, there is a lift $f_n : E_n \to I$ such that $p f_n = \iota_n$.
It follows that $E$ is a direct summand of $I$, and hence injective.
Since $\mathcal A$ is locally finitely generated, it follows from \cite[Theorem~11.2.12]{Krause2022} that $\mathcal A$ is locally Noetherian.
\end{proof}

\begin{lemma}\label{lem:precovering-noetherian}
Let $\mathcal A$ be a locally finitely generated Grothendieck category.
If $\mathcal{GI}(\mathcal A)$ forms a precovering class, then $\mathcal A$ is locally Noetherian.
\end{lemma}
\begin{proof}
For a countable family of injective objects $(E_n)_{n\in\mathbb N}$, take a Gorenstein injective precover $G\to E=\coprod_{n\in\mathbb N}E_n$ and an injective precover $I\to G$ from a complete injective resolution.
Let $p:I\to E$ denote such a composition.
For each structure morphism $\iota_n:E_n\to E$, there is a lift $f_n:E_n\to I$ such that $p f_n=\iota_n$.
It follows that $E$ is a direct summand of $I$, and hence injective.
Lemma~\ref{lem:cardinal-coproducts} shows that $\mathcal{A}$ is locally Noetherian.
\end{proof}

\begin{proposition}\label{prop:noetherian-reduction}
If any of conditions~\textup{(1)--(7)} of Theorem~\ref{thm:equivalence} holds, then $\mathcal A$ is locally Noetherian.
\end{proposition}
\begin{proof}
If $\mathrm{End}_{\mathcal A}(U)$ is finite, apply Lemma~\ref{lem:finite-endomorphism-generator}.
Suppose now that this ring is infinite.
Conditions~\textup{(1)--(4)} imply that $\mathcal{GI}(\mathcal A)$ forms a precovering class, so Lemma~\ref{lem:precovering-noetherian} applies.
Conditions~\textup{(5)--(7)} imply that countable coproducts of injective objects are in $\mathcal{GI}(\mathcal A)$, so Lemma~\ref{lem:cardinal-coproducts} applies.
This completes the proof.
\end{proof}

\subsection{Functorial complete injective resolution for \texorpdfstring{$\mathcal{GI}(\mathcal{A})$}{GIA}}\label{subsec:functorial-complete-resolutions}

Let $\mathcal A$ be a locally Noetherian Grothendieck category with generator $U$.
Fix a set $\mathcal E$ of representatives of the isomorphism
classes of indecomposable injective objects in $\mathcal A$.
We construct functorial complete injective resolutions of
Gorenstein injective objects.

\begin{lemma}\label{lem:gi-precover}
  For each $X \in \mathcal{A}$, denote the evaluation
  \[
  \epsilon_X : P(X) = \coprod_{E \in \mathcal{E}} \coprod_{f \in \mathrm{Hom}_\mathcal{A}(E,X)} E \to X
  \]
  Then for each $G \in \mathcal{GI}(\mathcal{A})$, the short exact sequence $0 \to \ker \epsilon_G \to P(G) \xrightarrow{\epsilon _G} G \to 0$ remains exact after applying $\mathrm{Hom}_\mathcal{A}(E,-)$ for each $E \in \mathcal{E}$, and $\ker \epsilon_G \in \mathcal{GI}(\mathcal{A})$.
\end{lemma}

\begin{proof}
  Since $\mathcal{A}$ is locally Noetherian, $P(X)$ is a well-defined injective object for each $X \in \mathcal{A}$.
  By decomposition of injective objects, $\epsilon_X$ is an injective precover.
  For each $G \in \mathcal{GI}(\mathcal{A})$, choose a short exact sequence $0 \to G' \to E' \to G \to 0$ from a complete injective resolution of $G$.
  The map $E'\to G$ is an injective precover by \eqref{eq:gi-ext-vanishing}, which factors through $\epsilon_G$.
  Hence $\epsilon_G$ is epic.
  Then Schanuel lemma \cite[Lemma~2.2]{EJO2001} gives $G' \oplus P(G) \cong E' \oplus \ker \epsilon_G$.
  Thus $\ker \epsilon_G \in \mathcal{GI}(\mathcal{A})$.
  The short exact sequence $0 \to \ker \epsilon_G \to P(G) \xrightarrow{\epsilon _G} G \to 0$ remains exact after applying $\mathrm{Hom}_\mathcal{A}(E,-)$ for each $E \in \mathcal{E}$.
\end{proof}

Then fix a injective cogenerator $J$.

\begin{lemma}\label{lem:gi-preenvelope}
  For each $X \in \mathcal{A}$, denote the embedding $\iota_X : X \to \prod_{\mathrm{Hom}_\mathcal{A}(X,J)}J$.
  Then for each $G \in \mathcal{GI}(\mathcal{A})$, the short exact sequence $0 \to G \xrightarrow{\iota_G} \prod_{\mathrm{Hom}_\mathcal{A}(G,J)}J \to \mathrm{coker} \iota_G \to 0$ remains exact after applying $\mathrm{Hom}_\mathcal{A}(E,-)$ for each $E \in \mathcal{E}$, and $\mathrm{coker} \iota_G \in \mathcal{GI}(\mathcal{A})$.
\end{lemma}

\begin{proof}
  For each $X \in \mathcal{A}$, since $J$ is a cogenerator and $\mathrm{Hom}_\mathcal{A}(\iota_X, J)$ is surjective, $\iota_X$ is monic.
  By the dual procedure in the proof of Lemma~\ref{lem:gi-precover}, $\mathrm{coker} \iota_G$ is a summand of a Gorenstein injective object, and hence Gorenstein injective.
  Since $G \in \mathcal{GI}(\mathcal{A})$, the short exact sequence $0 \to G \xrightarrow{\iota_G} \prod_{\mathrm{Hom}_\mathcal{A}(G,J)}J \to \mathrm{coker} \iota_G \to 0$ remains exact after applying $\mathrm{Hom}_\mathcal{A}(E,-)$ for each $E \in \mathcal{E}$.
\end{proof}

\begin{proposition}\label{lem:gi-basic}
There is a functor $\mathbf I:\mathcal{GI}(\mathcal A)\to\mathrm{Ch}(\mathcal A)$ whose values are totally acyclic complexes of injective objects, with a natural isomorphism $Z^0\mathbf I\cong\mathrm{Id}_{\mathcal{GI}(\mathcal A)}$.
\end{proposition}

\begin{proof}
  For each $G \in \mathcal{GI}(\mathcal{A})$, construct a complete injective resolution $\mathbf I(G)$ by iterating the constructions in Lemmas~\ref{lem:gi-precover} and \ref{lem:gi-preenvelope} on the kernel and cokernel sides, respectively. 
\end{proof}

\begin{remark}
  Note that this functor is not required to be additive in the further study.
\end{remark}

\subsection{A comparison of two colimits in \texorpdfstring{$K(\mathrm{Inj}\mathcal{A})$}{KInjA}}\label{subsec:colimit-comparison}

This subsection extends \cite[Theorem~13.3]{BensonKrause} from finite group algebra modules to locally Noetherian Grothendieck categories.
Benson and Krause construct homotopy colimits using a bar resolution in the diagram category \cite[\S13]{BensonKrause}.
We simplify the proof by computing homotopy colimits as derived colimits through the projective model structures on dg modules and their diagram categories.
We show that over a locally Noetherian Grothendieck category, for any filtered diagram $F:D\to \mathrm{Ch}(\mathrm{Inj}\,\mathcal A)$, the homotopy colimit of $F$ coincides with the colimit of $F$ in $\mathrm{Ch}(\mathcal{A})$.

Following \cite[\S3.1]{Keller2006}, let $\mathrm{Ch}_{dg}(\mathcal{A})$ denote the \emph{dg category} with objects $\mathrm{Ch}(\mathcal{A})$ and morphisms given by the Hom complexes.
A \emph{right dg module} over a dg category $\mathcal{C}$ is a functor $F: \mathcal{C}^{\mathrm{op}}\to \mathrm{Ch}_{dg}(\mathrm{Ab})$ such that $\mathrm{Hom}_{\mathcal{C}}^\bullet (Y,X) \to \mathrm{Hom}_{\mathrm{Ch}(\mathrm{Ab})}^\bullet(FX, FY)$ are chain maps.

Choose a set $\mathcal N$ of representatives of the Noetherian objects of $\mathcal A$, and fix for each $N \in \mathcal{N}$ an injective resolution $N\to iN$.
Let $\mathcal C \subseteq \mathrm{Ch}_{dg}(\mathcal{A})$ be the full dg subcategory of complexes with objects $iN$.
The derived category $D(\mathcal C)$ is the localisation of the category of right dg $\mathcal C$-modules, by inverting those module morphisms $f:F\to G$ for which $f(iN)$ is a quasi-isomorphism for every $N\in\mathcal N$.

Let $\mathcal M$ be the model category of right dg $\mathcal C$-modules with its projective model structure \cite[Theorem~3.2]{Keller2006}, which is determined by the following classes of morphisms:
\[
\begin{aligned}
 \mathrm{Weq}_{\mathcal M}
 &=\{f\mid f(c)\text{ is a quasi-isomorphism for every }c\in\mathrm{Ob}\,\mathcal C\},\\
 \mathrm{Fib}_{\mathcal M}
 &=\{f\mid f(c)^n\text{ is surjective for every }c\in\mathrm{Ob}\,\mathcal C,
                 \ n\in\mathbb Z\},\\
 \mathrm{Cof}_{\mathcal M}
 &={}^{\square}(\mathrm{Fib}_{\mathcal M}\cap\mathrm{Weq}_{\mathcal M}).
\end{aligned}
\]
Here ${}^{\square}\mathcal P$ denotes the maps with the left lifting property against every map in $\mathcal P$.
Since $\mathcal C$ is small, $\mathcal{M}$ is cofibrantly generated \cite[\S3, following Definition~3.1]{Toen2007}.
For a small category $D$, the diagram category $\mathcal M^D=\mathrm{Fun}(D,\mathcal M)$ again admits a cofibrantly generated projective model structure, where $\mathrm{Fib}_{\mathcal M^{D}}$ and $\mathrm{Weq}_{\mathcal M^{D}}$ are checked pointwise \cite[Theorem~11.6.1]{Hirschhorn2003}.

For any $D$-indexed diagram $F:D\to\mathcal M$, let $q:QF\twoheadrightarrow F$ denotes its cofibrant replacement in $\mathcal M^D$.
Here $q \in \mathrm{Weq}_{\mathcal M^{D}} \cap \mathrm{Fib}_{\mathcal M^{D}}$, and hence $q_d$ is a surjective quasi-isomorphism for every $d\in D$.
The homotopy colimit of $F$ in $\mathcal{M}$ is defined in the derived category $D(\mathcal C) = \mathcal M [\mathrm{Weq}_{\mathcal M}{}^{-1}] = \mathsf{Ho}(\mathcal{M})$:
\[
 \mathrm{hocolim}_D F=(\mathbf L\mathrm{colim}_D)(F)
 \cong\mathrm{colim}_D(QF)
 \qquad\text{in }D(\mathcal C).
\]
Indeed, the ordinary adjunction $\mathrm{colim}_D\dashv\Delta$ is a Quillen adjunction, since the constant diagram functor $\Delta$ preserves fibrations and weak equivalences \cite[Proposition~8.5.3]{Hirschhorn2003}.
It therefore induces the derived adjunction
\[
 \mathrm{Hom}_{D(\mathcal C)}(\mathrm{hocolim}_D F,Y)
 \cong\mathrm{Hom}_{\mathrm{Ho}(\mathcal M^D)}(F,\Delta Y).
\]

The dg Yoneda functor $\Phi(X)(iN)=\mathrm{Hom}_{\mathcal A}^{\bullet}(iN,X)$ induces an triangle equivalence
\[
 \Phi:K(\mathrm{Inj}\,\mathcal A)\xlongrightarrow{\sim}D(\mathcal C)
 \qquad\text{\cite[Proposition~A.1]{Krause2005}}.
\]
Fix a quasi-inverse $\Phi^{-1}$ of the triangle equivalence $\Phi$.
For a diagram $T:D\to\mathrm{Ch}(\mathrm{Inj}\,\mathcal A)$, define its homotopy colimit by
\[
 \mathrm{hocolim}_{d\in D}T_d
 :=\Phi^{-1}\!\left(\mathrm{hocolim}_{d\in D}\Phi(T_d)\right).
\]
Here $\Phi^{-1}$ is an arbitrary fixed quasi-inverse.

\begin{proposition}\label{lem:localizing-colimit}
Let $\mathcal A$ be a locally Noetherian Grothendieck category and let $T:D\to\mathrm{Ch}(\mathrm{Inj}\,\mathcal A)$ be a filtered diagram of complexes of injective objects and chain maps.
Then its colimit in $\mathrm{Ch}(\mathcal A)$ is a complex of injective objects representing the homotopy colimit.
\end{proposition}

\begin{proof}
Put $X=\mathrm{colim}_{d\in D}T_d$ the colimit in $\mathrm{Ch}(\mathcal A)$.
Since $\mathcal{A}$ is locally Noetherian, the injective objects are closed under filtered colimits \cite[Theorem~11.2.12]{Krause2022}.
Thus $X \in \mathrm{Ch}(\mathrm{Inj}\mathcal{A})$.
The canonical maps $T_d\to X$ induce a compatible family of morphisms of dg module $\Phi (T_d)\to\Phi(X)$, and hence a comparison map $\alpha:\mathrm{colim}_{d\in D}\Phi (T_d)\to\Phi(X)$.

For $N\in\mathcal N$ and a complex $Y$ of injective objects, the injective resolution $N\to iN$ yields a quasi-isomorphism $r_{N,Y}:\mathrm{Hom}_{\mathcal A}^{\bullet}(iN,Y)\to\mathrm{Hom}_{\mathcal A}^{\bullet}(N,Y)$ is a quasi-isomorphism: its $n$th cohomology map is
$\mathrm{Hom}_{K(\mathcal A)}(iN,Y[n])\to\mathrm{Hom}_{K(\mathcal A)}(N,Y[n])$, which is an isomorphism by \cite[Lemma~2.1]{Krause2005} for every $n\in\mathbb Z$.
Naturality in $Y$ gives the commutative diagram of complexes of abelian groups
\[
\xymatrix@C=2.5em@R=2em{
 \mathrm{colim}_{d\in D}\mathrm{Hom}_{\mathcal A}^{\bullet}(iN,T_d)
 \ar[r]^-{\alpha(iN)}\ar[d]_-{\mathrm{colim}_{d\in D}(r_{N,T_d})}
 &\mathrm{Hom}_{\mathcal A}^{\bullet}(iN,X)\ar[d]^{r_{N,X}}\\
 \mathrm{colim}_{d\in D}\mathrm{Hom}_{\mathcal A}^{\bullet}(N,T_d)
 \ar[r]_-{\beta_N}^-{\sim}
 &\mathrm{Hom}_{\mathcal A}^{\bullet}(N,X).
}
\]
Since $N$ is finitely presented, $\beta_N$ is an isomorphism.
Recall that $r_{N,X}$ and each $r_{N, T_d}$ are quasi-isomorphisms.
By exactness of filtered colimits, $\mathrm{colim}_{d\in D}(r_{N,T_d})$ is a quasi-isomorphism.
It follows that $\alpha(iN)$ is a quasi-isomorphism for every $N\in\mathcal N$, hence $\alpha$ is a quasi-isomorphism of dg modules.

Now choose a projective cofibrant diagram replacement $q:Q \Phi (T)\twoheadrightarrow \Phi (T)$.
Each component $q_d$ is a quasi-isomorphism of dg modules.
Since filtered colimits of dg modules are computed objectwise and preserve quasi-isomorphisms, $\mathrm{colim}_{d\in D}q_d$ is also a quasi-isomorphism.
Since $\mathrm{colim}_{d\in D}(Q\Phi (T))_d$ represents $\mathrm{hocolim}_D \Phi (T)$, we obtain
\[
 \mathrm{hocolim}_D \Phi (T)
 \cong\mathrm{colim}_{d\in D}(Q\Phi (T))_d
 \xlongrightarrow{\sim}\mathrm{colim}_{d\in D}\Phi(T_d)
 \xlongrightarrow[\sim]{\alpha}\Phi(X)
 \qquad\text{in }D(\mathcal C).
\]
Applying the chosen quasi-inverse $\Phi^{-1}$ completes the proof.
\end{proof}

Recall that a model category is \emph{stable} if it is pointed and suspension induces an equivalence of its homotopy category.
It is \emph{combinatorial} if it is cofibrantly generated and its underlying category is locally presentable.

\begin{lemma}\label{lem:localising-hocolimits}
Let $\mathcal M$ be a stable combinatorial model category and let $\mathcal S$ be a localising subcategory of $\mathrm{Ho}(\mathcal M)$.
For any small diagram $F:D\to\mathcal M$ whose values represent objects of $\mathcal S$, one has $\mathrm{hocolim}_D F\in\mathcal S$.
\end{lemma}

\begin{proof}
Since $\mathcal M$ is combinatorial, it is cofibrantly generated. 
Thus the diagram category $\mathcal M^D$ admits the projective model structure, i.e., $\mathrm{Weq}$ and $\mathrm{Fib}$ are degreewise defined and $\mathrm{Cof} = {}^{\square}(\mathrm{Fib}\cap\mathrm{Weq})$ \cite[Theorem~11.6.1]{Hirschhorn2003}.

Suppose first that $\mathcal M$ is simplicial.
Take a projectively cofibrant diagram replacement $q:QF\to F$ in $\mathcal M^D$.
Each $(QF)_d$ is cofibrant, and $q_d$ is a weak equivalence, so $(QF)_d$ still represents an object of $\mathcal S$.
Since $\mathcal S$ is closed under coproducts, cones and extensions, \cite[Proposition~3.4]{CGR2014} applies and gives $\mathrm{hocolim}_D F\cong\mathrm{hocolim}_D QF\in\mathcal S$.

In general, by \cite[Proposition~3.1]{CGR2014}, $\mathcal{M}$ is Quillen equivalent to some stable simplicial combinatorial model category $\mathcal N$, i.e., $L:\mathcal M\rightleftarrows\mathcal N:R$ is a Quillen equivalence with $L$ left adjoint.
Hence one has the triangle equivalence $\Psi = \mathbf L L:\mathsf{Ho}(\mathcal M) \xlongrightarrow{\sim}\mathsf{Ho}(\mathcal N)$.
By \cite[Theorem~11.6.5]{Hirschhorn2003}, this induces equivalences $\Psi_D:\mathsf{Ho}(\mathcal M^D)\xlongrightarrow{\sim}\mathsf{Ho}(\mathcal N^D)$, where the diagram categories carry their projective model structures.
Since $\Psi$ preserves triangles and coproducts, $\mathcal S'=\Psi(\mathcal S)$ is localising.
Take a cofibrant replacement $QF\to F$ in $\mathcal M^D$.
Set $F'=L^D(QF)$ and $F'_d=L((QF)_d)$.
Then
\[
\begin{aligned}
\Psi(\mathrm{hocolim}_D F)& = (\mathbf L L)(\mathrm{hocolim}_D F)\cong(\mathbf L L)(\mathrm{colim}_D QF) \cong L(\mathrm{colim}_D QF)\\
&\cong\mathrm{colim}_D(L^D QF)=\mathrm{colim}_D F'\cong\mathrm{hocolim}_D F' \in \mathcal{S}'.
\end{aligned}
\]
Here $\mathrm{colim}_D$ and $L^D$ are left Quillen, $\mathrm{colim}_D QF$ and $F'$ are cofibrant in $\mathcal M$ and $\mathcal N^D$, respectively.
Since $\Psi$ is an equivalence, one has $\mathrm{hocolim}_D F\in\mathcal S$.
\end{proof}

\begin{corollary}\label{cor:localizing-colimit}
Let $\mathcal A$ be a locally Noetherian Grothendieck category and let $\mathcal L$ be a localising subcategory of $K(\mathrm{Inj}\,\mathcal A)$.
If $T:D\to\mathcal L$ is a filtered diagram, then its colimit in $\mathrm{Ch}(\mathcal A)$ belongs to $\mathcal L$.
\end{corollary}

\begin{proof}
The projective model structure $\mathcal M$ for dg modules is stable, with suspension given by the shift.
$\mathcal{M}$ is also combinatorial, since it is cofibrantly generated, and the Grothendieck category of dg modules is locally presentable \cite[Proposition~3.10]{Beke2000}.
The triangulated equivalence $\Phi$ preserves coproducts, so $\Phi(\mathcal L)$ is localising in $D(\mathcal C)$.
Proposition~\ref{lem:localizing-colimit} and Lemma~\ref{lem:localising-hocolimits} give
\[
 \Phi\!\left(\mathrm{colim}_{d\in D}T_d\right)
 \cong\mathrm{hocolim}_{d\in D}\Phi(T_d)\in\Phi(\mathcal L).
\]
Since $\Phi$ is an equivalence and $\mathcal L$ is closed under isomorphisms, $\mathrm{colim}_{d\in D}T_d\in\mathcal L$.
\end{proof}

\subsection{The cardinal criterion for coproducts}\label{subsec:cardinal-criterion}

We continue to assume that $\mathcal A$ is locally Noetherian, with generator $U$.
This subsection proves $(7)\Longrightarrow(6)$ of Theorem~\ref{thm:equivalence}.
By Lemma~\ref{lem:finite-endomorphism-generator}, we may assume that $\mathrm{End}_{\mathcal A}(U)$ is infinite.
Put $\kappa = |\mathrm{End}_{\mathcal A}(U)|$.

\begin{proposition}\label{prop:card-criterion}
  Each indecomposable injective object $E$ is a quotient of a coproduct of at most $\kappa$ Noetherian subobjects of $U$.
\end{proposition}

\begin{proof}
Let $(U_i)_{i\in I}$ be the Noetherian subobjects of $U$; they form a generating family since their directed union is $U$.
Choose a nonzero Noetherian subobject $B_0\subseteq E$.
We construct $B_n$ inductively for each $n \in \mathbb N_+$.
When $B_n$ is constructed, put
\[
\Omega_n=\{(i,S,f)\mid i\in I,\ S\subseteq U_i,\ f:S\to B_n\}.
\]
Since $E$ is injective, for each $f : S \to B_n$ with $S \subseteq U_i$ one can fix an extension $\widetilde f : U_i \to E$. 
This defines $B_{n+1}=B_n+\sum_{(i,S,f)\in\Omega_n}\mathrm{Im}\,\widetilde f$.

We claim that $E=\bigcup_{n<\omega}B_n$.
Since $S\subseteq U_i$ is Noetherian, every map $f:S\to\bigcup_{n<\omega}B_n$ factors through some $B_n$ and hence extends over $U_i$ with image in $B_{n+1}$.
The proof of Baer's criterion, applied to the generating family $(U_i)$, makes $\bigcup_{n<\omega}B_n$ injective \cite[\S1.10, Lemme~1, p.~136]{Grothendieck1957}.
Its inclusion in $E$ therefore splits; since it contains $B_0\ne0$ and $E$ is indecomposable, $\bigcup_{n<\omega}B_n=E$.

It remains to show that $|\Omega_n| \leq \kappa$ for each $n \in \mathbb N$.
We claim that
\begin{enumerate}
  \item The collection of Noetherian subobjects of $U$ has cardinality at most $\kappa$,
  \item For Noetherian objects $M$ and $N$, one has $|\mathrm{Hom}_{\mathcal A}(M,N)|\leq\kappa$.
\end{enumerate}
Indeed, each Noetherian subobject $U_i$ of $U$ is a quotient of finite coproducts of $U$, and hence $|I|$ is bounded by $|\mathrm{End}_\mathcal{A}(U)| = \kappa$.
For Noetherian objects $M$ and $N$, there is an injection of sets
\[
\mathrm{Hom}_\mathcal{A}(M, N) \hookrightarrow \{\text{subobjects of } M \oplus N\}, \qquad f \longmapsto \big[\begin{psmallmatrix}1\\f\end{psmallmatrix} : M \to M \oplus N\big].
\]
Since the collection of subobjects of $M \oplus N$ is bounded by $\kappa$, one has $|\mathrm{Hom}_\mathcal{A}(M, N)| \leq \kappa$.

We then show for each $n \in \mathbb N$ inductively the proposition $\mathfrak P(n)$:
\begin{quote}
  $B_n$ is a quotient of a coproduct of at most $\kappa$ objects $U_i$, and $|\Omega_n|\leq\kappa$.
\end{quote}
Since $B_0$ is Noetherian, it is a quotient of a finite coproduct of the $U_i$; the preceding estimates give $|\Omega_0|\leq\kappa$, so $\mathfrak P(0)$ holds.
Suppose $\mathfrak P(n)$ holds for some $n \in \mathbb N$.
By construction $B_{n+1}$ is a quotient of $B_n\oplus\coprod_{(i,S,f)\in\Omega_n}U_i$.
Since $|\Omega_n|\leq\kappa$, it is again a quotient of a coproduct of at most $\kappa$ objects $U_i$.
The preceding estimates bound the number of pairs $(i,S)$ and the size of each $\mathrm{Hom}_{\mathcal A}(S,B_{n+1})$ by $\kappa$, so $|\Omega_{n+1}|\leq\kappa$.
This yields $\mathfrak P(n+1)$, and the claim follows by induction.

Since $E=\bigcup_{n<\omega}B_n$, it is a quotient of $\coprod_{n<\omega}B_n$.
Combining the epimorphisms onto the $B_n$ and using that $\kappa$ is infinite gives the required quotient of at most $\kappa$ objects $U_i$.
This completes the proof.
\end{proof}

\begin{corollary}\label{lem:small-injectives}
  For a set of objects $(Y_t)_{t \in T}\subseteq \mathcal A$, every map from an indecomposable injective object $E$ to $\coprod_{t\in T}Y_t$ factors through some $\coprod_{t \in T'} Y_{t}$, where $T' \subseteq T$ and $|T'| \leq \kappa$.
\end{corollary}

\begin{proof}
  By Proposition~\ref{prop:card-criterion}, $E$ is a quotient of $\coprod_{s \in S} U_s$ where each $U_s$ is a Noetherian subobject of $U$ and $|S| \leq \kappa$.
  Then consider the composition $g : \coprod_{s \in S} U_s \twoheadrightarrow E \to \coprod_{t\in T}Y_t$.
  For each $s \in S$, the restriction $g : U_s \to \coprod_{t\in T}Y_t$ factors through a finite subcoproduct $\coprod_{t \in T_s} Y_t$ with $T_s \subseteq T$ finite.
  Set $T'=\bigcup_{s\in S}T_s$; then $|T'|\leq\kappa$, and $g$ factors through $\coprod_{t\in T'}Y_t$.
  The projection to $\coprod_{t\in T\setminus T'}Y_t$ vanishes after the quotient epimorphism, hence vanishes on $E$.
  Thus the original map factors through the same subcoproduct.
\end{proof}

\begin{proposition}\label{prop:cardinal-criterion}
  The class $\mathcal{GI}(\mathcal A)$ is closed under arbitrary coproducts if it is closed under $\kappa$-indexed coproducts.
\end{proposition}

\begin{proof}
Let $(G_i)_{i\in I}$ be Gorenstein injective objects, and choose complete injective resolutions $T_i$ with $Z^0T_i\cong G_i$.
Then $T=\coprod_{i\in I}T_i$ is acyclic with injective terms and $Z^nT=\coprod_{i\in I}Z^nT_i$.
To see that $T$ is totally acyclic, it suffices to show that $\mathrm{Hom}_{\mathcal A}(E,T)$ is acyclic for each indecomposable injective object $E$: this follows from the canonical isomorphism $\prod_{i} \mathrm{Ext}^1_\mathcal{A}(X_i, Y) \cong \mathrm{Ext}^1_\mathcal{A}(\coprod_{i}X_i, Y)$ for Grothendieck categories.

For $E\in\mathcal E$, every map $f : E\to Z^nT$ factors, by Corollary~\ref{lem:small-injectives}, through $\coprod_{j\in J}Z^nT_j$ for some $J\subseteq I$ with $|J|\leq\kappa$.  
Put the restriction $f' : E\to \coprod_{j \in J}Z^n T_j$.
The exact sequence
\[
 0\longrightarrow\coprod_{j\in J}Z^{n-1}T_j\longrightarrow\coprod_{j\in J}T_j^{n-1}
 \longrightarrow\coprod_{j\in J}Z^nT_j\longrightarrow0
\]
has Gorenstein injective left term, so $f'$ lifts to $\coprod_{j \in J}T_j^{n-1}$, and hence $f$ lifts to $T^{n-1}$.
Thus $\mathrm{Hom}_{\mathcal A}(E,T)$ is acyclic.
It follows that $T$ is totally acyclic, and hence $Z^0T\cong\coprod_{i\in I}G_i$ is Gorenstein injective.
\end{proof}

\subsection{From filtered colimits to covers}\label{subsec:small-tac}

Let $\mathcal A$ be locally Noetherian, and let $\mathcal C \subseteq \mathrm{Ch}(\mathcal A)$ denote the class of totally acyclic complexes of injective objects.
We prove $(5)\Longrightarrow(1)$ in three steps: first $\mathcal C$ is Kaplansky; then closure of $\mathcal{GI}(\mathcal A)$ under filtered colimits makes $\mathcal C$ a covering class; finally $\mathcal{C}$-covers yields $\mathcal{GI}(\mathcal{A})$-covers under certain assumptions.

The zero category is immediate.
For this construction, choose a generator $U$ with infinite endomorphism ring, replacing a generator by its countable coproduct if necessary.
Put $\kappa=|\mathrm{End}_{\mathcal A}(U)|$ and $\lambda=(2^\kappa)^+$.
For a regular cardinal $\lambda$, a category is \emph{$\lambda$-filtered} if every diagram in it indexed by a category with fewer than $\lambda$ morphisms admits a cocone.
An object $X$ is \emph{$\lambda$-presentable} if $\mathrm{Hom}_\mathcal{A}(X, -)$ preserves $\lambda$-filtered colimits.

A class $\mathcal F$ of objects is \emph{$\delta$-Kaplansky} if $\delta$ is regular and, for every $T\in\mathcal F$ and every $\delta$-presentable subobject $Y\subseteq T$, there is a $\delta$-presentable subobject $V$ with $Y\subseteq V\subseteq T$ and $V,T/V\in\mathcal F$.
$\mathcal{F}$ is \emph{Kaplansky} if it is $\delta$-Kaplansky for some regular cardinal $\delta$.

The next four estimates replace the module-cardinality bounds in the zig-zag construction.
Together with Noetherian covers of pullbacks, they allow the construction to be carried out intrinsically in $\mathcal A$.

\begin{lemma}\label{lem:tac-smallness}
Every indecomposable injective object $E$ is $\kappa^+$-presentable.
\end{lemma}
\begin{proof}
We first make explicit the cardinal estimates used below.
A Noetherian subobject of $U^{\oplus r}$ is the sum of the images of finitely many maps $U\to U^{\oplus r}$, so there are at most $\kappa$ such subobjects.
Every Noetherian object $M$ is a quotient of a Noetherian subobject $L\subseteq U^{\oplus r}$: choose an epimorphism $U^{\oplus r}\twoheadrightarrow M$ and use the directed union of Noetherian subobjects of its source to find $L$ mapping onto $M$.
Preimages in $L$ therefore bound the subobjects of $M$ by $\kappa$.
Graphs in $M\oplus N$ then give $|\mathrm{Hom}_{\mathcal A}(M,N)|\leq\kappa$ for Noetherian $M,N$.
By Proposition~\ref{prop:card-criterion}, $E$ is a quotient of a coproduct of at most $\kappa$ Noetherian objects.
The images of the finite subcoproducts form a directed family of at most $\kappa$ Noetherian subobjects with union $E$.
Noetherian objects in $\mathcal A$ are finitely presented, so this expresses $E$ as a $\kappa^+$-small colimit of $\kappa^+$-presentable objects.
Hence $E$ is $\kappa^+$-presentable, since $\kappa^+$-filtered colimits of sets commute with $\kappa^+$-small limits.
\end{proof}

\begin{corollary}\label{cor:small-coproduct-presentable}
A coproduct of $\leq 2^\kappa$ indecomposable injective objects is $\lambda$-presentable.
\end{corollary}
\begin{proof}
By Lemma~\ref{lem:tac-smallness}, each summand is $\lambda$-presentable.
Hom out of the coproduct is a product of fewer than $\lambda=(2^\kappa)^+$ Hom functors, and $\lambda$-filtered colimits of sets commute with such products.
\end{proof}

\begin{lemma}\label{lem:small-coproduct-subobjects}
A coproduct of at most $2^\kappa$ indecomposable injective objects has at most $2^\kappa$ Noetherian subobjects.
\end{lemma}
\begin{proof}
By Proposition~\ref{prop:card-criterion}, every finite coproduct of indecomposable injectives is a directed union of at most $\kappa$ Noetherian subobjects, obtained as images of finite subcoproducts of a Noetherian presentation.
Each has at most $\kappa$ subobjects, and every Noetherian subobject of the union lies in one of them.
Thus a finite coproduct has at most $\kappa$ Noetherian subobjects.
For the asserted bound, every Noetherian subobject lies in a finite subcoproduct, and there are at most $2^\kappa$ finite subcoproducts.
\end{proof}

\begin{corollary}\label{cor:small-coproduct-hom}
For every indecomposable injective object $E$ and every coproduct $X$ of at most $2^\kappa$ indecomposable injective objects, one has $|\mathrm{Hom}_{\mathcal A}(E,X)|\leq2^\kappa$.
\end{corollary}
\begin{proof}
For a Noetherian object $N$, every map $N\to X$ has Noetherian image.
There are at most $2^\kappa$ possible images by Lemma~\ref{lem:small-coproduct-subobjects}, and at most $\kappa$ maps to each by the Noetherian Hom estimate recalled above.
Choose an epimorphism $\coprod_{i\in I}U_i\twoheadrightarrow E$ as in Proposition~\ref{prop:card-criterion}, with $|I|\leq\kappa$ and each $U_i$ Noetherian.
Then
\[
 |\mathrm{Hom}_{\mathcal A}(E,X)|
 \leq\prod_{i\in I}|\mathrm{Hom}_{\mathcal A}(U_i,X)|
 \leq(2^\kappa)^\kappa=2^\kappa.
\]
\end{proof}

\begin{proposition}\label{prop:small-tac-subcomplex}
The class $\mathcal C$ is $\lambda$-Kaplansky in $\mathrm{Ch}(\mathcal A)$.
\end{proposition}

\begin{proof}
The result is known when $\mathcal A$ is a module category; \cite[Lemma~4.1.9 and Remark~4.1.10]{EnochsJenda2011VolII}.
This lemma generalises the result to locally finitely generated Grothendieck category, where Noetherian subobjects replace finite sets of elements.

Fix $T \in \mathcal C$ and a $\lambda$-presentable subcomplex $Y\subseteq T$.
It is clear that each $Y^n$ is $\lambda$-presentable.
Then decompose $T^n=\coprod_{a\in A^n}E_{n,a}$ into indecomposable injectives.
It remains to construct a $\lambda$-presentable complex $V$ with $Y \subseteq V \subseteq T$ and $V, T/V \in \mathcal C$.
The construction is given by transfinite induction on a continuous $\kappa^+$-chain, which yields $Y \subseteq V_0 \subseteq V_1 \subseteq \cdots \subseteq V_\beta \subseteq V_{\beta + 1} \subseteq \cdots$ with $\bigcup_{\beta < \kappa^+} V_\beta = V$, and is supposed to satisfy
\[
 V_\beta^n=\coprod_{a\in A_\beta^n}E_{n,a},
 \qquad A_\beta^n\subseteq A^n,
 \qquad |A_\beta^n|\leq2^\kappa,
\]

We first construct $V_0$ by induction.
Since the subcoproducts of $T^n$ with $\leq 2^\kappa$ summands form a $\lambda$-filtered union and $Y^n$ is $\lambda$-presentable, the inclusion $Y^n \subseteq T^n$ passes through some subcoproduct $\coprod_{a \in {}^{(0)}A^{n}_0}E_{n,a}$, denoted by ${}^{(0)}V_0^n$.
Note that $d({}^{(0)}V_0^{n})$ is not necessarily contained in ${}^{(0)}V_0^{n+1}$.
Hence we take the refinement ${}^{(1)}V_0^{n+1} = \coprod_{a \in {}^{(1)}A^{n+1}_0}E_{n+1,a}$ so that ${}^{(0)}V_0^{n+1} + d({}^{(0)}V_0^{n}) \subseteq {}^{(1)}V_0^{n+1}$.
By Corollary~\ref{lem:small-injectives}, the differential of each indecomposable summand factors through $\leq \kappa$ summands in the next degree.
Hence ${}^{(0)}V_0^{n+1} + d({}^{(0)}V_0^{n})$ factors through $\leq 2^\kappa + 2^\kappa \cdot \kappa = 2^\kappa$ summands and thus we may assume $|{}^{(1)}A_0^{n+1}| \leq 2^\kappa$.
By induction, one obtains
\[
\xymatrix@C=1cm@R=0.5cm{
\vdots \ar[dr]
& \vdots \ar[dr]
& \vdots
&
& \vdots \ar[d]
\\
{}^{(0)}V_0^{n-1} \ar@{}[r]|{\subset} \ar[dr]
& {}^{(1)}V_0^{n-1} \ar@{}[r]|{\subset} \ar[dr]
& {}^{(2)}V_0^{n-1} \ar@{}[r]|{\subset}
& \cdots
& V_0^{n-1} \ar[d]
\\
{}^{(0)}V_0^n \ar@{}[r]|{\subset} \ar[dr]
& {}^{(1)}V_0^n \ar@{}[r]|{\subset} \ar[dr]
& {}^{(2)}V_0^n \ar@{}[r]|{\subset}
& \cdots
& V_0^n \ar[d]
\\
{}^{(0)}V_0^{n+1} \ar@{}[r]|{\subset} \ar[dr]
& {}^{(1)}V_0^{n+1} \ar@{}[r]|{\subset} \ar[dr]
& {}^{(2)}V_0^{n+1} \ar@{}[r]|{\subset}
& \cdots
& V_0^{n+1} \ar[d]
\\
\vdots & \vdots & \vdots && \vdots
}
\]
Here ${}^{(m)}V_0^n = \coprod_{a \in {}^{(m)}A_0^n}E_{n,a}$, $|{}^{(m)}A_0^n| \leq 2^\kappa$ and ${}^{(m)}V_0^n+d_T^{n-1}({}^{(m)}V_0^{n-1})\subseteq{}^{(m+1)}V_0^n$.
Each solid arrow denote the differential restricted to subobjects.
Taking $V_0^n = \bigcup_{m \in \mathbb N}{}^{(m)}V_0^n$ completes the construction of $V_0$.

Suppose now that $V_\beta$ is constructed for some $\beta < \kappa^+$, then we construct $V_{\beta + 1}$.
Consider $Z^n V_\beta \subseteq V_\beta^n$.
For each Noetherian subobject $N \subseteq Z^n V_\beta$, one can take another noetherian object $N' \subseteq T^{n-1}$ such that the restricted $d : N' \to N$ is epic: indeed $N \subseteq Z^n T = B^n T$, and the preimage of $N$ along $d$ is a filtered union of its Noetherian subobjects.
Each such $N' \subseteq T^{n-1}$ lies in a finite subcoproduct indexed by $S_\beta^{n-1}(N')$, and one fixes the correspondence $N \mapsto S_\beta^{n-1}(N')$.
For each morphism $f$ from indecomposable injective $E$ to $Z^n T$, fix a lifting $\widetilde f : E \to T^{n-1}$.
By Corollary~\ref{lem:small-injectives}, $\widetilde f$ factors through a subcoproduct indexed by a set $L_\beta^{n-1}(\widetilde f)\subseteq A^{n-1}$ with $|L_\beta^{n-1}(\widetilde f)|\leq\kappa$.
Fix such $L_\beta^{n-1}(\widetilde f)$ for each $f$.
Take
\[
{}^{(0)}A_{\beta+1}^{n-1} = A_{\beta}^{n-1} \cup \bigcup_{N \subseteq Z^n V_\beta \ \text{Noetherian}} S_\beta^{n-1}(N') \cup \bigcup_{\substack{f : E \to Z^n V_\beta\\ E \ \text{indecomposable injective}}}L_\beta^{n-1}(\widetilde f).
\]
By construction $|A_{\beta}^{n-1}| \leq 2^\kappa$.
Since $V^n_\beta$ is a coproduct of $\leq 2^\kappa$ indecomposable injective objects, $Z^n V_\beta$ has $\leq 2^\kappa$ Noetherian subobjects.
Hence $|\bigcup_{N \subseteq Z^n V_\beta \ \text{Noetherian}} S_\beta^{n-1}(N')| \leq 2^\kappa$.
Each indecomposable injective object is an injective envelope of some Noetherian object, thus they has cardinal $\leq \kappa$.
By Corollary~\ref{cor:small-coproduct-hom}, $|\mathrm{Hom}_{\mathcal A}(E, Z^n V_\beta)|\leq 2^\kappa$.
It follows that $|{}^{(0)}A_{\beta+1}^{n-1}| \leq 2^\kappa$.
By the same construction in the initial setup, one obtains ${}^{(m)}A_{\beta+1}^{n}$ for each $n \in \mathbb Z$ and $m \in \mathbb N$.
Taking $A_{\beta + 1}^n = \bigcup_{m \in \mathbb N} {}^{(m)} A_{\beta+1}^n$ completes the construction of $V_{\beta+1}$.

Suppose that for any limit ordinal $\delta \leq \kappa^+$, each $(V_\beta)_{\beta < \delta}$ is constructed.
Take $V_\delta = \bigcup_{\beta < \delta} V_\beta$.
By construction $V_\delta ^n = \coprod_{a \in A_{\beta+1} ^n} E_{n,a}$, where $A_{\beta+1}^n =\bigcup _{\beta < \delta} A_\beta ^n$ has cardinal $\leq 2^\kappa$.
We claim that $V_\delta$ is an exact complex.
On one hand, $V_\beta$ is a complex since each Noetherian subobject $N \subseteq V_\delta^n$ lies in some $V_\alpha^n$ for $\alpha < \delta$, thus $N$ is annihilated by $d^2$.
On the other hand, each Noetherian subobject $N \subseteq Z^n V_\delta$ lies in $Z^n V_\beta$ for some $\beta < \delta$, and thus lies in $B^n V_{\beta+1}$ by construction.

Take $V=V_{\kappa^+}$.
By construction, $V$ is an acyclic complex with injective terms, and each $V^n$ is $\lambda$-presentable by Corollary~\ref{cor:small-coproduct-presentable}.
The equaliser description of chain maps uses only countable products of degreewise Hom sets, which commute with $\lambda$-filtered colimits; hence $V$ is $\lambda$-presentable in $\mathrm{Ch}(\mathcal A)$.
To see that $V\in\mathcal C$, it remains to show that every map $f:E\to Z^{n+1}V$, with $E$ indecomposable injective, lifts to $V^n$.
By Corollary~\ref{lem:small-injectives}, the composite $E\to Z^{n+1}V\hookrightarrow V^{n+1}$ factors through a subcoproduct with at most $\kappa$ summands.
Since $\kappa^+$ is regular, all these summands lie in some $V_\beta^{n+1}$ with $\beta<\kappa^+$.
Thus $f$ factors through $Z^{n+1}V_\beta$ and, by the successor construction, lifts to $V_{\beta+1}^n\subseteq V^n$.
Hence $\mathrm{Hom}_{\mathcal A}(E,V)$ is acyclic for every indecomposable injective $E$.
Every injective object is a coproduct of indecomposable injectives, and products of exact complexes of abelian groups are exact; therefore $V\in\mathcal C$.

Each $V^n$ is a subcoproduct of $T^n$, so $0\to V\to T\to T/V\to 0$ is degreewise split and $T/V$ has injective terms.
Since $V$ and $T$ are acyclic, so is $T/V$.
For every injective object $E$, applying $\mathrm{Hom}_{\mathcal A}(E,-)$ gives a short exact sequence of complexes of abelian groups.
Hence $\mathrm{Hom}_{\mathcal A}(E,T/V)$ is acyclic.
Thus $T/V\in\mathcal C$.

We have constructed a $\lambda$-presentable subcomplex $V$ with $Y\subseteq V\subseteq T$ and $V,T/V\in\mathcal C$.
Hence $\mathcal C$ is $\lambda$-Kaplansky.
\end{proof}

\begin{proposition}\label{lem:tac-approximations}
Let $\mathcal A$ be a locally Noetherian Grothendieck category, and let $\mathcal C$ be the class of totally acyclic complexes of injective objects in $\mathrm{Ch}(\mathcal A)$.
If $\mathcal{GI}(\mathcal A)$ is closed under filtered colimits, then $\mathcal C$ is a covering class in $\mathrm{Ch}(\mathcal A)$.
\end{proposition}

\begin{proof}
By Proposition~\ref{prop:small-tac-subcomplex}, the class $\mathcal C$ is Kaplansky, and it is clearly extension-closed.
By the assumption that $\mathcal{GI}(\mathcal A)$ is closed under filtered colimits and the exactness of filtered colimits, $\mathcal C$ is closed under filtered colimits.
By \cite[Proposition~1.6]{CEI2017}, an extension-closed Kaplansky class closed under filtered colimits in a Grothendieck category is covering.
Applying this result in $\mathrm{Ch}(\mathcal A)$ proves the assertion.
\end{proof}

\begin{corollary}\label{cor:tac-covers-to-gi-covers}
Let $\mathcal A$ be an abelian category with enough injective objects.
If the totally acyclic complexes of injective objects form a covering class in $\mathrm{Ch}(\mathcal A)$, then $\mathcal{GI}(\mathcal A)$ forms a covering class in $\mathcal A$.
\end{corollary}

\begin{proof}
Let $\mathcal C$ denote the class of totally acyclic complexes of injective objects.
The functor $Z^0:\mathcal C\to\mathcal{GI}(\mathcal A)$ is essentially surjective and full and there are natural isomorphisms
\[
\Phi : \mathrm{Hom}_{\mathrm{Ch}(\mathcal A)}(T,M[1])
\cong\mathrm{Hom}_{\mathcal A}(Z^0T,M)
\qquad(T\in\mathcal C).
\]
Here $M$ is the stalk complex concentrated in degree $0$.
It is straightforward to show that a $\mathcal C$-(pre)cover $\tau:T\to M[1]$, yields a $\mathcal{GI}(\mathcal A)$-(pre)cover $\Phi(\tau):Z^0T\to M$.
Hence $\mathcal{GI}(\mathcal A)$ is covering whenever $\mathcal C$ is covering.
\end{proof}

\begin{corollary}\label{cor:gi-covering-from-colimits}
Let $\mathcal A$ be a locally Noetherian Grothendieck category.
If $\mathcal{GI}(\mathcal A)$ is closed under filtered colimits, then it forms a covering class.
\end{corollary}
\begin{proof}
Combine Proposition~\ref{lem:tac-approximations} with Corollary~\ref{cor:tac-covers-to-gi-covers}.
\end{proof}

\subsection{Proof of \texorpdfstring{Theorem~\ref{thm:equivalence}}{Theorem A}}\label{sec:ProofA}

Let $\mathcal A$ be a locally finitely generated Grothendieck category with a generator $U$.
Theorem~\ref{thm:equivalence} states that the following conditions are equivalent.
\begin{enumerate}
\item $\mathcal{GI}(\mathcal A)$ forms a covering class.
\item $\mathcal{GI}(\mathcal A)$ forms a precovering class.
\item $\mathcal{GI}(\mathcal A)$ forms a special covering class.
\item $\mathcal{GI}(\mathcal A)$ forms a special precovering class.
\item $\mathcal{GI}(\mathcal A)$ is closed under all small filtered colimits.
\item $\mathcal{GI}(\mathcal A)$ is closed under arbitrary coproducts.
\item $\mathcal{GI}(\mathcal A)$ is closed under $|\mathrm{End}_{\mathcal A}(U)|$-indexed coproducts.
\item $\mathcal A$ is locally Noetherian and the following canonical inclusion has a right adjoint
\[
 K_{\mathrm{tac}}(\mathrm{Inj}\,\mathcal A)\longrightarrow K(\mathrm{Inj}\,\mathcal A).
\]
\end{enumerate}

\begin{proof}[Proof of Theorem~\ref{thm:equivalence}]
By Proposition~\ref{prop:noetherian-reduction}, each of the conditions implies that $\mathcal A$ is locally Noetherian.
We therefore assume throughout that $\mathcal A$ is locally Noetherian.
The proof is organised as follows:
\[
\xymatrix@C=2cm@R=.5cm{
(4) \ar@{->}[r] & (2) \ar@{->}[r] & (6) \ar@{->}[d]_{\square} \ar@<-0.6ex>@{->}[r] & (7) \ar@<-0.6ex>@{->}[l]_{\star} \\
(3) \ar@{->}[u] & (1) \ar@{->}[l] & (5) \ar@{->}[l]_{\triangle} \ar@{->}[r] & (8) \ar@{->}[lu]
}.
\]

\noindent\textit{$(1)\implies (3) \implies (4) \implies (2)$.}
By Wakamatsu lemma \cite[Lemma~7.1.13\textup{(1)}]{Krause2022}, An extension-closed covering class is special covering.
Hence $(1)$ implies $(3)$.
The implications $(3) \implies (4) \implies (2)$ are clear.

\noindent\textit{$(2)\implies (6)$.}
Suppose that $\mathcal{GI}(\mathcal A)$ is a precovering class.
Let $(G_i)_{i \in I}$ be a collection of Gorenstein injective objects.
By assumption, there is a $\mathcal{GI}(\mathcal A)$-precover $G \to \coprod_{i \in I} G_i$.
Since each structure map $G_i \to \coprod_{i \in I} G_i$ lifts to $G_i \to G$, $\coprod_{i \in I}G$ is a summand of $G$.
Hence $\mathcal{GI}(\mathcal A)$ is closed under coproducts, and $(2) \implies (6)$ follows.

\noindent\textit{$(5)\Longrightarrow(1)$.}
Apply Corollary~\ref{cor:gi-covering-from-colimits}.

\noindent\textit{$(6)\Longrightarrow(5)$.}
Assume that $\mathcal{GI}(\mathcal A)$ is closed under coproducts.
Let $\mathcal C$ be the class of totally acyclic complexes of injective terms.
Then $\mathcal C$ is closed under coproducts, and thus $K_{\mathrm{tac}}(\mathrm{Inj}\,\mathcal A) \subseteq K(\mathrm{Inj}\,\mathcal A)$ is localising.
For a filtered diagram $(G_d)_{d \in D}$, the functorial construction of complete injective resolution Proposition~\ref{lem:gi-basic} yields a diagram $\mathbf I(G_d)$.
Then by Corollary~\ref{cor:localizing-colimit}, the degreewise colimit $\varinjlim_{d \in D} \mathbf I (G_d)$ is totally acyclic.
Exactness of filtered colimits gives $\varinjlim_dG_d\cong Z^0\!\left(\varinjlim_d\mathbf I(G_d)\right)$, proving \textup{(5)}.

\noindent\textit{$(6)\Longleftrightarrow(7)$.}
We only prove $(7) \implies (6)$.
If $\mathrm{End}_{\mathcal A}(U)$ is finite, Lemma~\ref{lem:finite-endomorphism-generator} gives \textup{(1)}, hence \textup{(6)} and \textup{(7)};
otherwise apply Proposition~\ref{prop:cardinal-criterion}.

\noindent\textit{$(5)\Longrightarrow(8)$.}
Put $\mathcal T=K(\mathrm{Inj}\,\mathcal A)$ and $\mathcal S=K_{\mathrm{tac}}(\mathrm{Inj}\,\mathcal A)$.
Since precover of chain complexes passes to its homotopy category, Proposition~\ref{lem:tac-approximations} gives right $\mathcal S$-approximations in $\mathcal T$.
Since $\mathcal T$ has countable coproducts and $\mathcal S$ is thick, both have split idempotents.
Thus the inclusion $i_\ast : \mathcal S \subseteq \mathcal T$ of localising subcategory has a right adjoint by \cite[Proposition~3.2]{KrauseApproximations}.

\noindent\textit{$(8)\Longrightarrow(6)$.}
Suppose that $i_\ast:\mathcal S\to\mathcal T$ has right adjoint $i^!$.
For a family $(S_i)_{i \in I}$ in $\mathcal S$, put $S=\coprod_{i \in I} S_i$. Since $\mathcal{S}$ is a thick subcategory of $\mathcal{T}$, it follows from \cite[Lemma~3.2]{Krause2005} that the Verdier quotient functor $j^*:\mathcal{T}\longrightarrow \mathcal{T}/\mathcal{S}$ admits a right adjoint. Thus $j^*$ preserves coproducts, and hence $S\in \mathrm{Ker}(j^*)$. Since $\mathcal{S}$ is a thick subcategory of $\mathcal{T}$, $\mathrm{Ker}(j^*)=\mathcal{S}$ and hence $S\in \mathcal{S}$. Therefore $\mathcal{S}$ is closed under coproducts.

\end{proof}

\begin{question}
    Let $\mathcal A = \mathrm{Mod}_R$ for some ring $R$.
    If $\mathcal{GI}(\mathcal A)$ is closed under countable coproducts, then is it closed under arbitrary coproducts? 
    
    The author has obtained an affirmative answer under $V=L$. Thus, assuming the consistency of ZFC, no counterexample can be constructed in ZFC alone.
\end{question}

\subsection{Adjoints and recollements}\label{subsec:tac-recollement}
Let $\mathcal A$ be a locally Noetherian Grothendieck category.
Put $\mathcal T=K(\mathrm{Inj}\,\mathcal A)$ and $\mathcal S=K_{\mathrm{tac}}(\mathrm{Inj}\,\mathcal A)$.
This subsection aims to study the existence of the adjoint triple
\[
\xymatrix@C=3cm{
\mathcal S
\ar@{ >->}[r]|{i_\ast}
& \mathcal T
\ar@{..>}@<1.2ex>[l]^{i^!}
\ar@{..>}@<-1.2ex>[l]_{i^\ast}
}.
\]

\begin{proposition}\label{prop:leftadj}
    The following statements are equivalent:
    \begin{enumerate}
        \item $i_\ast$ admits a left adjoint $i^\ast$,
        \item $({}^\perp \mathcal{GI}(\mathcal A), \mathcal{GI}(\mathcal A))$ is a complete cotorsion pair,
        \item there exists a generator $U \in \mathcal A$ such that $U \in {}^\perp \mathcal{GI}(\mathcal A)$.
    \end{enumerate}
\end{proposition}

\begin{proof}
The equivalence $(2)\Longleftrightarrow(3)$ follows from \cite[Corollary~5.8]{EG26}, since a generating set in ${}^\perp\mathcal{GI}(\mathcal A)$ can be replaced by its coproduct.

\noindent\textit{$(3)\Longrightarrow(1)$.}
Let $\mathcal C \subseteq \mathrm{Ch}(\mathrm{Inj}\mathcal A)$ be the collection of totally acyclic complexes.
By \cite[Corollary~4.3]{EG26}, every $T\in\mathcal T$ admits a special $\mathcal C$-preenvelope $e_T:T \to C_T$, which yields degreewise split short exact sequence in $\mathrm{Ch}(\mathcal A)$
\[
0\longrightarrow T\longrightarrow C_T\longrightarrow V_T\longrightarrow0, \qquad C_T \in \mathcal C, \quad V_T \in {}^\perp \mathcal C.
\]
This yields a distinguished triangle $T \xrightarrow{e_T}C_T \to V_T \to T[1]$ in $\mathcal T$.
For each $S \in \mathcal S$ and $n \in \mathbb Z$, the orthogonal relation yields that $\mathrm{Hom}_{\mathcal T}(V_T[n], S) = 0$.
Hence
\[
\mathrm{Hom}_{\mathcal S}(C_T, S) \cong \mathrm{Hom}_{\mathcal T}(C_T, S) \xrightarrow[\sim]{- \circ e_T} \mathrm{Hom}_{\mathcal T}(T, S).
\]
It follows that $T \mapsto C_T$ defines a left adjoint of $i_\ast$.

\noindent\textit{$(1)\Longrightarrow(3)$.}
Let $E$ be an injective resolution of a generator $U$.
Then complete the unit $\eta_E : E \to i_\ast i^\ast E$ to the distinguished triangle $E\xlongrightarrow{\eta_E} i_\ast i^\ast E\longrightarrow C\longrightarrow E[1]$.
For each $S \in \mathcal S$, adjunction $i^\ast \dashv i_\ast$ gives $\mathrm{Hom}_{\mathcal T}(\eta_E, S) : \mathrm{Hom}_{\mathcal T}(i_\ast i^\ast E, S) \cong \mathrm{Hom}_{\mathcal T}(E, S)$.
Hence $\mathrm{Cone}(\eta_E)\in{}^\perp\mathcal S$.
Since $i_\ast i^\ast E$ is acyclic, $H^{-1}C\longrightarrow H^0E\cong U$.
Since $U$ is a quotient of $Z^{-1} C$, one sees that $Z^{-1}C$ is a generator.
Since $H^{>0}(C[-1])=0$, \cite[Lemma~7.1]{Krause2005} gives
\[
\overline{\mathrm{Hom}}_{\mathcal A}(Z^{-1}C,Z^0S)
\cong\mathrm{Hom}_{\mathcal T}(C[-1],S)=0
\qquad(S\in\mathcal S),
\]
where $\overline{\mathrm{Hom}}_{\mathcal A}(-,?) = \frac{{\mathrm{Hom}}_{\mathcal A}(-,?)}{\{\text{morphisms factors through injectives}\}}$.
Apply $\mathrm{Hom}_{\mathcal A}(Z^{-1}C,-)$ to
\[
0\longrightarrow Z^{-1}S\longrightarrow S^{-1}
\xlongrightarrow \pi Z^0S\longrightarrow0.
\]
Since $S^{-1}$ is injective, this gives
\[
\mathrm{Ext}^1_{\mathcal A}(Z^{-1}C,Z^{-1}S)
\cong
\frac{\mathrm{Hom}_{\mathcal A}(Z^{-1}C,Z^0S)}
{\mathrm{Im}\ \mathrm{Hom}_{\mathcal A}(Z^{-1}C,\pi)}.
\]
By total acyclicity, every map from an injective object to $Z^0S$ passes through $S^{-1}\to Z^0S$.
Thus $\mathrm{Im}\ \mathrm{Hom}_{\mathcal A}(Z^{-1}C,\pi)$ consists precisely of the maps in $\mathrm{Hom}_{\mathcal A}(Z^{-1}C,Z^0S)$ factoring through injective objects.
Hence
\[
\mathrm{Ext}^1_{\mathcal A}(Z^{-1}C,Z^{-1}S)
\cong\overline{\mathrm{Hom}}_{\mathcal A}(Z^{-1}C,Z^0S)=0.
\]
Every object in $\mathcal{GI}(\mathcal A)$ occurs as $Z^{-1}S$ for some $S\in\mathcal S$.
Thus $Z^{-1}C\in{}^\perp\mathcal{GI}(\mathcal A)$.
\end{proof}

\begin{remark}
    It is not asserted that $(^\perp \mathcal{GI}(\mathcal A), \mathcal{GI}(\mathcal A))$ is a cotorsion pair.
\end{remark}

\begin{remark}
    If $i_\ast$ has a left adjoint, then $\mathcal S$ has arbitrary coproducts
    \[
    \coprod_{i \in I} ^{\mathcal S} S_i \cong i^\ast \left(\coprod_{i \in I}^{\mathcal T} i_\ast (S_i)\right).
    \]
    Note that this coproduct is not computed in $\mathcal T$ is general.
    In particular $\mathcal S$ is not necessarily a localising subcategory of $\mathcal T$.
    By Theorem~\ref{thm:equivalence}, if $\mathcal S$ admits a right adjoint, then $\mathcal S$ has coproducts which is computed in $\mathcal T$.
    In this case $\mathcal S$ a localising subcategory of $\mathcal T$.
\end{remark}

\begin{proposition}\label{prop:tac-recollement}
The inclusion $i_\ast$ has both adjoints if and only if
\[
\mathcal S\xrightarrow{i_\ast}\mathcal T
\xrightarrow{j^\ast}\mathcal Q:=\mathcal T/\mathcal S
\]
extends to a recollement.
In this case, $\mathcal Q$ is locally small, and
\[
i^\ast\dashv i_\ast\dashv i^!,\qquad
j_!\dashv j^\ast\dashv j_\ast,\qquad
{}^\perp\mathcal S\simeq\mathcal Q\simeq\mathcal S^\perp,
\]
where $j_!$ and $j_\ast$ identify $\mathcal Q$ with ${}^\perp\mathcal S$ and $\mathcal S^\perp$, respectively.
\end{proposition}
\begin{proof}
Apply \cite[Lemma~3.2]{Krause2005} and its dual to the thick subcategory $\mathcal S\subseteq\mathcal T$.
\end{proof}

In particular, if $D(\mathcal A)$ is compactly generated, then $i_\ast$ has a left adjoint \cite[Lemma~7.3]{Krause2005}.

\begin{example}
Let $R$ be a ring such that $\mathcal{GI}(R)$ forms a covering class.
Put $\mathcal S = K_{\mathrm{tac}}(\mathrm{Inj}\,R)$ and $\mathcal T = K(\mathrm{Inj}\,R)$, respectively.
Then one has the adjoint triple
\[
\xymatrix@C=3cm{
\mathcal S
\ar@{ >->}[r]|{i_\ast}
& \mathcal T
\ar@<1.2ex>[l]^{i^!}
\ar@<-1.2ex>[l]_{i^\ast}
}
\]
where $i_\ast:\mathcal S\to\mathcal T$ is the canonical inclusion.

Let $\mathcal M$ be a set of representatives of the finitely generated objects of $\mathrm{Mod}_R$, and choose an injective resolution $\iota M$ for each $M\in\mathcal M$.
Then the objects $i^\ast(\iota M)$, $M\in\mathcal M$, form a set of compact generators of $\mathcal S$, and the compact objects of $\mathcal S$ are
\[
 \mathcal S^c=\operatorname{thick}_{\mathcal S}\{i^\ast(\iota M) \mid M\in\mathcal M\}.
\]
\end{example}

\begin{proof}
By Theorem~\ref{thm:equivalence}, $R$ is Noetherian and $\mathcal S$ is localising in $\mathcal T$.
The category $\mathcal T$ is compactly generated by the $\iota M$ \cite[Proposition~2.3]{Krause2005}.
The left adjoint $i^\ast$ exists taking $U = R$ in Lemma~\ref{prop:leftadj} (3).
The inclusion $i_\ast:\mathcal S\to\mathcal T$ preserves coproducts.
By adjunction
\[
 \mathrm{Hom}_{\mathcal S}(i^\ast(\iota M),-)\cong\mathrm{Hom}_{\mathcal T}(\iota M,i_\ast(-)),
\]
every $i^\ast(\iota M)$ is compact.
If $\mathrm{Hom}_{\mathcal S}(i^\ast(\iota M),X[n])=0$ for every $M\in\mathcal M$ and $n\in\mathbb Z$, adjunction gives $i_\ast X=0$, hence $X=0$.
Thus these objects generate $\mathcal S$, and its compact objects are their thick closure \cite[Proposition~3.4.15]{Krause2022}.
\end{proof}

\section{A Noetherian ring without Gorenstein injective covers}\label{sec:nonclosure}

This section shows that Gorenstein injective modules over a commutative Noetherian ring need not form a covering class.

More precisely, we construct a countable one-dimensional Cohen--Macaulay local ring $R$ and a sequence $(G_n)_{n\in\mathbb N_+}$ of Gorenstein injective $R$-modules whose coproduct has infinite Gorenstein injective dimension and admits no Gorenstein injective precover.

We first outline the proof in Subsection~\ref{subsec:exampleB-outline}, then establish the basic properties of $R$ and determine its indecomposable injective modules in the subsequent subsections. 
In particular, the construction for $(G_n)_{n \in \mathbb N_+}$ are explained in Subsection~\ref{The Gorenstein injective modules}.

\subsection{Outline of the example}\label{subsec:exampleB-outline}

Let $\mathbb Q[\![t]\!]$ be the formal power series ring, and let $\mathbb Q[\![t]\!][t^{-1}]$ be the field of formal Laurent series.
Fix $s_1=e^t-1$ and $s_2=e^{t^2}-1$ that are algebraically independent over $\mathbb{Q}(t)$.
Then $F=\mathbb Q(t)(s_1,s_2)$ is the field of rational functions with $3$ indeterminates.
For each $i = 1,2$, define the formal derivations
\[
\delta_i  : F \longrightarrow F, \quad f(t,s_1,s_2) \longmapsto \tfrac{\partial }{\partial s_i} f(t,s_1,s_2).
\]
Define $A=\left\{ a\in F\cap \mathbb Q[\![t]\!] \mid \delta_1(a),\delta_2(a)\in \mathbb Q[\![t]\!] \right\}$.
Then $A$ is Noetherian by Corollary~\ref{cor:A-noetherian} and $\operatorname{spec} A = \{0, \mathfrak m_A\}$ by Lemma~\ref{lem:dimA}.
Then the final ring for the example is
\[
R=
\left\{
\begin{psmallmatrix}
a&b&c\\
0&a&0\\
0&0&a
\end{psmallmatrix}
 \mid a,b,c\in A
\right\}.
\]
Since the surjection of the diagonal $\pi : R \twoheadrightarrow A$ has a nilpotent kernel, one has
\[
\operatorname{spec}(R) = \{\pi^{-1}(\mathfrak m_A), \pi^{-1}(0)\} = \{\mathfrak m_R , \mathfrak n\}.
\]
It is shown in Lemma~\ref{lem:R} that $R$ is a Cohen–Macaulay local ring of Krull dimension one.

Take the finitely generated $R$-module $C = \mathrm{Hom}_A(R, A)$.
Multiplying by $t$ gives an $\mathbb N$-indexed chain $C \xrightarrow{t} C \xrightarrow{t} C \xrightarrow{t} \cdots$, which yields a short exact sequence
\begin{equation}
0 \longrightarrow C^{(\mathbb N_+)} \xlongrightarrow{1 - \nabla}  C^{(\mathbb N_+)} \longrightarrow \varinjlim C \longrightarrow 0.\label{eq:B-telescope}
\end{equation}
Then Proposition~\ref{prop:compGI} and Lemma~\ref{lem:Tn} yield a sequence $(G_n)_{n \in \mathbb N_+}$ of Gorenstein injective $R$-modules such that
\begin{equation}
t^n \cdot \Ext^1_R(C, G_n) = 0, \qquad t^{n-1} \cdot \Ext^1_R(C, G_n) \neq 0, \quad (n \in \mathbb N_+). \label{eq:B-Ext-torsion}
\end{equation}
Applying $\mathrm{Hom}_R(-, \coprod _{n \in \mathbb N_+}G_n)$ to \eqref{eq:B-telescope} gives a commutative diagram
\[
\xymatrix@C=1cm@R=.5cm{
\left(\coprod_{n\in\mathbb N_+}
\mathrm{Ext}^1_R\left(C,G_n\right)\right)^{\mathbb N_+}
\ar[r]^-{\simeq} \ar[d]_{1-\Delta}
&
\mathrm{Ext}^1_R\left(C^{(\mathbb N_+)},\coprod_{n\in\mathbb N_+}G_n\right)
\ar[d]^{\mathrm{Ext}^1(1-\nabla,\  \coprod_{n\in\mathbb N_+}G_n)}
\\
\left(\coprod_{n\in\mathbb N_+}
\mathrm{Ext}^1_R\left(C,G_n\right)\right)^{\mathbb N_+}
\ar[r]^-{\simeq} \ar@{->>}[d]
&
\mathrm{Ext}^1_R\left(C^{(\mathbb N_+)},\coprod_{n\in\mathbb N_+}G_n\right)
\ar[d]
\\
\lim^1\left(\coprod_{n\in\mathbb N_+}\mathrm{Ext}^1_R\left(C,G_n\right)\right)
\ar@{ >.>}[r]
&
\mathrm{Ext}^2_R\left(\varinjlim C,\coprod_{n\in\mathbb N_+}G_n\right)
}
\]
Here the columns are exact at middle, and $\Delta$ is the shift of the tower
\[
\cdots \xlongrightarrow {t} \coprod_{n\in\mathbb N_+}
\mathrm{Ext}^1_R\left(C,G_n\right) \xlongrightarrow {t} \coprod_{n\in\mathbb N_+}
\mathrm{Ext}^1_R\left(C,G_n\right)\xlongrightarrow {t} \coprod_{n\in\mathbb N_+}
\mathrm{Ext}^1_R\left(C,G_n\right).
\]
By \eqref{eq:B-Ext-torsion} there is $\eta_n \in \Ext^1_R(C, G_n)$ with $t^n\cdot \eta_n = 0$ while $t^{n-1}\cdot \eta_n \neq 0$ for each $n \in \mathbb N_+$.
Then $(\eta_n)_{n \in \mathbb N_+} \in \left(\coprod_{n\in\mathbb N_+}
\mathrm{Ext}^1_R\left(C,G_n\right)\right)^{\mathbb N_+} \setminus \mathrm{im}(1- \Delta)$, and thus $\mathrm{Ext}^2_R\left(\varinjlim C,\coprod_{n\in\mathbb N_+}G_n\right) \neq 0$.
Here $\varinjlim C\cong C[t^{-1}]\cong\mathrm{Hom}_A(R,A[t^{-1}])$.
By Lemma~\ref{lem:A-localisation} and Corollary~\ref{cor:J-Hom}, one has $\mathrm{Hom}_A(R,A[t^{-1}]) = \mathrm{Hom}_A(R,F) \cong E_R(R/\mathfrak n)$.
It follows that $\coprod_{n\in\mathbb N_+} G_n$ is not Gorenstein injective.
By Theorem~\ref{thm:equivalence}, $\mathcal{GI}(R)$ is not a covering class.

We then show that $\coprod_{n\in\mathbb N_+}G_n$ has infinite $\mathcal{GI}(R)$-injective dimension.
By construction in Subsection~\ref{The Gorenstein injective modules}, we may assume the complete injective resolution $G_n$ to be a one-periodic complex.
Hence there are short exact sequences $0 \to G_n \to I_n \to G_n \to 0$ which remain exact by applying $\mathrm{Hom}_R(E_R(R/\mathfrak n), -)$.
Taking coproducts yields an exact sequence $0 \to \coprod_{n\in\mathbb N_+}G_n \to \coprod_{n\in\mathbb N_+}I_n \to \coprod_{n\in\mathbb N_+}G_n \to 0$.
Since $R$ is Noetherian, $\coprod_{n\in\mathbb N_+}I_n$ is injective.
Hence dimension shift gives $\mathrm{Ext}_R^d(E_R(R/\mathfrak n),\coprod_{n\in\mathbb N_+}G_n)\ne0$ for each $d \geq 1$.
It follows that $\coprod_{n\in\mathbb N_+}G_n$ has infinite injective dimension.
Note that $\coprod_{n\in\mathbb N_+}G_n$ has infinite Gorenstein injective dimension; otherwise a finite Gorenstein injective coresolution of $M$ would imply
$\mathrm{Ext}_R^d(E,M)=0$ for every module $E$
and all sufficiently large $d$, a contradiction.

We finally show that $\coprod_{n\in\mathbb N_+}G_n$ admits no $\mathcal{GI}(R)$-precover.
Suppose that $f:G\to \coprod_{n\in\mathbb N_+}G_n$ were a $\mathcal{GI}(R)$-precover.
Each inclusion $G_n\to \coprod_{n\in\mathbb N_+}G_n$ factors through $f$, and the resulting maps assemble into $h:M\to G$ with $fh=1_M$.
It follows that $\mathrm{Ext}_R^2(E_R(R/\mathfrak n),\coprod_{n\in\mathbb N_+}G_n) \neq 0$ is a direct summand of
$\mathrm{Ext}_R^2(E_R(R/\mathfrak n),G)=0$, contradiction.
Hence $M$ admits no $\mathcal{GI}(R)$-precover.

\subsection{The one dimensional CM local ring \texorpdfstring{$R$}{R}}

We show that $A$ is a one dimensional CM local domain, then pass these properties to $R$.

\begin{lemma}\label{lem:A-local}
$A$ is a local domain with maximal ideal $\mathfrak m_A=A\cap t\mathbb Q[\![t]\!]=(t,s_1,s_2)$.
\end{lemma}

\begin{proof}
It is straightforward to check that $A$ is a ring. 
Since $A\subseteq F$, $A$ is a domain. 

We claim that $\mathfrak m_A=A\cap t\mathbb Q[\![t]\!]$ is the unique maximal ideal.
It suffices to show that each $a \in A \setminus \mathfrak m_A$ is a unit.
Note that $a$ is a unit in $F \cap \mathbb Q[\![t]\!]$.
It remains to show that $a^{-1} \in A$.
By construction, $\delta_i(a^{-1}) = -a^{-2}\delta _i (a) \in \mathbb Q[\![t]\!]$ for $i = 1,2$.
Thus $a^{-1} \in A$, proving the claim.

We next prove that $\mathfrak m_A=(t,s_1,s_2)$. 
Clearly $(t,s_1,s_2)\subseteq\mathfrak m_A$. 
Conversely, fix an arbitrary $b\in\mathfrak m_A$.
For each $i =1 ,2$, $b$ takes the form $\delta_i(b)=c_i+td_i$ with $c_i \in \mathbb Q$ and $d_i \in \mathbb Q[\![t]\!]$.
Set $c=b-c_1s_1-c_2s_2$.
Then $c\in t\mathbb Q[\![t]\!]$ and hence $c/t\in F\cap \mathbb Q[\![t]\!]$. 
Since for each $i = 1,2$, \(
\delta_i(c/t)=\frac{\delta_i(c)}{t}=\frac{\delta_i(b)-c_i}{t}=d_i\in \mathbb Q[\![t]\!].\) By definition $c / t \in A$, and hence $c\in tA$. 
It follows that $b\in At+As_1+As_2$, proving $\mathfrak m_A \subseteq (t,s_1,s_2)$. 
This completes the proof.
\end{proof}

\begin{lemma}
    For the local domain $(A, \mathfrak m_A)$, one has $\mathfrak m_A^2=t\mathfrak m_A$.
\end{lemma}

\begin{proof}
On one hand, any $b,c\in\mathfrak m_A$ yield that $bc/t\in t\mathbb Q[\![t]\!]$. 
Since \(\delta_i(bc/t)=\frac{b\delta_i(c)+c\delta_i(b)}{t}\in \mathbb Q[\![t]\!]\), $bc/t\in\mathfrak m_A$. 
Thus \(\mathfrak m_A^2\subseteq t\mathfrak m_A\). 
On the other hand, $\mathfrak m_A^2 \supseteq t\mathfrak m_A$ follows from $t\in\mathfrak m_A$.
\end{proof}

\begin{proposition}\label{lem:A-noeth}
For every $a\in A \setminus \{0\}$, $A/(a)$ is a finite-dimensional $\mathbb Q$-vector space. 
\end{proposition}

\begin{proof}
If $a$ is a unit, there is nothing to prove. Otherwise, let $r\in\mathbb N$ be the largest integer such that $a\in t^r\mathbb Q[\![t]\!]$. Then $r\geq1$.
Then $t^{2r}/a\in F\cap \mathbb Q[\![t]\!]$, and
\[
\delta_i(t^{2r}/a)=-\frac{t^{2r}\delta_i(a)}{a^2}\in \mathbb Q[\![t]\!]
\qquad(i=1,2).
\]
Thus $t^{2r}/a\in A$.
It follows that $t^{2r}\in aA$, and $\mathfrak m_A^{2r+1}=t^{2r}\mathfrak m_A\subseteq aA$.
Since $\mathfrak m_A=(t,s_1,s_2)$ is finitely generated with residue field $A/\mathfrak m_A\cong\mathbb Q$,
each $\mathfrak m_A^k/\mathfrak m_A^{k+1}$ is finite-dimensional over $\mathbb Q$.
Hence $A/\mathfrak m_A^{2r+1}$, and therefore its quotient $A/aA$, is finite-dimensional over $\mathbb Q$.
\end{proof}

\begin{corollary}\label{cor:A-noetherian}
    $A$ is Noetherian.
\end{corollary}
\begin{proof}
Let $0\ne I\subseteq A$ be an ideal and choose $0\ne a\in I$.
By Proposition~\ref{lem:A-noeth}, $I/aA$ is finite-dimensional over $\mathbb Q$.
Lifting a basis gives $b_1,\ldots,b_r\in I$ with $I=(a,b_1,\ldots,b_r)$.
Thus every ideal of $A$ is finitely generated.
\end{proof}

\begin{lemma}\label{lem:dimA}
One has
\(\Spec A=\{(0),\mathfrak m_A\}\) and \(\dim A=1\).
\end{lemma}
\begin{proof}
Let $0\ne\mathfrak q\in\Spec A$ and choose $0\ne a\in\mathfrak q$.
Then $A/\mathfrak q$ is a domain and a quotient of the finite-dimensional
$\mathbb Q$-algebra $A/aA$. Hence $A/\mathfrak q$ is a field, so
$\mathfrak q$ is maximal. Since $A$ is local, $\mathfrak q=\mathfrak m_A$.
As $t$ is a nonzero nonunit, \((0)\subsetneq\mathfrak m_A\),
and therefore $\dim A=1$.
\end{proof}

\begin{lemma}\label{lem:R}
$R$ is a commutative Cohen--Macaulay local ring of Krull dimension one.
\end{lemma}

\begin{proof}
Since $R$ is a finitely generated $A$-algebra, $R$ is Noetherian.
Consider the ring surjection with the nilpotent kernel
\[
R \longrightarrow A, \qquad \begin{psmallmatrix}
a&b&c\\
0&a&0\\
0&0&a
\end{psmallmatrix} \longmapsto a.
\]
Hence $\operatorname{Spec}A = \mathrm{Spec} R$.
Thus $R$ is local with Krull dimension $1$.

Since multiplying by $t$ is injective on $R$, one has $\operatorname{depth}R\ge1$.
On the other hand \(\operatorname{depth}R \le \dim R=1\).
Hence $R$ is Cohen--Macaulay.
\end{proof}

\subsection{Indecomposable injective modules}\label{subsec:R-injectives}

The indecomposable injective $R$-modules are $E_R(R/\mathfrak n)$ and $E_R(\mathbb Q)$.
We first determine $E_R(R/\mathfrak n)$.

\begin{lemma}\label{lem:A-localisation}
One has $A[t^{-1}]=F$.
\end{lemma}
\begin{proof}
The inclusion $A[t^{-1}]\subseteq F$ is clear. Conversely, let $f\in F$.
Choose $N$ sufficiently large that \(t^Nf, t^N\delta_1(f), t^N\delta_2(f)\)
all lie in $\mathbb Q[\![t]\!]$. Since $\delta_i(t)=0$,
\(\delta_i(t^Nf)=t^N\delta_i(f)\in \mathbb Q[\![t]\!]\).
Thus $t^Nf\in A$, and hence $f\in A[t^{-1}]$.
\end{proof}

\begin{corollary}
    $R[t^{-1}] \cong \left(A[u,v]/(u,v)^2\right)[t^{-1}]= (A[t^{-1}])[u,v]/(u,v)^2 = F[u,v]/(u,v)^2$.
\end{corollary}

\begin{lemma}\label{lem:J-injective}
    $E_{R[t^{-1}]}(F) \cong E_R(R/\mathfrak n)$ as $R$-modules.
\end{lemma}

\begin{proof}
Since $R/\mathfrak n\subseteq E_R(R/\mathfrak n)$ is an essential extension of $R$-modules, and multiplying by $t$ is injective on $R / \mathfrak n \cong A$, one sees that multiplying by $t$ is injective on $E_R(R/\mathfrak n)$.
Since $E_R(R/\mathfrak n)$ is and indecomposable injective module, $t : E_R(R/\mathfrak n) \to E_R(R/\mathfrak n)$ is an automorphism.
By \cite[Lemma~3.2.5]{BrunsHerzog}, injective envelope is compatible with localisations, which gives
\[
E_R(R/\mathfrak n)\cong E_R(R/\mathfrak n)[t^{-1}]\cong E_{R[t^{-1}]}((R/\mathfrak n)[t^{-1}])=E_{R[t^{-1}]}(F).
\]
This completes the proof.
\end{proof}

The following proposition computes injective envelopes via ring changes.

\begin{proposition}\label{prop:finite-local-injective}
Let $(B,\mathfrak p)\to(T,\mathfrak m)$ be a homomorphism of commutative Noetherian local rings.
Suppose that $T$ is a finitely generated $B$-module, and the morphism induces an isomorphism of residue fields $k = B/ \mathfrak p \cong T/\mathfrak m$.
Then $E_T(k)\cong\mathrm{Hom}_B(T,E_B(k))$.
\end{proposition}

\begin{proof}
By \cite[Lemma~3.1.6]{BrunsHerzog}, $\mathrm{Hom}_B(T,E_B(k))$ is injective over $T$, and adjunction gives
\[
\mathrm{Hom}_T(k,\mathrm{Hom}_B(T,E_B(k)))\cong\mathrm{Hom}_B(k,E_B(k))\cong k.
\]
The module $E_B(k)$ is $\mathfrak p$-torsion, i.e., every element is
annihilated by some power of $\mathfrak p$ \cite[Theorem~18.4(v)]{Matsumura1989}.
Since $T$ is a finitely generated $B$-module, $\mathrm{Hom}_B(T,E_B(k))$ is also $\mathfrak p$-torsion.
Moreover, $T/\mathfrak pT$ is an Artinian local ring, so $\mathfrak m^h\subseteq\mathfrak pT$ for some $h$.
Hence $\mathrm{Hom}_B(T,E_B(k))$ is $\mathfrak m$-torsion with one-dimensional socle $k$.
In this case $k$ is essential, and thus $E_T(k) \cong \mathrm{Hom}_B(T,E_B(k))$.
This completes the proof.
\end{proof}

\begin{corollary}\label{cor:J-Hom}
    $E_R(R/\mathfrak n) \cong \mathrm{Hom}_A(R,F)$ as $R$-modules.
\end{corollary}

\begin{proof}
    By Proposition~\ref{prop:finite-local-injective}, $F\to R[t^{-1}]$ yields
    \[
    E_R(R/\mathfrak n)\cong\mathrm{Hom}_F(R[t^{-1}],F)\cong\mathrm{Hom}_A(R,F).
    \]
\end{proof}

Then we determine $E_R(\mathbb Q)$ via the ring completion $\widehat R = \varprojlim_{q \geq 1} R / (\mathfrak m_R)^q$.

\begin{lemma}
    $\widehat A \cong \mathbb Q[\![t]\!] \otimes _{\mathbb Q} \mathbb Q[x_1, x_2]/(x_1,x_2)^2$.
\end{lemma}

\begin{proof}
We first claim that the homomorphism
\[
A\longrightarrow \mathbb Q[\![t]\!][x_1, x_2]/(x_1, x_2)^2, \qquad 
a\longmapsto a+\delta_1(a)x_1+\delta_2(a)x_2
\]
induces, for each $q\ge1$, an isomorphism $A/t^qA\cong (\mathbb Q[\![t]\!]/t^q\mathbb Q[\![t]\!])[x_1,x_2]/(x_1,x_2)^2$.
\begin{enumerate}
    \item It is injective. 
    If $a,\delta_1(a),\delta_2(a)\in t^q\mathbb Q[\![t]\!]$, then $a \in t^qA$.
    \item It is surjective. 
    For $i=1,2$, choose $p_{i,q}\in\mathbb Q[t]$ with $s_i-p_{i,q}\in t^q\mathbb Q[\![t]\!]$.
    Since $\delta_j(t)=0$ and $\delta_j(s_i-p_{i,q})=\delta_{ij}$, the map sends $t$ to $\bar t$ and $s_i-p_{i,q}$ to $x_i$.
    Note that $\bar t,x_1,x_2$ generate $(\mathbb Q[\![t]\!]/t^q\mathbb Q[\![t]\!])[x_1,x_2]/(x_1,x_2)^2$ as a $\mathbb Q$-algebra,
    thus the induced map is surjective.
\end{enumerate}

Finally, $(t)$-adic and $\mathfrak m_A$-adic topologies are equivalent, since
\[
\mathfrak m_A^{q+1}=t^q\mathfrak m_A
\subseteq t^qA\subseteq\mathfrak m_A^q.
\]
Taking inverse limits gives 
\[
\widehat A
\cong \varprojlim_{q \geq 1} A/t^qA
\cong \varprojlim_{q \geq 1} (\mathbb Q[\![t]\!]/t^q \mathbb Q[\![t]\!])[x_1, x_2]/(x_1,x_2)^2 \cong \mathbb Q[\![t]\!] \otimes \mathbb Q[x_1,x_2] / (x_1,x_2)^2.
\]
This completes the proof.
\end{proof}

Since $R\cong A[u,v]/(u,v)^2$ is finite over $A$, we obtain
$\widehat R\cong\widehat A[u,v]/(u,v)^2$.

\begin{lemma}\label{lem:E-injective}
There is an isomorphism of $R$-modules
\[
E_R(\mathbb Q)\cong \bigl(\mathbb Q[\![t]\!][t^{-1}]/\mathbb Q[\![t]\!]\bigr)\otimes_{\mathbb Q}
\bigl(\mathbb Q[x_1,x_2]/(x_1,x_2)^2\bigr)^*
\otimes_{\mathbb Q}\bigl(\mathbb Q[u,v]/(u,v)^2\bigr)^*.
\]
Here $(-)^*=\mathrm{Hom}_{\mathbb Q}(-,\mathbb Q)$, with the action on each dual factor given by $(af)(b)=f(ab)$.
\end{lemma}

\begin{proof}
Since $\mathbb Q[\![t]\!]$ is a discrete valuation ring, $\mathbb Q[\![t]\!][t^{-1}]/\mathbb Q[\![t]\!]$ is injective over $\mathbb Q[\![t]\!]$, and its socle $t^{-1}\mathbb Q[\![t]\!]/\mathbb Q[\![t]\!]\cong\mathbb Q$ is essential.
Thus $E_{\mathbb Q[\![t]\!]}(\mathbb Q)\cong \mathbb Q[\![t]\!][t^{-1}]/\mathbb Q[\![t]\!]$.
The preceding completion formula shows that $\widehat R$ is a free $\mathbb Q[\![t]\!]$-module of rank $9$.
Apply Proposition~\ref{prop:finite-local-injective} to $\mathbb Q[\![t]\!] \to \widehat {R}$
therefore gives
\[
E_{\widehat R}(\mathbb Q)\cong\mathrm{Hom}_{\mathbb Q[\![t]\!]}\bigl(\widehat R,\mathbb Q[\![t]\!][t^{-1}]/\mathbb Q[\![t]\!]\bigr).
\]
Finally, $E_R(\mathbb Q)\cong E_{\widehat R}(\mathbb Q)$ as $R$-modules by \cite[Theorem~18.6(iii)]{Matsumura1989}.
Since both polynomial quotient factors are finite-dimensional over $\mathbb Q$, tensor--Hom adjunction gives
\begin{align*}
\mathrm{Hom}_{\mathbb Q[\![t]\!]}\!\left(
\widehat R,\frac{\mathbb Q[\![t]\!][t^{-1}]}{\mathbb Q[\![t]\!]}\right)
&\cong\mathrm{Hom}_{\mathbb Q}\!\left(
\frac{\mathbb Q[x_1,x_2]}{(x_1,x_2)^2}\otimes_{\mathbb Q}
\frac{\mathbb Q[u,v]}{(u,v)^2},
\frac{\mathbb Q[\![t]\!][t^{-1}]}{\mathbb Q[\![t]\!]}\right)\\
&\cong\mathrm{Hom}_{\mathbb Q}\!\left(
\frac{\mathbb Q[x_1,x_2]}{(x_1,x_2)^2},
\mathrm{Hom}_{\mathbb Q}\!\left(
\frac{\mathbb Q[u,v]}{(u,v)^2},
\frac{\mathbb Q[\![t]\!][t^{-1}]}{\mathbb Q[\![t]\!]}\right)\right)\\
&\cong
\frac{\mathbb Q[\![t]\!][t^{-1}]}{\mathbb Q[\![t]\!]}
\otimes_{\mathbb Q}\left(\frac{\mathbb Q[x_1,x_2]}{(x_1,x_2)^2}\right)^*
\otimes_{\mathbb Q}\left(\frac{\mathbb Q[u,v]}{(u,v)^2}\right)^*.
\end{align*}
\end{proof}

\subsection{Construction of \texorpdfstring{$G_n \in \mathcal{GI}(R)$}{GnGIR}}
\label{The Gorenstein injective modules}

We construct $G_n\in\mathcal{GI}(R)$ such that
\[
t^n \cdot \mathrm{Ext}_R^1(C,G_n) = 0 , \qquad t^{n-1} \cdot \mathrm{Ext}_R^1(C,G_n) \neq 0. 
\]
The construction uses the dual of the cone of $t^n$ on a one-periodic complex $P$ adapted from the construction in the proof of \cite[Theorem~5.7]{KMOT}.
For simplicity, denote
\[
X=\mathbb Q[x_1,x_2]/(x_1,x_2)^2,\qquad
Y=\mathbb Q[u,v]/(u,v)^2,\qquad
V=\mathbb Q[\![t]\!][t^{-1}]/\mathbb Q[\![t]\!].
\]
Thus $R=A\otimes_{\mathbb Q}Y$ and
$\widehat R=\mathbb Q[\![t]\!]\otimes_{\mathbb Q}X\otimes_{\mathbb Q}Y$.
By Lemma~\ref{lem:E-injective} and its proof, we identify
\[
E_R(\mathbb Q)\cong E_{\widehat R}(\mathbb Q)
\cong V\otimes_{\mathbb Q}X^*\otimes_{\mathbb Q}Y^*.
\]

\begin{lemma}\label{lem:C-base-change}
For $C=\mathrm{Hom}_A(R,A)$, there are natural isomorphisms
\[
C\cong A\otimes_{\mathbb Q}Y^*,\qquad
\widehat R\otimes_R C
\cong\mathbb Q[\![t]\!]\otimes_{\mathbb Q}X\otimes_{\mathbb Q}Y^*.
\]
In particular, $C$ is finitely generated over $R$.
\end{lemma}

\begin{proof}
The first isomorphism follows since $Y$ is finite-dimensional;
the second follows by extension of scalars.
Thus $C$ is finitely generated over $A$ and so is over $R$.
\end{proof}

We now construct an exact complex $P$ which remains exact after tensoring with $E_R(\mathbb Q)$, but $C\otimes_RP$ has nonzero homology.

Let $u,v$ denote the residue classes of the variables in $R\cong A[u,v]/(u,v)^2$, and also their images in $\widehat R$.
Similarly, regard the residue classes $x_1,x_2\in X$ as elements of $\widehat R$ via the above presentation of $\widehat R$.
Thus $u^2=uv=v^2=0$ and $x_1^2=x_1x_2=x_2^2=0$.
Let $N=\coprod_{r,s\geq0}\widehat R e_{r,s}$ be the free $\widehat R$-module on the indicated basis.
For all $r,s\geq0$, define
\begin{align*}
    a : N \longrightarrow N , &\qquad e_{2r,s} \longmapsto u e_{r,s}, \quad e_{2r+1,s} \longmapsto v e_{r,s};\\
    b : N \longrightarrow N , &\qquad e_{r,s} \longmapsto x_1 e_{r,2s} + x_2 e_{r,2s+1}.
\end{align*}
Then set $d=\begin{psmallmatrix}a&b\\b&-a\end{psmallmatrix}$.
Note that $d^2 = 0$.
Then one has a $1$-periodic complex
\[
P=(\cdots\xlongrightarrow dN\oplus N\xlongrightarrow dN\oplus N\xlongrightarrow d\cdots).
\]

\begin{lemma}\label{lem:E-tensor}
Both $P$ and $E_R(\mathbb Q)\otimes_{\widehat R} P$ are acyclic complexes.
\end{lemma}

\begin{proof}
For $P$, define a $\mathbb Q[\![t]\!] \otimes X$-homomorphism $g : N \longrightarrow N$ (not an $\widehat R$-homomorphism)  via
\[ 
    e_{r,s} \longmapsto 0, \qquad ue_{r,s} \longmapsto e_{2r, s}, \qquad
    ve_{r,s} \longmapsto e_{2r+1,s}.
\]
Then $$\begin{psmallmatrix}g&0\\0&-g\end{psmallmatrix} \cdot \begin{psmallmatrix}a&b\\b&-a\end{psmallmatrix} + 
\begin{psmallmatrix}a&b\\b&-a\end{psmallmatrix} \cdot \begin{psmallmatrix}g&0\\0&-g\end{psmallmatrix} = 1_P.$$
It follows that $1_P$ is setwise homotopic to zero, and thus $P$ is acyclic by \cite{BianZhangZhang2009}.

Write $E_R(\mathbb Q)=X^*\otimes_{\mathbb Q}V\otimes_{\mathbb Q}Y^*$, and let $\eta_0,\eta_1,\eta_2$ be dual to $1,x_1,x_2$.
We identify
\[
E_R(\mathbb Q)\otimes_{\widehat R}N
\cong\coprod_{r,s\geq0}(X^*\otimes_{\mathbb Q}V\otimes_{\mathbb Q}Y^*)e_{r,s},
\]
and use $a,b,d$ also for the induced maps.
Define the $\mathbb Q$-linear map $k: E_R(\mathbb Q)\otimes_{\widehat R}N
\longrightarrow E_R(\mathbb Q)\otimes_{\widehat R}N$
via
\[
\begin{aligned}
(\eta_0\otimes w)e_{r,2s}
&\longmapsto (\eta_1\otimes w)e_{r,s},\\
(\eta_0\otimes w)e_{r,2s+1}
&\longmapsto (\eta_2\otimes w)e_{r,s},\\
(\eta_i\otimes w)e_{r,s}
&\longmapsto 0 \qquad (i=1,2),
\end{aligned}
\]
for $r,s\ge0$ and $w\in V\otimes_{\mathbb Q}Y^*$.
Then
\[
\begin{psmallmatrix}0&k\\k&0\end{psmallmatrix}
\begin{psmallmatrix}a&b\\b&-a\end{psmallmatrix}
+
\begin{psmallmatrix}a&b\\b&-a\end{psmallmatrix}
\begin{psmallmatrix}0&k\\k&0\end{psmallmatrix}
=
1_{\left(E_R(\mathbb Q)\otimes_{\widehat R}N\right)^{\oplus2}}.
\]
Thus $E_R(\mathbb Q)\otimes_{\widehat R}P$ is acyclic
\cite{BianZhangZhang2009}.
\end{proof}

Put $\mathbf C\langle n\rangle = \mathrm{Cone}(t^n : P \to P)$, and $\mathbf I \langle n\rangle = \mathrm{Hom}_{\widehat R}( \mathbf C \langle n \rangle , E_R(\mathbb Q))$.

\begin{proposition}\label{prop:compGI}
    For each $n \in \mathbb N$, $\mathbf I \langle n\rangle$ is a totally acyclic complex with injective terms.
\end{proposition}
\begin{proof}
Each term of $\mathbf I\langle n\rangle$ is a product of copies of $E_R(\mathbb Q)$, hence an injective $R$-module.
By Lemma~\ref{lem:E-tensor}, $P$ and thus $\mathbf C\langle n\rangle$ are acyclic.
Since $E_R(\mathbb Q) \cong E_{\widehat R}(\mathbb Q)$ is an injective $\widehat R$-module, it follows that $\mathbf I\langle n\rangle$ is acyclic.

Then we show that $\mathbf I\langle n\rangle$ is totally acyclic.
For an injective $R$-module $L$, one has
\[
\mathrm{Hom}_R(L,\mathbf I\langle n\rangle)
\cong\mathrm{Hom}_{\widehat R}
\bigl(L\otimes_R\mathbf C\langle n\rangle,E_R(\mathbb Q)\bigr).
\]
It suffices to show that $L\otimes_R\mathbf C\langle n\rangle$ is acyclic.
Here we need only consider $L=E_R(\mathbb Q)$ and $L=E_R(R/\mathfrak n)$.
For the former, $\widehat R\otimes_R E_R(\mathbb Q)\cong E_R(\mathbb Q)$ by \cite[proof of Theorem~18.6(iii)]{Matsumura1989}.
\begin{align*}
L\otimes_R\mathbf C\langle n\rangle
\cong E_R(\mathbb Q)\otimes_{\widehat R}\mathbf C\langle n\rangle
\cong\mathrm{Cone}\bigl(t^n:E_R(\mathbb Q)\otimes_{\widehat R}P
\longrightarrow E_R(\mathbb Q)\otimes_{\widehat R}P\bigr),
\end{align*}
which is acyclic by Lemma~\ref{lem:E-tensor}.
For the latter, $L\otimes_R\mathbf C\langle n\rangle$ is the cone of $t^n : E_R(R/\mathfrak n)\otimes_RP \to E_R(R/\mathfrak n)\otimes_RP$, which is contractible since $t$ is an automorphism of $E_R(R/\mathfrak n)$; see the proof of Lemma~\ref{lem:J-injective}.
\end{proof}

For $n\geq1$, put $G_n=Z^0(\mathbf I\langle n\rangle)$.
By Proposition~\ref{prop:compGI} each $G_n$ is Gorenstein injective.
Put $Q=C\otimes_R P$, viewed as a one-periodic complex of $\widehat R$-modules via the action on $P$.
One-periodicity identifies $H_i(Q)$ with $H_0(Q)$ for every $i\in\mathbb Z$.

\begin{lemma}\label{lem:Q-free}\label{lem:Q}
The homology $H_0(Q)$ is a nonzero free $\mathbb Q[\![t]\!]$-module.
\end{lemma}

\begin{proof}
By Lemma~\ref{lem:C-base-change}, for each $i\in\mathbb Z$,
\[
Q_i=C\otimes_RP_i
\cong\left(\coprod_{r,s\geq0}(C\otimes_R\widehat R)e_{r,s}\right)^{2}\cong\left(\coprod_{r,s\geq0}
(\mathbb Q[\![t]\!]\otimes_{\mathbb Q}X\otimes_{\mathbb Q}Y^*)e_{r,s}\right)^{2}.
\]
The differential matrices $\begin{psmallmatrix}a&b\\b&-a\end{psmallmatrix}$ of $Q$
have entries in $\{0,1,-1\}\subseteq\mathbb Q$.
Let $Q'$ be the complex of $\mathbb Q$-vector spaces with terms $Q'_i=\left(\coprod_{r,s\geq0}(X\otimes_{\mathbb Q}Y^*)e_{r,s}\right)^{2}$ and the same differential matrices.
Then $Q\cong\mathbb Q[\![t]\!]\otimes_{\mathbb Q}Q'$ as complexes of $\mathbb Q[\![t]\!]$-modules.
Since $\mathbb Q$ is a field, the functor $\mathbb Q[\![t]\!]\otimes_{\mathbb Q}-$ is exact and therefore commutes with homology.
Consequently,
\[
H_0(Q)\cong\mathbb Q[\![t]\!]\otimes_{\mathbb Q}H_0(Q').
\]
A $\mathbb Q$-basis of $H_0(Q')$ indexed by $\Lambda$ gives $H_0(Q)\cong\coprod_{\lambda\in\Lambda}\mathbb Q[\![t]\!]$.
Hence $H_0(Q)$ is free.

To see that $H_0(Q) \neq 0$, consider $z=\bigl((1\otimes x_2\otimes\epsilon_v)e_{0,0},0\bigr)\in Q_0$.
By construction $z$ is a cycle.
Define a $\mathbb Q[\![t]\!]$-linear functional
$\lambda:Q_0\to\mathbb Q[\![t]\!]$ on the displayed basis by
\[
( (1\otimes\xi\otimes\epsilon)e_{r,s},0)
\longmapsto
\begin{cases}
1,&\xi=x_2,\ \epsilon=\epsilon_v,\ r=s=0,\\
0,&\text{otherwise},
\end{cases}
\qquad 
(0,(1\otimes\xi\otimes\epsilon)e_{r,s})\longmapsto 0.
\]
Thus $\lambda$ reads exactly the coefficient of
$(1\otimes x_2\otimes\epsilon_v)e_{0,0}$ in the first component, and $\lambda(z)=1$.

Then we show that $\lambda$ vanishes on boundaries, write an element of $Q_1$ as
$(p,q)$, where $p=\sum_{r,s}p_{r,s}e_{r,s}$ and
$q=\sum_{r,s}q_{r,s}e_{r,s}$, with finitely many nonzero coefficients in
$\mathbb Q[\![t]\!]\otimes_{\mathbb Q}X\otimes_{\mathbb Q}Y^*$.
Then $d(p,q)=(ap+bq,bp-aq)$.
Only the coefficient at $e_{0,0}$ in $ap+bq$ contributes to $\lambda$.
By construction, this coefficient is $u p_{0,0}+v p_{1,0}+x_1q_{0,0}$.
Now $u \cdot Y^*+v \cdot Y^*=\mathbb Q\epsilon_0$, so the first two terms have zero
$\epsilon_v$-coefficient.
Also $x_1 \cdot X=\mathbb Qx_1$, so the last term has zero $x_2$-coefficient.
Consequently,
\[
\lambda d(p,q)=0\quad\text{for every }(p,q)\in Q_1,
\qquad\lambda(z)=1.
\]
Hence $z\notin\mathrm{im}\,d_1$.
Since $z$ is a cycle, its class in $H_0(Q)$ is nonzero.
\end{proof}

\begin{lemma}\label{lem:Tn}
For every $n\geq1$, one has $t^n\mathrm{Ext}_R^1(C,G_n)=0$ and $t^{n-1}\mathrm{Ext}_R^1(C,G_n)\ne0$.
\end{lemma}

\begin{proof}
The adjunction gives
\[
\mathrm{Hom}_R(C,\mathbf I\langle n\rangle)
=\mathrm{Hom}_R(C,\mathrm{Hom}_{\widehat R}( \mathbf C \langle n \rangle , E_R(\mathbb Q)))
\cong\mathrm{Hom}_{\widehat R}
\bigl(C\otimes_R\mathbf C\langle n\rangle,E_R(\mathbb Q)\bigr).
\]
Since $C\otimes_R\mathbf C\langle n\rangle
\cong\mathrm{Cone}(t^n:Q\to Q)$ taking homology groups of the distinguished triangle $Q \xrightarrow{t^n}Q \to C \otimes _R \mathbf C\langle n \rangle \to Q[1]$ gives a long exact sequence
\[
\cdots \longrightarrow H_{0}(Q) \xlongrightarrow{t^n} H_{0}(Q) \longrightarrow H_0(C \otimes _R \mathbf C\langle n \rangle) \longrightarrow H_{-1}(Q) \longrightarrow \cdots .
\]
By Lemma~\ref{lem:Q-free}, $t^n$ is injective on $H_i(Q)$.
Hence one has isomorphisms of $\widehat R$-modules
\[
H_i(C\otimes_R\mathbf C\langle n\rangle)
\cong H_0(Q)/t^nH_0(Q)\qquad(i\in\mathbb Z).
\]
Since $E_R(\mathbb Q)\cong E_{\widehat R}(\mathbb Q)$, one has the isomorphisms
\[
\mathrm{Ext}_R^1(C,G_n)\cong H^1\mathrm{Hom}_R(C,\mathbf I\langle n\rangle) \cong\mathrm{Hom}_{\widehat R}(H_0(Q)/t^nH_0(Q),E_R(\mathbb Q)).
\]
By Lemma~\ref{lem:Q}, choose a nonempty set $\Lambda$ such that $H_0(Q)\cong\coprod_{\lambda\in\Lambda}\mathbb Q[\![t]\!]$ is an isomorphism of $\mathbb Q[\![t]\!]$-modules.
Hence one has
\begin{align*}
    \mathrm{Ext}_R^1(C,G_n)
&\cong\mathrm{Hom}_{\mathbb Q[\![t]\!]}
\bigl(H_0(Q)/t^nH_0(Q),V\bigr)
\cong\prod_{\lambda\in\Lambda}
\mathrm{Hom}_{\mathbb Q[\![t]\!]}
\bigl(\mathbb Q[\![t]\!]/t^n\mathbb Q[\![t]\!],V\bigr)\\
&\cong \prod_{\lambda\in\Lambda}
t^{-n}\mathbb Q[\![t]\!]/\mathbb Q[\![t]\!]
\cong \prod_{\lambda\in\Lambda}
\mathbb Q[\![t]\!]/t^n\mathbb Q[\![t]\!].
\end{align*}
This product is annihilated by $t^n$ but not by $t^{n-1}$, proving the assertions.
\end{proof}

\section{Zariski locality and recollements}\label{sec:locality}
We prove that the covering property for Gorenstein injective quasi-coherent sheaves is local with respect to open covers of Noetherian schemes.

By Theorem~\ref{thm:equivalence}, it suffices to study coproduct closure of totally acyclic complexes of injectives.
Direct image along open immersions allows us to pass from a scheme to its open subschemes, while a Mayer--Vietoris argument gives the converse.

Subsection~\ref{c:sec:scheme} establishes the open-immersion adjunctions on homotopy categories of injectives, and Subsection~\ref{c:sec:acyclicity} studies their behaviour on totally acyclic complexes.
These results yield the proof of Theorem~C in Subsection~\ref{c:sec:proof}.
Finally, under the covering hypothesis, Subsection~\ref{c:sec:restriction-recollements} shows that restriction preserves total acyclicity and relates the resulting adjunctions to recollements.

\subsection{Open immersions and adjoints}\label{c:sec:scheme}

For a scheme $X$, write $\mathrm{QCoh}(X)$ for the category of quasi-coherent $\mathcal O_X$-modules, and $\Gamma(V,\mathcal F)$ for sections over $V$.
For $X=\mathrm{Spec}\,R$, the associated sheaf functor $M\mapsto\widetilde M$ and global sections give mutually inverse equivalences \cite[Corollaire~(1.3.8) and Theorem~(1.4.1)]{EGAI}:
\[
 \mathrm{Mod}_R\simeq\mathrm{QCoh}(\mathrm{Spec}\,R).
\]
On an affine chart, $\mathcal O_{X,\mathfrak p}\cong R_{\mathfrak p}$ and $\widetilde M_{\mathfrak p}\cong M_{\mathfrak p}$. 

For every Noetherian scheme $U$, put
$\mathcal A_U = \mathrm{QCoh}(U)$, which is a locally Noetherian Grothendieck category.
Put $\mathcal T_U = K(\mathrm{Inj}(\mathcal A_U))$ and the canonical functor $q_U : \mathcal T_U \to \mathcal D_U$.
Here $\mathcal D_U = D(\mathcal A_U)$ is the unbounded derived category.
The category $\mathcal T_U$ is compactly generated by \cite[Proposition~2.3]{Krause2005}.

For an open immersion $j:U\hookrightarrow X$, write $j^*$ for restriction and $j_*$ for direct image of sheaves.
The symbols $R^qj_*$ and $\mathbf Rj_*$ denote its right derived functors, and $\mathbf R\Gamma$ denotes derived global sections.

Let $u : U \hookrightarrow X$ be an open immersion of Noetherian schemes; in particular each open subscheme of $X$ is again Noetherian.
Since $u$ is quasi-compact, $u_\ast$ preserves quasi-coherence by \cite[Proposition~(6.7.1)]{EGAI}.
Then it induces the direct image functor:
\[
u_\ast : \mathcal A_U = \mathrm{QCoh}(U)\to \mathrm{QCoh}(X) = \mathcal A_X.
\]
Since restriction preserves quasi-coherence, the inverse image $u^\ast : \mathcal A_X \to \mathcal A_U$ is its well-defined left adjoint.
Since $U$ is Noetherian, taking sections on its open subsets
commutes with filtered colimits \cite[Chapter~III, Proposition~2.9]{Hartshorne1977}.
Thus $u_\ast$ preserves coproducts.

\begin{lemma}\label{c:open-adjoints}
Then adjunction $u^\ast \dashv u_\ast$ of quasicoherent sheaves degreewise induces the adjoint pair  between $\mathcal T_U$ and $\mathcal T_X$, which is again denoted by $u^\ast \dashv u_\ast$.
Moreover, $u_\ast : \mathcal T_U \to \mathcal T_X$ admits a right adjoint, which is denoted by $u^\times$.
The functors are
\[
\xymatrix@C=1.4cm{
U \ \ar@{^{(}->}[r]^u & X
}
\qquad
\xymatrix@C=2cm{
\mathcal A_U \ar[r]|{u_\ast}
& \mathcal A_X \ar@<-1.2ex>[l]_{u^\ast}
}
\qquad
\xymatrix@C=2cm{
\mathcal T_U \ar[r]|{u_\ast}
& \mathcal T_X
  \ar@<-1.2ex>[l]_{u^\ast}
  \ar@<1.2ex>[l]^{u^\times}
}
\]
\end{lemma}

\begin{proof}
Since $u_\ast : \mathcal A_U \to \mathcal A_X$ admits an exact left adjoint, it preserves injective objects.
Hence $u_\ast : \mathcal T_U \to \mathcal T_X$ is well-defined.
By construction, $u_\ast$ preserves suspensions and conesm and thus is a triangle functor.
By \cite[Theorem~A.3]{EfimovPositselski}, $u^\ast$ also preserves injective objects.
Again the termwise application induces a triangle functor
$u^\ast:\mathcal T_X\to\mathcal T_U$.

The unit and counit of the adjunction $u^\ast\dashv u_\ast$ descend to the homotopy categories, where they still satisfy the triangle identities.
Hence $u^\ast\dashv u_\ast$ between $\mathcal T_X$ and $\mathcal T_U$.

Since $\mathcal A_U$ is a locally Noetherian Grothendieck category,
$\mathcal T_U$ is compactly generated by
\cite[Proposition~2.3]{Krause2005}.
Coproducts in $\mathcal T_U$ and $\mathcal T_X$ are computed degreewise.
Since $u_\ast:\mathcal A_U\to\mathcal A_X$ preserves coproducts,
so does the induced triangle functor
$u_\ast:\mathcal T_U\to\mathcal T_X$.
Thus $u_\ast : \mathcal T_U\to\mathcal T_X$ admits a right adjoint by Brown representability \cite[Proposition~3.3\textup{(1)}]{Krause2005}.
\end{proof}

\begin{remark}
It is not necessarily that $u_\ast : \mathcal A_U \to \mathcal A_X$ admits a right adjoint.
For instance, consider $X = \operatorname{spec} k[x,y]$ and $U = X \setminus \{(x,y)\}$.
    Hence we denote the right adjoint of $u_\ast : \mathcal T_U \to \mathcal T_X$ by $u^\times$ rather then $u^!$.
\end{remark}

\subsection{The restricted adjunction on totally acyclic complexes}\label{c:sec:acyclicity}

Denote the triangulated subcategory $\mathcal S_U=K_{\mathrm{tac}}(\mathrm{Inj}\,\mathcal A_U)$.
By Theorem~\ref{thm:equivalence}, $\mathcal{GI}(\mathcal A_U)$ is covering if and only if $\mathcal S_U$ is localising in $\mathcal T_U$.
We study whether the adjoint triple $u^\ast \dashv u_\ast \dashv u^\times$ decents to $\mathcal S$.

\begin{lemma}\label{c:pushforward-acyclicity}
The functor $u_\ast:\mathcal T_U\to\mathcal T_X$ preserves and reflects total acyclicity, i.e., for every $I\in\mathcal T_U$, $u_\ast I$ is totally acyclic if and only if $I$ is totally acyclic.
\end{lemma}

\begin{proof}
Let $I$ be an acyclic complex of injective objects on $U$.
To see that $u_\ast I$ is acyclic, it suffices to show for each affine open subscheme $V \subseteq X$ that $(u_\ast I)|_V$ is acyclic.
Write $V\cong\mathrm{Spec}(R)$. Then $\Gamma(V,-):\mathrm{QCoh}(V)\xrightarrow{\sim}\mathrm{Mod}_R$ is an equivalence by \cite[Chapter~II, Corollary~5.5]{Hartshorne1977}.
Thus
\begin{align}
\mathcal H^n\bigl((u_\ast I)|_V\bigr)=0 & \iff \Gamma\!\left(V,
\mathcal H^n\bigl((u_\ast I)|_V\bigr)\right)=0 
\iff H^n\!\left(\Gamma\bigl(V,(u_\ast I)|_V\bigr)\right)=0\notag\\
&\iff H^n\bigl(\Gamma(U\cap V,I)\bigr)=0
\iff H^1(U\cap V,Z^{n-1}I)=0 \notag\\
&\iff \cdots 
\iff H^q(U\cap V,Z^{n-q}I)=0 \quad (\exists q \geq 1).\label{eq:C-local-acyclicity}
\end{align}
Now fix an affine open $V=\mathrm{Spec}\,R\subseteq X$.
If $U \cap V=\varnothing$, there is nothing to prove on $V$.
Otherwise, $U \cap V$ is quasi-compact, so choose $f_1,\ldots,f_m\in R$ such that
\[
D_V(f_i)\subseteq U \cap V,\qquad W=\bigcup_{i=1}^m D_V(f_i).
\]
For $1\leq i_0<\cdots<i_p\leq m$, write
\[
(U \cap V)_{i_0\cdots i_p}
=D_V(f_{i_0})\cap\cdots\cap D_V(f_{i_p})
=D_V(f_{i_0}\cdots f_{i_p}).
\]
For $\mathcal F\in\mathrm{QCoh}(U \cap V)$, the \v Cech complex has terms
\[
C^p(\mathcal F)=
\prod_{1\leq i_0<\cdots<i_p\leq m}
\Gamma((U \cap V)_{i_0\cdots i_p},\mathcal F)
\qquad(0\leq p\leq m-1),
\]
and zero terms otherwise. Its differential is
\[
(\delta^p s)_{i_0\cdots i_{p+1}}
=\sum_{a=0}^{p+1}(-1)^a
s_{i_0\cdots\widehat{i_a}\cdots i_{p+1}}
\big|_{(U \cap V)_{i_0\cdots i_{p+1}}}.
\]
Thus it is the finite complex
\begin{multline*}
0\longrightarrow\prod_{i=1}^m\Gamma(D_V(f_i),\mathcal F)
\longrightarrow\prod_{1\leq i<j\leq m}\Gamma(D_V(f_if_j),\mathcal F) \longrightarrow \cdots \\
\cdots \longrightarrow \prod_{i=1}^{m}
\Gamma\!\bigg(D_V\!\big(\prod_{j\ne i}f_j\big),\mathcal F\bigg)\longrightarrow
\Gamma(D_V(f_1\cdots f_m),\mathcal F)\longrightarrow0,
\end{multline*}
with the first nonzero term in degree zero.
All intersections $W_{i_0\cdots i_p}$ are affine, and $W$ is separated as an open subscheme of the affine scheme $V$.
By \cite[Chapter~III, Theorem~4.5]{Hartshorne1977}, this complex computes sheaf cohomology.
Hence $H^q(U \cap V, \mathcal F)=0$ for each $q \geq m$.
Put $\mathcal F = (Z^{n-m} I)|_{U \cap V}$.
Then $H^m(U \cap V, Z^{n-m}I)=0$, satisfying \eqref{eq:C-local-acyclicity}.

To see that $u_\ast I$ is totally acyclic, for any injective object $E\in\mathcal A_X$, the object $u^\ast E$ is injective in $\mathcal A_U$, and adjunction gives an isomorphism of complexes
\[
\mathrm{Hom}_{\mathcal A_X}(E,u_\ast I)
\cong \mathrm{Hom}_{\mathcal A_U}(u^\ast E,I).
\]
The latter is acyclic since $I$ is totally acyclic.
Thus $u_\ast I$ is totally acyclic.

Conversely, suppose that $u_\ast I$ is totally acyclic.
Since $u^\ast$ is exact and $u^\ast u_\ast I\cong I$, the complex $I$ is acyclic.
For every injective object $E\in\mathcal A_U$, $u_\ast E$ is injective in $\mathcal A_X$, since $u_\ast$ admits an exact left adjoint.
Full faithfulness of $u_\ast$ gives an isomorphism of complexes
\[
\mathrm{Hom}_{\mathcal A_U}(E,I)
\cong\mathrm{Hom}_{\mathcal A_X}(u_\ast E,u_\ast I).
\]
Since $\mathrm{Hom}_{\mathcal A_X}(u_\ast E,u_\ast I)$ is acyclic, $I$ is also totally acyclic.
\end{proof}

\begin{lemma}\label{c:right-adjoint-acyclicity}
The functor $u^\times : \mathcal T_X \to \mathcal T_U$ preserves total acyclicity.
\end{lemma}

\begin{proof}
We show that $u^\times I$ is acyclic whenever $I\in\mathcal T_X$ is acyclic.
We first treat affine $X$, then the case $X=U\cup V$ with $V$ affine.
The general case follows by adjoining the members of a finite affine cover one at a time.

Suppose first that $X=\mathrm{Spec}(R)$.
Choose a finite principal cover $U=\bigcup_{i=1}^mD(f_i)$ as in the proof of Lemma~\ref{c:pushforward-acyclicity}.
Let $\mathcal O _U$ be the structral sheaf.
Fix a bounded-below injective resolution $\mathcal O_U\to\iota\mathcal O_U$.
For every $J\in\mathcal T_U$, one has
\begin{align*}
    H^n\Gamma(U,u^\times I) & \cong H^n \mathrm{Hom}_{\mathcal A_U}(\mathcal O_U, u^\times I) \cong \mathrm{Hom}_{K(\mathcal A_U)}(\mathcal O_U, u^\times I[n])\\
    & \cong \mathrm{Hom}_{\mathcal T_U} (\iota \mathcal O_U, u^\times I[n]) \cong \mathrm{Hom}_{\mathcal T_X} (u_\ast \iota \mathcal O_U, I[n]) .
\end{align*}
Under the equivalence $\Gamma(X,-) : \mathrm{QCoh}(X)\simeq\mathrm{Mod}_R$, the bounded-below complex $u_\ast\iota\mathcal O_U$ is an injective resolution of the finite \v Cech complex (with first term in degree zero)
\[
 C^\bullet=
 \left[
 \prod_iR_{f_i}\longrightarrow
 \prod_{i<j}R_{f_if_j}\longrightarrow\cdots\longrightarrow
 R_{f_1\cdots f_m}
 \right].
\]
Hence $\mathrm{Hom}_{\mathcal T_X}(u_\ast\iota\mathcal O_U,I[n])
 \cong H^n\mathrm{Hom}_R(C^\bullet,I)$.
For each $f\in R$, $R_f$ has free resolution
\[
 0\longrightarrow\coprod_{k\geq0}R
 \xrightarrow{\ e_k\mapsto e_k-fe_{k+1}\ }
 \coprod_{k\geq0}R\longrightarrow R_f\longrightarrow0.
\]
Hence $\mathrm{pd}_R R_f\leq1$.
It follows from dimension shift that
\[
 H^n\mathrm{Hom}_R(R_f,I)
 \cong\mathrm{Ext}_R^1(R_f,Z^{n-1}I)
 \cong\mathrm{Ext}_R^2(R_f,Z^{n-2}I)=0.
\]
Since $C^\bullet$ is bounded, $\mathrm{Hom}_R(C^\bullet,I)$ is acyclic by induction on length.
The previous isomorphisms shows that $\Gamma(U,u^\times I)$ is acyclic.
Since $U \subseteq X$ is open, $\Gamma(U,u^\times I)=\Gamma(X,u_\ast u^\times I)$.
Note that and $\Gamma(X,-)$ is an equivalence, $u_\ast u^\times I$ is acyclic.
Since $u^\ast$ is exact and $u_\ast$ is fully faithful, $u^\times I \cong u^\ast u_\ast u^\times I$ is acyclic.

Suppose next that $X=U\cup V$ with $V$ affine, and denote the inclusions by
\[
\xymatrix@C=2cm@R=.0cm{
& \ U\  \ar[dr]^-u & \\
U\cap V \
\ar[ur]^-{u'}
\ar[rr]^-z
\ar[dr]_-{v'}
&& \ U\cup V\\
& \ V \  \ar[ur]_-v &
}
\]
For an acyclic $I\in\mathcal T_X$, the Mayer--Vietoris sequence $ 0\to I\to u_\ast u^\ast I\oplus v_\ast v^\ast I \to z_\ast z^\ast I\to 0$ splits in each degree, since its terms are injective.
It gives a distinguished triangle
\[
 I\longrightarrow u_\ast u^\ast I\oplus v_\ast v^\ast I
 \longrightarrow z_\ast z^\ast I\longrightarrow I[1]
\]
in $\mathcal T_X$.
Since $u_\ast$ is fully faithful, $u^\times u_\ast\cong1$.
Open base change gives $v^\ast u_\ast\cong v'_\ast u'^\ast$.
Taking the right adjoints gives $u^\times v_\ast\cong u'_\ast v'^\times$.
Together with $z_\ast=u_\ast u'_\ast$, these yield
\[
 u^\times u_\ast u^\ast I\cong u^\ast I,\qquad
 u^\times v_\ast v^\ast I\cong u'_\ast v'^\times v^\ast I,\qquad
 u^\times z_\ast z^\ast I\cong u'_\ast z^\ast I.
\]
All three complexes are acyclic: 
restrictions $u^\ast$, $v^\ast$ and $z^\ast$ are exact; 
$v'^\times$ preserves acyclicity by the affine case;
$u'_\ast$ preserves acyclicity by the proof of Lemma~\ref{c:pushforward-acyclicity}.
Applying $u^\times$ to the displayed triangle therefore shows that $u^\times I$ is acyclic.

To see that $u^\times I$ is totally acyclic when $I$ is totally acyclic, regard an injective object $E\in\mathcal A_U$ as a complex concentrated in degree zero.
For every $n\in\mathbb Z$,
\[
 \mathrm{Hom}_{\mathcal T_U}(E,u^\times I[n])
 \cong\mathrm{Hom}_{\mathcal T_X}(u_\ast E,I[n])=0,
\]
since $u_\ast E$ is injective.
These are the cohomology groups of $\mathrm{Hom}_{\mathcal A_U}(E,u^\times I)$.
Together with the acyclicity already proved, this gives total acyclicity.
\end{proof}

\begin{proposition}\label{c:tac-adjunction}
The adjoint triple $u^\ast \dashv u_\ast \dashv u^\times$ induces the adjoint pair $\widetilde{u_\ast}\dashv \widetilde{u^\times}$
\[
\xymatrix@C=3.5cm@R=.7cm{
\mathcal T_U
\ar[r]|{u_\ast}
& \mathcal T_X
\ar@<-1.2ex>[l]_{u^\ast}
\ar@<1.2ex>[l]^{u^\times}
\\
\mathcal S_U
\ar@{ >->}[u]^{i_U{}_{\ast}}
\ar@{.>}[r]|{\widetilde{u_\ast}}
& \mathcal S_X
\ar@{ >->}[u]_{i_X{}_{\ast}}
\ar@{.>}@<1.2ex>[l]^{\widetilde{u^\times}}
}
\]
with $u_\ast i_U{}_\ast = i_X{}_\ast \widetilde{u_\ast}$ and $u^\times i_X{}_\ast = i_U{}_\ast \widetilde{u^\times}$.
\end{proposition}

\begin{proof}
By Lemmas~\ref{c:pushforward-acyclicity} and~\ref{c:right-adjoint-acyclicity}, $u_\ast$ and $u^\times$ restrict to acyclic complexes.
For arbitrary $A\in\mathcal S_U$ and $B\in\mathcal S_X$ the adjunction gives natural isomorphisms
\[
\mathrm{Hom}_{\mathcal S_X}(\widetilde{u_\ast}A,B)
\cong\mathrm{Hom}_{\mathcal T_X}(u_\ast A,B)
\cong\mathrm{Hom}_{\mathcal T_U}(A,u^\times B)
\cong\mathrm{Hom}_{\mathcal S_U}(A,\widetilde{u^\times}B).
\]
Thus $\widetilde{u_\ast}\dashv\widetilde{u^\times}$, and compatibility with the inclusions holds by definition.
\end{proof}

\begin{remark}\label{rmk:uleft}
By \cite[Lemma~7.3]{Krause2005}, the inclusion
$i_U{}_{\ast}:\mathcal S_U\to\mathcal T_U$
admits a left adjoint $i_U{}^\ast$.
Hence $\widetilde{u_\ast}$ has a left adjoint $i_U{}^\ast u^\ast i_X{}_\ast$.
This adjoint coincides with the restriction of $u^\ast$ if and only if $u^\ast$ preserves totally acyclic complexes; this does not hold in general.
\end{remark}

\begin{example}
This example shows that $u^\ast$ need not preserve totally acyclic complexes.
Retain Example~\ref{thm:main}, take $X = \mathrm{Spec}(R) = \{\mathfrak m_R, \mathfrak n\}$ and $U = D(t) = \{\mathfrak n\} \cong \mathrm{Spec}(R[t^{-1}])$.
Then $u_\ast : \mathrm{Mod}_{R[t^{-1}]} \to \mathrm{Mod}_R$ is the forgetful functor with left adjoint
\[
u^\ast : \mathrm{Mod}_R \longrightarrow \mathrm{Mod}_{R[t^{-1}]}, \qquad M \longmapsto M[t^{-1}] = R[t^{-1}] \otimes _R M.
\]
Then take $\prod_{n \geq 1} \mathbf I\langle n \rangle \in \mathcal S_X$, where $\prod_{n \geq 1} G_n = \prod_{n \geq 1} Z^0 (\mathbf I\langle n \rangle) \cong  Z^0 (\prod_{n \geq 1} \mathbf I\langle n \rangle) \in \mathcal {GI}(R)$.

Note that $u^\ast$ is the localisation with respect to the multiplication system $\{1,t,t^2, \ldots\}$.
Since $R$ is Noetherian and $C$ is finitely generated, one has caninonical isomorphisms
\[
\mathrm{Ext}^1_{u^\ast R} \left(u^\ast C, u^\ast\prod_{n \geq 1} G_n \right) \cong 
u^\ast\mathrm{Ext}^1_{R} \left( C, \prod_{n \geq 1} G_n \right) \cong u^\ast \left(\prod_{n\geq1}\mathrm{Ext}_R^1(C,G_n)\right)
\]
By Lemma~\ref{lem:Tn}, $t^{n} \mathrm{Ext}_R^1(C,G_n) =0$ whereas $t^{n-1} \mathrm{Ext}_R^1(C,G_n) \neq 0$. 
Hence the above is non-zero.
In particular, since $u^\ast C$ is injective, $u^\ast\prod_{n \geq 1} G_n$ is not Gorenstein injective.
It follows that $u^\ast\left(\prod_{n\geq1}\mathbf I\langle n\rangle\right)$ is not totally acyclic.

Therefore $u^\ast$ need not preserve totoally acyclic complexes.
\end{example}

\subsection{Proof of \texorpdfstring{Theorem~\ref{thm:locality}}{Theorem C}}\label{c:sec:proof}
Let $\mathcal A_U$ denote $\mathrm{QCoh}(U)$ for each Noetherian schemes. 
Let $X$ be a Noetherian scheme.
Theorem~\ref{thm:locality} states that the following conditions are equivalent.
\begin{enumerate}
\item For every open subscheme $U\subseteq X$, $\mathcal{GI}(\mathcal A_U)$ forms a covering class.
\item $\mathcal{GI}(\mathcal A_X)$ forms a covering class.
\item There is an open cover $X=\bigcup_{i\in I}U_i$ such that $\mathcal{GI}(\mathcal A_{U_i})$ forms a covering class for every $i\in I$.
\item There is an affine open cover $X=\bigcup_{i\in I}\mathrm{Spec}\,R_i$ such that $\mathcal{GI}(R_i)$ forms a covering class for every $i\in I$.
\item For every open cover $X=\bigcup_{i\in I}U_i$, $\mathcal{GI}(\mathcal A_{U_i})$ forms a covering class for every $i\in I$.
\item For every affine open cover $X=\bigcup_{i\in I}\mathrm{Spec}\,R_i$, $\mathcal{GI}(R_i)$ forms a covering class for every $i\in I$.
\end{enumerate}

\begin{proof}[Proof of Theorem~\ref{thm:locality}]
The implications $\textup{(1)}\Longrightarrow\textup{(5)}\Longrightarrow\textup{(6)}\Longrightarrow\textup{(4)}\Longrightarrow\textup{(3)}$ are immediate.

\noindent\textit{$(2)\Longrightarrow(1)$.}
Suppose $\mathcal{GI}(\mathcal A_X)$ is covering and let $u:U\hookrightarrow X$ be open.
By Theorem~\ref{thm:equivalence}, $\mathcal S_X$ is closed under coproducts. 
Let $(I_\alpha)_{\alpha\in\Lambda}$ be a family in $\mathcal S_U$.
By Lemma~\ref{c:pushforward-acyclicity}, each $u_\ast I_\alpha$ is totally acyclic.
By Lemma~\ref{c:open-adjoints}, $u_\ast$ preserves coproducts, 
thus $u_\ast\!\left(\coprod_{\alpha\in\Lambda} I_\alpha\right) \cong\coprod_{\alpha\in\Lambda} u_\ast I_\alpha$.
Again by Lemma~\ref{c:pushforward-acyclicity}, $\coprod_{\alpha\in\Lambda} I_\alpha\in \mathcal S_U$.
Hence $\mathcal{GI}(\mathcal A_U)$ is covering by Theorem~\ref{thm:equivalence}.

\noindent\textit{$(3)\Longrightarrow(2)$.}
Quasi-compactness reduces the given cover to a finite subcover, so it suffices to treat $X=U\cup V$.
Write $u:U\hookrightarrow X$, $v:V\hookrightarrow X$, $u':U\cap V\hookrightarrow U$, $v':U\cap V\hookrightarrow V$ and $z:U\cap V\hookrightarrow X$ for the inclusions.
By $\textup{(2)}\Longrightarrow\textup{(1)}$, the covering property also holds on $U\cap V$.
Denote the counit of $w_\ast\dashv w^\times$ by $\varepsilon^w$.

For every $I\in\mathcal T_X$, one has a distinguished triangle
\begin{equation}
z_\ast z^\times I
\xlongrightarrow{\begin{psmallmatrix}
(u_\ast\varepsilon^{u'}u^\times)_I\\
-(v_\ast\varepsilon^{v'}v^\times)_I
\end{psmallmatrix}} \label{eq:C-cone-triangle}
u_\ast u^\times I\oplus v_\ast v^\times I
\xlongrightarrow{\iota}
\mathrm{Cone}\begin{psmallmatrix}
(u_\ast\varepsilon^{u'}u^\times)_I\\
-(v_\ast\varepsilon^{v'}v^\times)_I
\end{psmallmatrix}
\longrightarrow (z_\ast z^\times I)[1].
\end{equation}
Then we claim that $\mathrm{Cone}\begin{psmallmatrix}(u_\ast\varepsilon^{u'}u^\times)_I\\-(v_\ast\varepsilon^{v'}v^\times)_I\end{psmallmatrix}
 \cong I$.

Apply $u^\times$ to $\begin{psmallmatrix}
(u_\ast\varepsilon^{u'}u^\times)_I\\
-(v_\ast\varepsilon^{v'}v^\times)_I
\end{psmallmatrix}: z_\ast z^\times I
\to
u_\ast u^\times I\oplus v_\ast v^\times I$, one has the natural isomorphisms $u^\times z_\ast z^\times \cong u^\times u_\ast u'{}_\ast z^\times \cong u'{}_\ast z^\times $, $u^\times u_\ast u^\times = u^\times$, $u^\times v_\ast v^\times \cong u'_\ast v'^\times v^\times \cong u'_\ast z^\times$, $u^\times (u_\ast\varepsilon^{u'}u^\times) \cong \varepsilon ^{u'} u^\times$, and $u^\times(v_\ast\varepsilon^{v'}v^\times)\cong u'{}_\ast(v'^\times\varepsilon^{v'})v^\times$.
Since $v'{}_\ast$ is fully faithful, the unit of $v'{}_\ast \dashv v'^\times$ is invertible and thus $v'^\times \varepsilon ^{v'}$ is a natural isomorphism by triangle identity.
Applying $u^\times$ to \eqref{eq:C-cone-triangle} above gives the distinguished triangle
\[
u'_\ast z^\times I
\xlongrightarrow{\begin{psmallmatrix}(\varepsilon^{u'}u^\times)_I\\1\end{psmallmatrix}}
u^\times I\oplus u'_\ast z^\times I
\xlongrightarrow{\begin{psmallmatrix}1&-(\varepsilon^{u'}u^\times)_I\end{psmallmatrix}}
u^\times I\xlongrightarrow{0}(u'_\ast z^\times I)[1]
\]
with an isomorphism $\varphi : u^\times I \cong u^\times \mathrm{Cone}\begin{psmallmatrix}
(u_\ast\varepsilon^{u'}u^\times)_I\\
-(v_\ast\varepsilon^{v'}v^\times)_I
\end{psmallmatrix}$ 

Since $(\varepsilon^u)_I\circ(u_\ast\varepsilon^{u'}u^\times)_I =(\varepsilon^z)_I = (\varepsilon^v)_I\circ(v_\ast\varepsilon^{v'}v^\times)_I$,
one has
\[
\begin{psmallmatrix}(\varepsilon^u)_I&(\varepsilon^v)_I\end{psmallmatrix}
\circ
\begin{psmallmatrix}
(u_\ast\varepsilon^{u'}u^\times)_I\\
-(v_\ast\varepsilon^{v'}v^\times)_I
\end{psmallmatrix}=0, \qquad \begin{psmallmatrix}(\varepsilon^u)_I&(\varepsilon^v)_I\end{psmallmatrix} : u_\ast u^\times I\oplus v_\ast v^\times I \longrightarrow I.
\]
Hence $\begin{psmallmatrix}(\varepsilon^u)_I&(\varepsilon^v)_I\end{psmallmatrix}$ factors through $\iota$ via
\[
h:\mathrm{Cone}\begin{psmallmatrix}
(u_\ast\varepsilon^{u'}u^\times)_I\\
-(v_\ast\varepsilon^{v'}v^\times)_I
\end{psmallmatrix}\longrightarrow I,
\qquad
h\circ\iota=
\begin{psmallmatrix}(\varepsilon^u)_I&(\varepsilon^v)_I\end{psmallmatrix},
\]
After applying $u^\times$, the following diagram commutes
\[
\xymatrix@C=2cm@R=.7cm{
u'_\ast z^\times I
\ar[r]^-{\begin{psmallmatrix}(\varepsilon^{u'}u^\times)_I\\1\end{psmallmatrix}}
& u^\times I\oplus u'_\ast z^\times I
\ar[r]^-{\begin{psmallmatrix}1&-(\varepsilon^{u'}u^\times)_I\end{psmallmatrix}}
\ar[d]_-{u^\times\!\left(\begin{psmallmatrix}(\varepsilon^u)_I&(\varepsilon^v)_I\end{psmallmatrix}\right)}
& u^\times I
\ar[r]^-{0}
\ar[dl]^-{(u^\times h)\circ\varphi}
& (u'_\ast z^\times I)[1]
\\
& u^\times I
}
\]
Thus $u^\times h$ is an isomorphism.

Similarly $v^\times h$ is an isomorphism.
Put $K = \mathrm{Cone}(h)$.
It follows that $u^\times K = 0$ and $v^\times K = 0$.
Hence $z^\times K = 0$.
For each $A \in \mathcal T_X$, apply $\mathrm{Hom}_{\mathcal T_X}(-,K)$ to Mayer--Vietoris triangle $A\longrightarrow u_\ast u^\ast A\oplus v_\ast v^\ast A
\longrightarrow z_\ast z^\ast A\longrightarrow A[1]$.
Then one has a long exact sequence
\[
\longrightarrow
\mathrm{Hom}_{\mathcal T_X}(A, K)[-1]
\longrightarrow 
\mathrm{Hom}_{\mathcal T_X}(z_\ast z^\ast A, K)
\longrightarrow
\mathrm{Hom}_{\mathcal T_X}(u_\ast u^\ast A\oplus v_\ast v^\ast A, K)
\longrightarrow
\]
By adjunctions one has $\mathrm{Hom}_{\mathcal T_X}(A, K) = 0$.
In particular, taking $A=K$ gives that
$\mathrm{Hom}_{\mathcal T_X}(K,K)=0$, hence $1_K=0$.
Thus $K$ is contractible.
Consequently, $h$ is an isomorphism, proving the claim.

Now let $(I_\alpha)_{\alpha\in\Lambda}$ be a family in $\mathcal S_X$.
For $w\in\{u,v,z\}$, Proposition~\ref{c:tac-adjunction} shows that $w^\times I_\alpha$ is totally acyclic.
By assumption and the implication $(2) \implies (1)$, $\mathcal{GI}(\mathcal A_U)$, $\mathcal{GI}(\mathcal A_U)$ and $\mathcal{GI}(\mathcal A_{U \cap V})$ are covering classes, and thus Theorem~\ref{thm:equivalence} shows that $\coprod_\alpha w^\times I_\alpha$ is totally acyclic.
Lemma~\ref{c:open-adjoints} allows this coproduct to pass through $w_\ast$, and Lemma~\ref{c:pushforward-acyclicity} shows that
\[
 \coprod_\alpha w_\ast w^\times I_\alpha
 \cong w_\ast\!\left(\coprod_\alpha w^\times I_\alpha\right)
 \in\mathcal S_X.
\]
Taking coproducts of the triangles \eqref{eq:C-cone-triangle} for each $\alpha$, and using the isomorphisms of cones and $I_\alpha$ proved above, we obtain a distinguished triangle
\[
\coprod_\alpha z_\ast z^\times I_\alpha
\longrightarrow
\left(\coprod_\alpha u_\ast u^\times I_\alpha\right)
\oplus
\left(\coprod_\alpha v_\ast v^\times I_\alpha\right)
\longrightarrow
\coprod_\alpha I_\alpha
\longrightarrow
\left(\coprod_\alpha z_\ast z^\times I_\alpha\right)[1].
\]
The first two terms belong to $\mathcal S_X$ by the preceding argument.
Thus $\coprod_\alpha I_\alpha \in \mathcal S_X$.
Hence $\mathcal S_X$ is closed under coproducts, and Theorem~\ref{thm:equivalence} implies that $\mathcal{GI}(\mathcal A_X)$ is covering.

This completes the proof.
\end{proof}

\subsection{Restriction and recollements}\label{c:sec:restriction-recollements}

By Theorem~\ref{thm:locality}, the covering property passes to open subschemes.
Under this hypothesis, the open-immersion adjunctions restrict to totally acyclic complexes and descend to the quotient categories, as follows.

\begin{proposition}\label{c:restriction-recollements}
Suppose that $\mathcal{GI}(\mathcal A_X)$ is covering, and put
$\mathcal Q_Y=\mathcal T_Y/\mathcal S_Y$ for $Y=X,U$.
Then $u^\ast$ preserves total acyclicity.
The horizontal rows in the following diagram are recollements, and each vertical column is an adjoint triple:
\[
\xymatrix@C=4cm@R=1.5cm{
\mathcal S_U
\ar@{ >->}[r]|{i_U{}_\ast}
\ar[d]|{\widetilde{u_\ast}}
&
\mathcal T_U
\ar@<1.1ex>[l]^{i_U{}^!}
\ar@<-1.1ex>[l]_{i_U{}^\ast}
\ar[r]|{j_U{}^\ast}
\ar[d]|{u_\ast}
&
\mathcal Q_U
\ar@<1.1ex>[l]^{j_{U\ast}}
\ar@<-1.1ex>[l]_{j_{U!}}
\ar[d]|{\widehat{u_\ast}}
\\
\mathcal S_X
\ar@{ >->}[r]|{i_X{}_\ast}
\ar@<1.1ex>[u]^{\widetilde{u^\ast}}
\ar@<-1.1ex>[u]_{\widetilde{u^\times}}
&
\mathcal T_X
\ar@<1.1ex>[l]^{i_X{}^!}
\ar@<-1.1ex>[l]_{i_X{}^\ast}
\ar[r]|{j_X{}^\ast}
\ar@<1.1ex>[u]^{u^\ast}
\ar@<-1.1ex>[u]_{u^\times}
&
\mathcal Q_X
\ar@<1.1ex>[l]^{j_{X\ast}}
\ar@<-1.1ex>[l]_{j_{X!}}
\ar@<1.1ex>[u]^{\widehat{u^\ast}}
\ar@<-1.1ex>[u]_{\widehat{u^\times}}
}
\]
Here $j_Y{}^\ast$ is the quotient functor, with
$j_{Y!}\dashv j_Y{}^\ast\dashv j_{Y\ast}$.
The central squares commute:
\[
\begin{aligned}
u_\ast i_U{}_\ast
&\cong i_X{}_\ast\widetilde{u_\ast},
&\qquad u^\ast i_X{}_\ast
&\cong i_U{}_\ast\widetilde{u^\ast},
&\qquad u^\times i_X{}_\ast
&\cong i_U{}_\ast\widetilde{u^\times},\\
j_X{}^\ast u_\ast
&\cong\widehat{u_\ast}j_U{}^\ast,
&\qquad j_U{}^\ast u^\ast
&\cong\widehat{u^\ast}j_X{}^\ast,
&\qquad j_U{}^\ast u^\times
&\cong\widehat{u^\times}j_X{}^\ast.
\end{aligned}
\]
\end{proposition}

As mentioned in Remark~\ref{rmk:uleft}, the left adjoint $\widetilde {u_\ast}$ need not coincide with the restriction of $u^\ast$.
This proposition shows that $\widetilde{u^\ast} \dashv \widetilde{u_\ast}$ when $\mathcal{GI}(\mathcal A_X)$ is covering.

To prove the proposition, we use the tensor action introduced in \cite[\S2.8]{BILMP2026}.

\begin{definition}
For $A\in\mathcal D_X$, choose a \emph{$K$-flat resolution} $F\to A$ by flat quasi-coherent sheaves, meaning that $F\otimes_{\mathcal O_X}-$ preserves quasi-isomorphisms.
For $I\in\mathcal T_X$, define
\[
 A\star I=\mathrm{Tot}^{\coprod}(F\otimes_{\mathcal O_X}I),
 \qquad
 (A\star I)^n=\coprod_{p+q=n}F^p\otimes_{\mathcal O_X}I^q,
\]
with differential $d_F\otimes1+(-1)^p1\otimes d_I$ on $F^p\otimes_{\mathcal O_X}I^q$.

Since $X$ is Noetherian, tensoring an injective quasi-coherent sheaf with a flat one gives an injective sheaf, and coproducts of injectives are injective.
Thus $A\star I$ has injective terms.
The construction is independent of the chosen $K$-flat resolution \cite[\S2.8]{BILMP2026}.
\end{definition}

Recall that a complex is \emph{perfect} if each point has an affine open neighbourhood on which it is quasi-isomorphic to a bounded complex of sheaves associated to finitely generated projective modules \cite[\S3.1]{BondalVanDenBergh}.

\begin{lemma}\label{lem:PIS}
    For a perfect complex $P$ and $I \in \mathcal S_X$, one has $P \star I \in \mathcal S_X$.
\end{lemma}

\begin{proof}
Since $q_X(P\star I)\cong P\otimes^{\mathbf L}_{\mathcal O_X}q_XI=0$, $P\star I$ is acyclic.

Put $P^\vee=\mathbf R\mathcal Hom_{\mathcal O_X}(P,\mathcal O_X)$, which is again perfect \cite[\S2.8]{BILMP2026}.
Hence
\[
\mathrm{Hom}_{\mathcal T_X}(E,(P\star I)[n])
\cong\mathrm{Hom}_{\mathcal T_X}(P^\vee\star E,I[n])
=0
\qquad(n\in\mathbb Z),
\]
The final equality holds since $P^\vee \star E$ is bounded injective complex.
Thus $P\star I\in\mathcal S_X$.
\end{proof}

\begin{corollary}\label{cor:ais}
If $\mathcal S_X$ is localising, then
$A\star I\in\mathcal S_X$ for all $A\in\mathcal D_X$ and $I\in\mathcal S_X$.
\end{corollary}

\begin{proof}
For fixed $I\in\mathcal S_X$, the subcategory $\{A\in\mathcal D_X\mid A\star I\in\mathcal S_X\}$ is localising, since $-\star I$ is exact and preserves coproducts.
By Lemma~\ref{lem:PIS} this subcategory contains every perfect complex, hence equals $\mathrm{Loc}_{\mathcal D_X}\bigl(\mathrm{Perf}(X)\bigr)=\mathcal D_X$.
\end{proof}

\begin{lemma}\label{c:tensor-open-localisation}
For every $I\in\mathcal T_X$, there is a natural isomorphism
\[
 (\mathbf Ru_\ast\mathcal O_U)\star I\cong u_\ast u^\ast I
 \quad\text{in }\mathcal T_X.
\]
\end{lemma}
\begin{proof}
Since $u$ is an open immersion, the derived counit gives
\[
u^\ast\mathbf Ru_\ast\mathcal O_U
\simeq u^\ast u_\ast(\iota\mathcal O_U)
\cong\iota\mathcal O_U
\simeq\mathcal O_U.
\]
Hence there are natural isomorphisms in $\mathcal T_U$
\[
u^\ast((\mathbf Ru_\ast\mathcal O_U)\star I)
\cong(u^\ast \mathbf Ru_\ast\mathcal O_U)\star(u^\ast I)
\cong \mathcal O_U\star(u^\ast I)
\cong u^\ast I,
\]
whose composition corresponds to $\alpha_I:(\mathbf Ru_\ast\mathcal O_U)\star I\to u_\ast u^\ast I$ by adjunction $u^\ast \dashv u_\ast$.

The case $U=\varnothing$ is immediate.
For affine $X=\mathrm{Spec}(R)$, use the finite principal \v Cech complex $C^\bullet$ from the proof of Lemma~\ref{c:right-adjoint-acyclicity}.
It is bounded and flat, represents $\mathbf Ru_\ast\mathcal O_U$, and hence computes $(\mathbf Ru_\ast\mathcal O_U)\star I=\mathrm{Tot}(C^\bullet\otimes_R I)$.
The augmentation \v Cech double complex
\[
 \beta_I:u_\ast u^\ast I\longrightarrow\mathrm{Tot}(C^\bullet\otimes_R I)
\]
has finite exact horizontal rows, i.e., $(u_\ast u^\ast I)^n \to \mathrm{Tot}(C^\bullet \otimes _R I^n)$ is finite exact for each $n$.
For every $n\in\mathbb Z$ and $0\leq p<m$, Lemma~\ref{c:open-adjoints} gives
\[
u_\ast u^\ast I^n\in\mathrm{Inj}\,\mathcal A_X,
\qquad
C^p\otimes_R I^n
=\coprod_{i_0<\cdots<i_p}(I^n)_{f_{i_0}\cdots f_{i_p}}
\in\mathrm{Inj}\,\mathcal A_X,
\]
where we identify $R$-modules with quasi-coherent sheaves on $X=\mathrm{Spec}(R)$.
Thus, for every $n\in\mathbb Z$, the following sequence is split exact in $\mathrm{Mod}_R$:
\[
0\longrightarrow\Gamma(U,I^n|_U)
\longrightarrow C^0\otimes_R I^n
\longrightarrow C^1\otimes_R I^n
\longrightarrow\cdots
\longrightarrow C^{m-1}\otimes_R I^n
\longrightarrow0.
\]
Thus $\mathrm{Cone}(\beta_I)$ is contractible, and equivalently $\beta_I$ is a homotopy equivalence.

Note that in $\mathcal D_U$, one has the retract
\[
1_{\mathcal O_U} = \big[\mathcal O_U
\xrightarrow{\text{Čech augmentation}}
u^\ast\mathbf Ru_\ast\mathcal O_U
\xrightarrow{\text{counit}}
\mathcal O_U\big].
\]
Applying $- \star (u^\ast I)$ yields that 

\[
\xymatrix@C=1.5cm@R=.7cm{
\mathcal O_U\star(u^\ast I)
\ar[r]
\ar@/^1.5pc/@{=}[rr]
\ar[d]_-{\cong}
& (u^\ast\mathbf R u_\ast\mathcal O_U)\star(u^\ast I)
\ar[r]
\ar[d]^-{\cong}
& \mathcal O_U\star(u^\ast I)
\ar[d]^-{\cong}
\\
u^\ast I
\ar[r]^-{u^\ast\beta_I}
& u^\ast\bigl((\mathbf R u_\ast\mathcal O_U)\star I\bigr)
\ar[r]^-{u^\ast\alpha_I}
& u^\ast I
}
\]
Then applying $u_\ast$ and $u_\ast u^\ast \cong 1$ give that $\alpha_I\beta_I=1$, thus $\alpha_I=\beta_I^{-1}$.

For general $X$, let $v:V\hookrightarrow X$ be an affine open immersion and put $w:U\cap V\hookrightarrow V$.
Then open base change tensor action with restriction identify $v^\ast\alpha_I$ with $\alpha_{v^\ast I}$
\[
\xymatrix@C=1.2cm@R=.7cm{
v^\ast\bigl((\mathbf Ru_\ast\mathcal O_U)\star I\bigr)
\ar[r]^-{\cong}
\ar[d]_-{v^\ast\alpha_I}
& (v^\ast\mathbf Ru_\ast\mathcal O_U)\star(v^\ast I)
\ar[r]^-{\cong}
& (\mathbf Rw_\ast\mathcal O_{U\cap V})\star(v^\ast I)
\ar[d]^-{\alpha_{v^\ast I}}
\\
v^\ast u_\ast u^\ast I
\ar[r]^-{\cong}
& w_\ast(I|_{U\cap V})
\ar[r]^-{\cong}
& w_\ast w^\ast(v^\ast I)
}
\]
The affine case shows that $v^\ast\alpha_I$ is a homotopy equivalence, and so is $\alpha_{v^\ast I}$.
Restriction commutes with mapping cones, so $v^\ast\mathrm{Cone}(\alpha_I)
\cong\mathrm{Cone}(v^\ast\alpha_I)
\simeq0$.
Thus $\mathrm{Cone}(\alpha_I)$ is acyclic on any affine open cover, and hence $\mathrm{Cone}(\alpha_I)$ is acyclic.
Since $v^\ast$ is exact, $v^\ast Z^n\mathrm{Cone}(\alpha_I) \cong Z^n\bigl(v^\ast\mathrm{Cone}(\alpha_I)\bigr)$ is a direct summand of a term of  $v^\ast\mathrm{Cone}(\alpha_I)$, and thus is injective.
By \cite[Proposition~B.1]{CoupekStovicek2020}, each $Z^n\mathrm{Cone}(\alpha_I)$ is injective on $X$.
Consequently, the short exact sequences
\[
0\longrightarrow Z^n\mathrm{Cone}(\alpha_I)
\longrightarrow\mathrm{Cone}(\alpha_I)^n
\longrightarrow Z^{n+1}\mathrm{Cone}(\alpha_I)
\longrightarrow0
\]
split for every $n\in\mathbb Z$.
Hence $\mathrm{Cone}(\alpha_I)$ is contractible, and $\alpha_I$ is an isomorphism in $\mathcal T_X$.
\end{proof}

\begin{proof}[Proof of Proposition~\ref{c:restriction-recollements}]
When $\mathcal{GI}(\mathcal A_X)$ is covering, Theorem~\ref{thm:equivalence} shows that $\mathcal S_X$ is localising.
Then by Corollary~\ref{cor:ais} and Lemma~\ref{c:tensor-open-localisation}, one has
\[
 u_\ast u^\ast I\cong(\mathbf Ru_\ast\mathcal O_U)\star I\in\mathcal S_X
 \qquad(I\in\mathcal S_X).
\]
Since $u_\ast$ reflects total acyclicity by Lemma~\ref{c:pushforward-acyclicity}, $u^\ast I\in\mathcal S_U$.
Together with Proposition~\ref{c:tac-adjunction}, $u^\ast \dashv u_\ast \dashv u^\times$ restrict to the full subcategories $\mathcal S_X$ and $\mathcal S_U$, yielding $\widetilde{u^\ast}\dashv\widetilde{u_\ast}\dashv\widetilde{u^\times}$.

By Theorem~\ref{thm:locality}, $\mathcal{GI}(\mathcal A_U)$ is covering.
Thus, for $Y=X,U$, Theorem~\ref{thm:equivalence} supplies the right adjoint $i_Y{}^!$.
Denote by $D_{\mathrm{qc}}(\mathcal O_X)$ the full subcategory of the derived category of all $\mathcal O_X$-modules consisting of complexes with quasi-coherent cohomology. By \cite[Theorem~3.1.1]{BondalVanDenBergh},
$D_{\mathrm{qc}}(\mathcal O_Y)$ is compactly generated;
the equivalence
$\mathcal D_Y\simeq D_{\mathrm{qc}}(\mathcal O_Y)$
of \cite[Theorem~1.2]{HallNeemanRydh} shows that $\mathcal D_Y$ is compactly generated.
Hence \cite[Lemma~7.3]{Krause2005} gives the left adjoint $i_Y{}^\ast$, and Proposition~\ref{prop:tac-recollement} gives the recollements.

Since all three ambient functors preserve the totally acyclic subcategories, they descend to the Verdier quotients.
The adjunction units and counits descend with their triangle identities, giving $\widehat{u^\ast}\dashv\widehat{u_\ast}\dashv\widehat{u^\times}$.

The first three displayed squares commute by the definitions of the restricted functors, and the last three by the definitions of the quotient functors.
\end{proof}

\section{Failure of stalkwise detection}\label{sec:stalkwise-descent}
Theorem~\ref{thm:locality} shows that the covering property for Gorenstein injective quasi-coherent sheaves is Zariski local. In this section, we show that it cannot be checked pointwise.

We construct a countable commutative one-dimensional Cohen--Macaulay Noetherian nonlocal ring $S$ such that every $S_{\mathfrak p}$ admits a dualising complex and $\mathcal{GI}(S_{\mathfrak p})$ is covering, whereas $\mathcal{GI}(S)$ is not covering.

The main argument is given in Subsection~\ref{c:sec:ring}. The subsequent subsections justify its inputs: Subsection~\ref{c:sec:ring-properties} establishes the ring properties, Subsections~\ref{c:sec:artin} and~\ref{c:sec:lift} construct the Gorenstein injective modules $H_i \in \mathcal{GI}(S)$, and Subsection~\ref{c:sec:telescope} proves the required nonvanishing of $\varprojlim{}^1$, which yields that $\coprod_{u \in \mathbb N_+} H_i \notin \mathcal{GI}(S)$.

\subsection{Outline of the example}\label{c:sec:ring}

Let $t_1,t_2,\ldots$ be algebraically independent over $\mathbb Q$. 
Set $D_i=\mathbb Q[t_i^3,t_i^4,t_i^5]$.
Let $B_0$ denote the colimit of the countable inclusion of the system
\[
B_0 = \varinjlim \big[D_1 \xrightarrow{w \mapsto w \otimes 1} D_1 \otimes_{\mathbb Q} D_2\xrightarrow{w \mapsto w \otimes 1} D_1 \otimes_{\mathbb Q} D_2 \otimes D_3 \xrightarrow{w \mapsto w \otimes 1} \cdots \big].
\]
Note that $B_0$ is a domain, and $\mathfrak p_i = (t_i ^3, t_i ^4, t_i^5)$ is a prime ideal of $B_0$ for each $i \in \mathbb N_+$.
Hence $U = B_0 \setminus \bigcup_{i \in \mathbb N_+} \mathfrak p_i$ is a multiplicative set and denote the the localisation $B = U^{-1} B_0$.

Let $S = B[u,v]/(u,v)^2$ be the final ring, and put $\mathfrak n=(u,v)$.
By Corollary~\ref{c:lem:ringproperties} one has
\[
\mathrm{Spec}(S) = \{\mathfrak n\} \cup \{\mathfrak m_i\}_{i \in \mathbb N_+},
\qquad \mathfrak m_i = (t_i^3,t_i^4,t_i^5,u,v).
\]
Moreover, $S$ is Noetherian and Cohen--Macaulay of dimension one (Lemma~\ref{c:lem:S-properties}).

For each $i \in \mathbb N_+$, let $B_i$ be a subring of $B_0$ by taking each $t_i$ to be $1$, $K_i = \mathrm{Frac}(B_i)$ the fractional field.
Write $K = \mathrm{Frac}(B)$.
By Lemma~\ref{c:lem:S-localisations} the localisations are
\[
S_{\mathfrak m_i}
\cong K_i[t_i^3,t_i^4,t_i^5]_{(t_i^3,t_i^4,t_i^5)}[u,v]/(u,v)^2,\qquad
S_{\mathfrak n}\cong K[u,v]/(u,v)^2.
\]
Each maximal localisation is finite free of rank nine over the discrete valuation ring $K_i[t_i^3]_{(t_i^3)}$.
By Proposition~\ref{c:prop:local-dualising} its dualising complex can be $\mathrm{Hom}_{K_i[t_i^3]_{(t_i^3)}} \bigl(S_{\mathfrak m_i},K_i[t_i^3]_{(t_i^3)}\bigr)$, concentrated in degree zero.
At $\mathfrak n$, a dualising complex concentrated in degree zero is $\mathrm{Hom}_K(K[u,v]/(u,v)^2,K)$ (Proposition~\ref{c:prop:local-dualising}).
By Corollary~\ref{c:cor:local-covering}, $\mathcal{GI}(S_{\mathfrak q})$ is a covering class in $\mathrm{Mod}_{S_{\mathfrak q}}$ for each $\mathfrak q \in \mathrm{Spec}(S)$.

The construction and verification of Gorenstein injectivity is use the similar argument as in Example~\ref{thm:main}.
For each $i\in\mathbb N_+$, Subsection~\ref{c:sec:lift} constructs $H_i\in\mathcal{GI}(S)$ supported at $\mathfrak m_i$.
Put $\Omega=\mathrm{Hom}_B(S,B)$ and $a_n=\prod_{i=1}^n t_i^3$.
By Lemma~\ref{c:lem:colimit-injective}, the direct system
\[
\Omega\xrightarrow{a_1}\Omega\xrightarrow{a_2}\Omega\xrightarrow{a_3}\cdots
\]
has colimit $E_S(S/\mathfrak n)$ and yields the telescope sequence
\[
0\longrightarrow \Omega^{(\mathbb N_+)}
\xlongrightarrow{1-\nabla}\Omega^{(\mathbb N_+)}
\longrightarrow E_S(S/\mathfrak n)\longrightarrow0,
\]
where $\nabla$ sends the $n$th summand to the $(n+1)$st by multiplication by $a_n$.
Applying $\mathrm{Hom}_S(-,\coprod_{i\in\mathbb N_+}H_i)$ gives a commutative diagram
\[
\xymatrix@C=1cm@R=.5cm{
\left(\coprod_{i\in\mathbb N_+}\mathrm{Ext}_S^1(\Omega,H_i)\right)^{\mathbb N_+}
\ar[r]^-{\simeq}\ar[d]_{1-\Delta}
&
\mathrm{Ext}_S^1\left(\Omega^{(\mathbb N_+)},\coprod_{i\in\mathbb N_+}H_i\right)
\ar[d]^{\mathrm{Ext}_S^1(1-\nabla,\coprod_{i\in\mathbb N_+}H_i)}
\\
\left(\coprod_{i\in\mathbb N_+}\mathrm{Ext}_S^1(\Omega,H_i)\right)^{\mathbb N_+}
\ar[r]^-{\simeq}\ar@{->>}[d]
&
\mathrm{Ext}_S^1\left(\Omega^{(\mathbb N_+)},\coprod_{i\in\mathbb N_+}H_i\right)
\ar[d]
\\
\varprojlim{}^1_n\left(\coprod_{i\in\mathbb N_+}\mathrm{Ext}_S^1(\Omega,H_i)\right)
\ar@{ >.>}[r]
&
\mathrm{Ext}_S^2\left(E_S(S/\mathfrak n),\coprod_{i\in\mathbb N_+}H_i\right)
}
\]
with columns exact at middle.
Here $\Delta((y_n))=(a_ny_{n+1})$ is the shift of the inverse tower
\[
\cdots\xrightarrow{a_3}
\coprod_{i\in\mathbb N_+}\mathrm{Ext}_S^1(\Omega,H_i)
\xrightarrow{a_2}
\coprod_{i\in\mathbb N_+}\mathrm{Ext}_S^1(\Omega,H_i)
\xrightarrow{a_1}
\coprod_{i\in\mathbb N_+}\mathrm{Ext}_S^1(\Omega,H_i).
\]
Lemma~\ref{c:lem:global-obstruction} shows that this $\varprojlim{}^1$ is nonzero.
Consequently
\[
\mathrm{Ext}_S^2\left(E_S(S/\mathfrak n),\coprod_{i\in\mathbb N_+}H_i\right)\ne0.
\]
Since $E_S(S/\mathfrak n)$ is injective, $\coprod_{i\in\mathbb N_+}H_i$ is not Gorenstein injective.
Thus $\mathcal{GI}(S)$ is not closed under coproducts and is not covering by Theorem~\ref{thm:equivalence}.
The precover-splitting argument of Subsection~\ref{subsec:exampleB-outline} also shows that $\coprod_{i\in\mathbb N_+}H_i$ has no Gorenstein injective precover.

\subsection{The ring \texorpdfstring{$S$}{S}, its spectrum and localisations}\label{c:sec:ring-properties}
We computes $\mathrm{Spec}(S)$ and characterises its properties, then compute the dualising complex of $S_{\mathfrak q}$ for each $\mathfrak q \in \mathrm{Spec}(S)$.

Analogous to Subsection~\ref{subsec:exampleB-outline}, one has a surjective ring map 
\[
\pi : S \longrightarrow B, \qquad 1_S \longmapsto 1_B, \quad u \longmapsto 0, \quad v \longmapsto 0
\]
with nilpotent kernel $\mathfrak n$.
Thus $\mathrm{Spec}(S) = \{\pi^{-1}(\mathfrak q) \mid \mathfrak q \in \mathrm{Spec}(B)\}$.
The localisation gives
\[
\mathrm{Spec}(B)
\longleftrightarrow
\{\mathfrak p\in\mathrm{Spec}(B_0)\mid\mathfrak p\cap U=\varnothing\}.
\]
Hence one has a bijection
\[
\{ \mathfrak p \in\mathrm{Spec}(B_0)\mid\mathfrak p\cap U=\varnothing\}\xlongrightarrow{\sim} \mathrm{Spec}(S) , \qquad \mathfrak p \longmapsto \pi^{-1} (U^{-1}\mathfrak p).
\]

\begin{lemma}
For each $i \in \mathbb N_+$, $\mathfrak p_i=(t_i^3,t_i^4,t_i^5)$ is a prime ideal of $B_0$.
Then 
\[
\{P\in\mathrm{Spec}(B_0):P\cap U=\varnothing\}
=\{(0)\}\cup\{\mathfrak p_i:i\ge1\},
\]
\end{lemma}

\begin{proof}
We first claim that any nonzero prime ideal of $B_0$ with empty intersection with $U$ is contained in some $\mathfrak p_i$.
Otherwise, suppose that there exists a counterexample $P$, that is, $P \not\subseteq \mathfrak p_i$ for all $i \in \mathbb N_+$.
Then for any nonzero $a \in P$, put $F(a) = \{i \in \mathbb N_+\mid a \in \mathfrak p_i\}$.
By colimit construction, $a$ lies in the image of $D_1 \otimes \cdots \otimes D_m$ for some $m$ and thus $F(a)$ is finite.
Then by prime avoidance \cite[Proposition~1.11\textup{(i)}]{AtiyahMacdonald1969},
choose $b\in P\setminus\bigcup_{i\in F}\mathfrak p_i$.
Hence for each $i \in \mathbb N_+$, one has $a \notin \mathfrak p_i$ or $b \notin \mathfrak p_i$. 
Again $a$ and $b$ lies in the image of $D_1 \otimes \cdots \otimes D_n$ for some $n \in \mathbb N_+$, thus neither belongs to $\mathfrak p_i$ for $i > n$.
In this case, there is $\lambda \in \mathbb Q$ such that $(a + \lambda b) \notin \mathfrak p_i$ for each $i \in \mathbb N_+$.
Thus $(a + \lambda b) \in U$, a contradiction.

Then fix a nonzero prime ideal $P$ of $B_0$.
The above discussion gives $P \subseteq \mathfrak p_i$ for some $i \in \mathbb N_+$.
Since $K_i [t_i^3, t_i^4, t_i^5] \cong K_i[t_i^3] \oplus t_i^4 K_i[t_i^3] \oplus t_i^5 K_i[t_i^3]$ as a $K_i[t_i^3]$-module, $K_i [t_i^3, t_i^4, t_i^5]$ is an integral domain with $\dim K_i [t_i^3, t_i^4, t_i^5] = \dim K_i [t_i^3] = 1$.
Thus the local ring 
\[
(B_0)_{\mathfrak p_i} \cong K_i[t_i^3,t_i^4,t_i^5]_{(t_i^3,t_i^4,t_i^5)}
\]
is a one-dimensional domain.
Clearly $P$ is sent to a nonzero prime ideal and thus is the maximal ideal.
By assumption $P \subseteq \mathfrak p_i$, one has $P=\mathfrak p_i$.

Each $\mathfrak p_i$ and $0$ are disjoint from $U$, proving this lemma.
\end{proof}

\begin{corollary}\label{c:lem:ringproperties}
The prime spectrum of $S$ is
\[
\mathrm{Spec}(S) = \{\mathfrak n\} \cup \{\mathfrak m_i\}_{i \in \mathbb N_+},
\qquad \mathfrak m_i = (t_i^3,t_i^4,t_i^5,u,v).
\]
\end{corollary}

\begin{lemma}\label{c:lem:S-localisations}
The localisations of $S$ are
\[
S_{\mathfrak m_i}\cong B_{\mathfrak m_{B,i}}[u,v]/(u,v)^2,\qquad
S_{\mathfrak n}\cong K[u,v]/(u,v)^2.
\]
\end{lemma}
\begin{proof}
For $\mathfrak q\in\mathrm{Spec}(B)$, localising the coefficients gives $B_{\mathfrak q}[u,v]/(u,v)^2$, a local ring with maximal ideal $\mathfrak qB_{\mathfrak q}+(u,v)$.
Thus this ring is $S_{\mathfrak qS+\mathfrak n}$.
Taking $\mathfrak q=\mathfrak m_{B,i}$ and $\mathfrak q=(0)$, and using $B_{(0)}=K$, gives the assertions.
\end{proof}

\begin{lemma}\label{c:lem:S-properties}
The ring $S$ is Noetherian and Cohen--Macaulay of dimension one.
\end{lemma}

\begin{proof}
By computation of $\mathrm{Spec}(S)$, $S$ has Krull dimension $1$.
Since every prime ideal of $B$ is finitely generated by the preceding description, $B$ is Noetherian by Cohen's theorem \cite[Theorem~3.4]{Matsumura1989}.
Since $S$ is finite free of rank $3$ over $B$, $S$ is Noetherian.
By Lemma~\ref{c:lem:S-localisations}, each $S_{\mathfrak m_i}$ is free over the domain $B_{\mathfrak m_{B,i}}$, so $t_i^3$ is a nonzerodivisor in its maximal ideal.
Hence it has nonzero depth and thus coincides with the Krull dimension $1$.
The remaining localisation $S_{\mathfrak n}$ is again Artinian and therefore $S$ is Cohen--Macaulay.
\end{proof}

\begin{proposition}\label{c:prop:local-dualising}
Dualising complexes for $S_{\mathfrak m_i}$ and $S_{\mathfrak n}$ are respectively the modules
\[
\mathrm{Hom}_{K_i[t_i^3]_{(t_i^3)}}
\bigl(S_{\mathfrak m_i},K_i[t_i^3]_{(t_i^3)}\bigr),\qquad
\mathrm{Hom}_K(K[u,v]/(u,v)^2,K),
\]
both are viewed as complexes concentrated in degree zero.
\end{proposition}

\begin{proof}
Consider either of the pairs \(
(\Lambda,T)=
\bigl(K_i[t_i^3]_{(t_i^3)},S_{\mathfrak m_i}\bigr)
\quad\text{or}\quad
\bigl(K,K[u,v]/(u,v)^2\bigr),\) and put $\omega=\mathrm{Hom}_\Lambda(T,\Lambda)$.
In the first case, $T$ is finite free over $\Lambda$ with basis
\[
\{1,t_i^4,t_i^5,u,t_i^4u,t_i^5u,v,t_i^4v,t_i^5v\}.
\]
in the second case, it has basis $\{1,u,v\}$.
Thus $\omega$ is finite free over $\Lambda$.

Since $\Lambda$ is a discrete valuation ring or a field, it has an injective resolution of length at most one.
Applying $\mathrm{Hom}_\Lambda(T,-)$ gives an injective resolution of $\omega$ over $T$: exactness follows from the freeness of $T$, and injectivity follows from the adjunction
\[
\mathrm{Hom}_T\bigl(-,\mathrm{Hom}_\Lambda(T,J)\bigr)
\cong\mathrm{Hom}_\Lambda(-,J).
\]
In particular, $\mathrm{id}_T\omega\leq1$.

Finally, the same adjunction gives
\[
\mathrm{Ext}_T^q(\omega,\omega)
\cong\mathrm{Ext}_\Lambda^q(\omega,\Lambda)=0
\qquad(q>0),
\]
and the homothety map is identified as
\[
T\longrightarrow\mathrm{Hom}_T(\omega,\omega)
\cong\mathrm{Hom}_\Lambda(\omega,\Lambda), \qquad t \longmapsto \big(f \longmapsto f(t)\big)
\]
Thus $\omega$ is a dualising complex for $T$.
\end{proof}
\begin{corollary}\label{c:cor:local-covering}
For every prime $\mathfrak q$ of $S$, $\mathcal{GI}(S_{\mathfrak q})$ is covering.
\end{corollary}
\begin{proof}
Proposition~\ref{c:prop:local-dualising} gives dualising complexes at every prime of $S$.
The assertion follows from \cite[Theorem~6.9]{CFH} and Theorem~\ref{thm:equivalence}.
\end{proof}

\subsection{Gorenstein injective modules over \texorpdfstring{$S_{\mathfrak m_i}/t_i^3S_{\mathfrak m_i}$}{the Artinian quotients}}\label{c:sec:generic}\label{c:sec:artin}

The construction of $H_i \in \mathcal{GI}(S)$ is based on the construction of $L_i \in \mathcal{GI}(S_{\mathfrak m_i}/t_i^3S_{\mathfrak m_i})$, where
\[
S_{\mathfrak m_i}/t_i^3S_{\mathfrak m_i}
\cong K_i[a,b,u,v]/\bigl((a,b)^2,(u,v)^2\bigr),
\qquad a=\overline{t_i^4},\quad b=\overline{t_i^5}.
\]
Put $\Omega=\mathrm{Hom}_B(S,B)$.

\begin{lemma}\label{c:lem:genericinjective}
There are isomorphisms of $S$-modules
\[
\Omega\otimes_BK\cong\mathrm{Hom}_B(S,K)\cong E_S(S/\mathfrak n).
\]
\end{lemma}
\begin{proof}
As in the proof of Lemma~\ref{lem:J-injective}, every element outside $\mathfrak n$ acts injectively (thus invertibly) on $E_S(S/\mathfrak n)$, so \cite[Lemma~3.2.5]{BrunsHerzog} gives $E_S(S/\mathfrak n)\cong E_{S_{\mathfrak n}}(K)$.
By Lemma~\ref{c:lem:S-localisations}, $S_{\mathfrak n}=S\otimes_BK\cong K[u,v]/(u,v)^2$; applying Proposition~\ref{prop:finite-local-injective} to $K\to S_{\mathfrak n}$ yields
\[
E_S(S/\mathfrak n)\cong E_{S_{\mathfrak n}}(K)
\cong\mathrm{Hom}_K(S_{\mathfrak n},K)
\cong\mathrm{Hom}_B(S,K)\cong\Omega\otimes_BK.
\]
The last isomorphism uses finite freeness of $S$ over $B$.
\end{proof}

For each $i \in \mathbb N_+$, set $X_i=K_i[a,b]/(a,b)^2$ and $Y_i=K_i[u,v]/(u,v)^2$.
Then $S_{\mathfrak m_i}/t_i^3S_{\mathfrak m_i}\cong X_i\otimes_{K_i}Y_i$.
For a finite-dimensional $K_i$-space $W$, write $W^*=\mathrm{Hom}_{K_i}(W,K_i)$.

Use the matrix construction from Subsection~\ref{The Gorenstein injective modules}; here $X_i,Y_i,P_i$ play the roles of $X,Y,P$ there, over $K_i$ with variables $a,b$.
Let $N_i=\coprod_{r,s\geq0}\bigl(X_i\otimes_{K_i}Y_i\bigr) e_{r,s}$ be the free $\bigl(X_i\otimes_{K_i}Y_i\bigr)$-module on the indicated basis.
For all $r,s\geq0$, define
\begin{align*}
    \alpha:N_i\longrightarrow N_i,&\qquad e_{2r,s}\longmapsto ue_{r,s},\quad e_{2r+1,s}\longmapsto ve_{r,s};\\
    \beta:N_i\longrightarrow N_i,&\qquad e_{r,s}\longmapsto ae_{r,2s}+be_{r,2s+1}.
\end{align*}
Then set $d=\begin{psmallmatrix}\alpha&\beta\\\beta&-\alpha\end{psmallmatrix}$.
Note that $d^2=0$.
Then one has a $1$-periodic cochain complex
\[
P_i=(\cdots\xlongrightarrow dN_i\oplus N_i\xlongrightarrow dN_i\oplus N_i\xlongrightarrow d\cdots).
\]
By the similar analysis of Lemma~\ref{lem:E-tensor}, $P_i$ and $\bigl(X_i\otimes_{K_i}Y_i\bigr)^*\otimes_{X_i\otimes_{K_i}Y_i}P_i$ are acyclic.
Let $\varepsilon_0,\varepsilon_u,\varepsilon_v$ be the basis of $Y_i^*$ dual to $1,u,v$.
Under the natural identification
\[
\bigl(X_i\otimes_{K_i}Y_i^*\bigr)\otimes_{X_i\otimes_{K_i}Y_i}(P_i)^0
\cong\left(\coprod_{r,s\geq0}(X_i\otimes_{K_i}Y_i^*)e_{r,s}\right)^2,
\]
let $\lambda_i$ be the $K_i$-linear functional extracting the coefficient of $(b\otimes\varepsilon_v)e_{0,0}$ in the first component, and put $z_i=((b\otimes\varepsilon_v)e_{0,0},0)$.
The calculation in the proof of Lemma~\ref{lem:Q}, with coefficient field $K_i$ and $x_1,x_2$ replaced by $a,b$, gives
\[
dz_i=0,\qquad \lambda_i(z_i)=1,\qquad \lambda_i d=0.
\]
Thus $z_i$ represents a nonzero cohomology class; periodicity gives
\[
H^j\left(\bigl(X_i\otimes_{K_i}Y_i^*\bigr)\otimes_{X_i\otimes_{K_i}Y_i}P_i\right)\ne0
\qquad(j\in\mathbb Z).
\]

\begin{lemma}\label{c:lem:artin-periodic}
There is a Gorenstein injective $\bigl(X_i\otimes_{K_i}Y_i\bigr)$-module $L_i$ with an exact sequence
\begin{equation}\label{c:eq:Lperiodic}
0\longrightarrow L_i\longrightarrow I_i^0\longrightarrow L_i\longrightarrow0,
\qquad I_i^0\text{ injective over }\bigl(X_i\otimes_{K_i}Y_i\bigr),
\end{equation}
and $\mathrm{Ext}_{X_i\otimes_{K_i}Y_i}^j(\bigl(X_i\otimes_{K_i}Y_i^*\bigr),L_i)\ne0$ for each $j \geq 1$.
\end{lemma}
\begin{proof}
The duality $\mathrm{Hom}_{X_i\otimes_{K_i}Y_i}(-,\bigl(X_i\otimes_{K_i}Y_i\bigr)^*)\cong\mathrm{Hom}_{K_i}(-,K_i)$ makes $\bigl(X_i\otimes_{K_i}Y_i\bigr)^*$ an injective cogenerator.
For Noetherian rings, every injective $\bigl(X_i\otimes_{K_i}Y_i\bigr)$-module is a summand of coproduct of copies of $\bigl(X_i\otimes_{K_i}Y_i\bigr)^*$.

Set $I_i=\mathrm{Hom}_{X_i\otimes_{K_i}Y_i}(P_i,\bigl(X_i\otimes_{K_i}Y_i\bigr)^*)$.
Since the terms of $P_i$ are free, the terms of $I_i$ are products of the injective module $\bigl(X_i\otimes_{K_i}Y_i\bigr)^*$ and hence are injective.
Since $\mathrm{Hom}_{X_i\otimes_{K_i}Y_i}(-,\bigl(X_i\otimes_{K_i}Y_i\bigr)^*)$ is exact and  $P_i$ acyclic, $I_i$ is acyclic.
For any injective $\bigl(X_i\otimes_{K_i}Y_i\bigr)$-module $E$, the complex $E\otimes_{X_i\otimes_{K_i}Y_i}P_i$ is a direct summand of a coproduct of copies of the acyclic complex $\bigl(X_i\otimes_{K_i}Y_i\bigr)^*\otimes_{X_i\otimes_{K_i}Y_i}P_i$, so it is acyclic.
Thus
\[
\mathrm{Hom}_{X_i\otimes_{K_i}Y_i}(E,I_i)
\cong\mathrm{Hom}_{X_i\otimes_{K_i}Y_i}
\bigl(E\otimes_{X_i\otimes_{K_i}Y_i}P_i,\bigl(X_i\otimes_{K_i}Y_i\bigr)^*\bigr).
\]
It follows that $\mathrm{Hom}_{X_i\otimes_{K_i}Y_i}(E,I_i)$ is acyclic.
Thus $I_i$ is totally acyclic.
Put $L_i=Z^0I_i$, one has \eqref{c:eq:Lperiodic}.
For $j\geq1$, the exactness of $\mathrm{Hom}_{X_i\otimes_{K_i}Y_i}(-,\bigl(X_i\otimes_{K_i}Y_i\bigr)^*)$ gives
\[
\begin{aligned}
\mathrm{Ext}_{X_i\otimes_{K_i}Y_i}^j(\bigl(X_i\otimes_{K_i}Y_i^*\bigr),L_i)
&\cong H^j\mathrm{Hom}_{X_i\otimes_{K_i}Y_i}(\bigl(X_i\otimes_{K_i}Y_i^*\bigr),I_i)\\
&\cong\mathrm{Hom}_{K_i}\bigl(H^{-j}(\bigl(X_i\otimes_{K_i}Y_i^*\bigr)\otimes_{X_i\otimes_{K_i}Y_i}P_i),K_i\bigr)\ne0.
\end{aligned}
\]
By the identification $S_{\mathfrak m_i}/t_i^3S_{\mathfrak m_i}\cong X_i\otimes_{K_i}Y_i$, one has $L_i \in \mathcal{GI}(S_{\mathfrak m_i}/t_i^3S_{\mathfrak m_i} )$.
\end{proof}

\subsection{Construction of the modules \texorpdfstring{$H_i$}{Hi}}\label{c:sec:lift}

Put $H_i= E_{S_{\mathfrak m_i}}(L_i) / L_i$.

\begin{lemma}\label{c:restriction}
Let $S$ be a commutative Noetherian ring and $\Sigma\subseteq S$ a multiplicative set.
Every Gorenstein injective $\Sigma^{-1}S$-module is Gorenstein injective over $S$.
\end{lemma}
\begin{proof}
Let $I$ be a totally acyclic complex of injective $\Sigma^{-1}S$-modules.
The restriction $\mathrm{Mod}_{\Sigma^{-1}S} \to \mathrm{Mod}_S$ is exact and preserves injectives, since its left adjoint $\Sigma^{-1}(-)$ is exact.
For an injective $S$-module $E$, the module $\Sigma^{-1}E$ is injective over $\Sigma^{-1}S$, since $S$ is commutative Noetherian.
The adjunction $\mathrm{Hom}_S(E,I)\cong \mathrm{Hom}_{\Sigma^{-1}S}(\Sigma^{-1}E,I)$ proves that $S$ is totally acyclic over $S$.
\end{proof}

\begin{lemma}\label{c:lem:lift}
The module $H_i$ is Gorenstein injective over $S_{\mathfrak m_i}$ and over $S$.
\end{lemma}
\begin{proof}
Since $t_i^3$ is regular on $S_{\mathfrak m_i}$ and $L_i$ is Gorenstein injective over $S_{\mathfrak m_i}/t_i^3S_{\mathfrak m_i}$, \cite[Theorem~4.2]{BennisMahdou2010} gives $\mathrm{Gid}_{S_{\mathfrak m_i}}L_i
=\mathrm{Gid}_{S_{\mathfrak m_i}/t_i^3S_{\mathfrak m_i}}L_i+1=1$.
Applying \cite[Lemma~2.5\textup{(3)}]{BennisMahdou2010} to
\[
0\longrightarrow L_i\longrightarrow E_{S_{\mathfrak m_i}}(L_i)
\longrightarrow H_i\longrightarrow0
\]
gives $\mathrm{Gid}_{S_{\mathfrak m_i}}H_i\leq\max\{0,1-1\}=0$.
Thus $H_i$ is Gorenstein injective over $S_{\mathfrak m_i}$, and hence is also Gorenstein injective over $S$ by Lemma~\ref{c:restriction}.
\end{proof}

\begin{lemma}\label{c:lem:lift-ext}
There is an isomorphism of $S_{\mathfrak m_i}$ modules
\[
0\ne\mathrm{Ext}_{S_{\mathfrak m_i}}^1(\Omega_{\mathfrak m_i},H_i)\cong\mathrm{Ext}_{X_i\otimes_{K_i}Y_i}^2(\bigl(X_i\otimes_{K_i}Y_i^*\bigr),L_i).
\]
\end{lemma}

\begin{proof}
A straightforward calculation gives
\begin{align*}
\frac{\Omega_{\mathfrak m_i}}{t_i^3\Omega_{\mathfrak m_i}}
&\cong \Omega_{\mathfrak m_i}
\otimes_{B_{\mathfrak m_{B,i}}}
\frac{B_{\mathfrak m_{B,i}}}
{t_i^3B_{\mathfrak m_{B,i}}} \cong \Omega_{\mathfrak m_i}
\otimes_{B_{\mathfrak m_{B,i}}}
\frac{K_i[a,b]}{(a,b)^2} =\Omega_{\mathfrak m_i}
\otimes_{B_{\mathfrak m_{B,i}}}X_i \cong \Omega_{\mathfrak m_i}
\otimes_{B_{\mathfrak m_{B,i}}}X_i\\
&\cong \bigl(\Omega\otimes_B B_{\mathfrak m_{B,i}}\bigr)
\otimes_{B_{\mathfrak m_{B,i}}}X_i = \bigl(\mathrm{Hom}_B(S,B)
\otimes_B B_{\mathfrak m_{B,i}}\bigr)
\otimes_{B_{\mathfrak m_{B,i}}}X_i\\
&\cong \mathrm{Hom}_{B_{\mathfrak m_{B,i}}}
\bigl(S\otimes_B B_{\mathfrak m_{B,i}},
B_{\mathfrak m_{B,i}}\bigr)
\otimes_{B_{\mathfrak m_{B,i}}}X_i \cong \mathrm{Hom}_{B_{\mathfrak m_{B,i}}}
\bigl(S_{\mathfrak m_i},B_{\mathfrak m_{B,i}}\bigr)
\otimes_{B_{\mathfrak m_{B,i}}}X_i\\
&\cong \mathrm{Hom}_{X_i}
\bigl(S_{\mathfrak m_i}\otimes_{B_{\mathfrak m_{B,i}}}X_i,X_i\bigr)\cong \mathrm{Hom}_{X_i}
\bigl(S_{\mathfrak m_i}/t_i^3S_{\mathfrak m_i},X_i\bigr)\\
&\cong \mathrm{Hom}_{X_i}(X_i\otimes_{K_i}Y_i,X_i)
\cong X_i\otimes_{K_i}\mathrm{Hom}_{K_i}(Y_i,K_i) \quad = X_i\otimes_{K_i}Y_i^*.
\end{align*}
It follows that 
\[
\begin{aligned}
\mathrm{Ext}_{S_{\mathfrak m_i}}^1(\Omega_{\mathfrak m_i},H_i)
&\cong\mathrm{Ext}_{S_{\mathfrak m_i}}^2(\Omega_{\mathfrak m_i},L_i) \cong\mathrm{Ext}_{S_{\mathfrak m_i}/t_i^3S_{\mathfrak m_i}}^2
\left(\frac{\Omega_{\mathfrak m_i}}{t_i^3\Omega_{\mathfrak m_i}},L_i\right)\\
&\cong\mathrm{Ext}_{X_i\otimes_{K_i}Y_i}^2
(X_i\otimes_{K_i}Y_i^*,L_i)\ne0.
\end{aligned}
\]
The final inequality follows from Lemma~\ref{c:lem:artin-periodic}.
\end{proof}

\begin{lemma}\label{c:lem:lift-support}
As an $S$-module, $H_i$ satisfies $\mathrm{Supp}_S(H_i)=\{\mathfrak m_i\}$.
\end{lemma}
\begin{proof}
Since $L_i$ is a module over $S_{\mathfrak m_i}/t_i^3S_{\mathfrak m_i}\cong X_i\otimes_{K_i}Y_i$, whose maximal ideal $(a,b,u,v)$ has cube zero, one has $(\mathfrak m_iS_{\mathfrak m_i})^3L_i=0$.
The socle of $L_i$ is an essential submodule of $L_i$.
Thus $E_{S_{\mathfrak m_i}}(L_i)$ is a coproduct of copies of $E_{S_{\mathfrak m_i}}(K_i)$.
Each copy is $\mathfrak m_i$-power torsion by \cite[Theorem~18.4(v)]{Matsumura1989}, so the quotient $H_i$ is also $\mathfrak m_i$-power torsion.
Hence $(H_i)_{\mathfrak q}=0$ for every prime $\mathfrak q\ne\mathfrak m_i$.
Finally, $H_i\ne0$ by Lemma~\ref{c:lem:lift-ext}, and $(H_i)_{\mathfrak m_i}\cong H_i$.
\end{proof}

\subsection{The global obstruction}\label{c:sec:telescope}\label{c:sec:stalk}

Put $a_n=\prod_{i=1}^n t_i^3$. 

\begin{lemma}\label{c:lem:colimit-injective}
    One has $\varinjlim\bigl(\Omega\xlongrightarrow{a_1}\Omega\xlongrightarrow{a_2}\Omega\longrightarrow\cdots\bigr)
\cong \Omega\otimes_BK\cong E_S(S/\mathfrak n)$.
\end{lemma}

\begin{proof}
Consider the set of compatible morphisms
\[
\Omega\longrightarrow
\Omega\otimes_B B[(t_i^3)^{-1}\mid i\in\mathbb N_+],
\qquad
x \longmapsto
\bigg({\prod_{j=1}^{n-1}a_j}\bigg)^{-1} \cdot x.
\]
The universal property of colimits yields a morphism
\[
\varinjlim\bigl(\Omega\xlongrightarrow{a_1}\Omega\xlongrightarrow{a_2}\Omega\longrightarrow\cdots\bigr) \longrightarrow \Omega\otimes_B B[(t_i^3)^{-1}\mid i\in\mathbb N_+].
\]
It it injective since each $a_i$ is injective and $B$ is an integral domain.
It is surjective since each element on the right hand side concerns only finite many $t_i$'s, thus admits a preimage.

Finally, $\Omega\otimes_BK\cong E_S(S/\mathfrak n)$ follows from Lemma~\ref{c:lem:genericinjective}.
\end{proof}

\begin{lemma}\label{c:lem:global-obstruction}
Consider the tower
\[
\cdots\xlongrightarrow{a_3}\coprod_{i\geq1}\mathrm{Ext}_S^1(\Omega,H_i)
\xlongrightarrow{a_2}\coprod_{i\geq1}\mathrm{Ext}_S^1(\Omega,H_i)
\xlongrightarrow{a_1}\coprod_{i\geq1}\mathrm{Ext}_S^1(\Omega,H_i).
\]
Then one has $\varprojlim{}^1\coprod_{i\geq1}\mathrm{Ext}_S^1(\Omega,H_i)\ne0$.
\end{lemma}
\begin{proof}
By Lemma~\ref{c:lem:lift-ext}, $\mathrm{Ext}_S^1(\Omega,H_i)
\cong\mathrm{Ext}_{S_{\mathfrak m_i}}^1(\Omega_{\mathfrak m_i},H_i)\ne0$.
Since $X_i\otimes_{K_i}Y_i\cong S_{\mathfrak m_i}/t_i^3S_{\mathfrak m_i}$ is annihilated by $t^3$, one has $t_i^3\mathrm{Ext}_S^1(\Omega,H_i)
\cong t_i^3\mathrm{Ext}_{S_{\mathfrak m_i}}^1(\Omega_{\mathfrak m_i},H_i)= 0$.

For $j\ne i$, multiplication by $t_j^3$ on this $S_{\mathfrak m_i}$-module is invertible. Consequently, $a_1\cdots a_{i-1}$ acts invertibly on $\mathrm{Ext}_S^1(\Omega,H_i)$, whereas $(a_1\cdots a_i)\mathrm{Ext}_S^1(\Omega,H_i)=0$.
For each $i \in \mathbb N_+$, fix $0\ne\xi_i\in\mathrm{Ext}_S^1(\Omega,H_i) \subseteq \coprod_i\mathrm{Ext}_S^1(\Omega,H_i)$.
One gets a diagonal element
\[
(\xi_i)_{i\geq1}
\in\prod_{n\geq1}\coprod_{i\geq1}
\mathrm{Ext}_S^1(\Omega,H_i).
\]

We claim that $(\xi_i)_{i\geq1} \notin \mathrm{im}(1-\Delta)$, where $\Delta((y_i)_{i\geq 1})=(a_iy_{i+1})_{i \geq 1}$.
Otherwise, assume there were $(y_i)_{i\geq 1}$ with $y_i-a_iy_{i+1}=\xi_i$.
Iterating $y_i-a_i y_{i+1}=\xi_i$ gives
\[
y_1=\sum_{k=1}^m\left(\prod_{j=1}^{k-1}a_j\right)\xi_k
+\left(\prod_{j=1}^m a_j\right)y_{m+1}.
\]
In the $\mathrm{Ext}_S^1(\Omega,H_m)$-component, all $\xi_k$ with $k\ne m$ vanish, and the last term vanishes since $t_m^3 \in \prod_{j=1}^m a_j$ is the annihilator.
Thus this component of $y_1$ equals $\left(\prod_{j=1}^{m-1}a_j\right)\xi_m\ne0$,
since every factor acts as an automorphism on $\mathrm{Ext}_S^1(\Omega,H_m)$.
This holds for every $m$, contradicting the finite support of $y_1$.

Therefore, $(\xi_i)_{i\geq1}\notin\mathrm{im}(1-\Delta)$.
Thus $\mathrm{coker}(1-\Delta)
\cong\varprojlim{}^1\coprod_{i\geq1}\mathrm{Ext}_S^1(\Omega,H_i)\ne0$.
\end{proof}

\end{document}